\documentclass[english]{article}
\usepackage{mathptmx}
\usepackage[T1]{fontenc}
\usepackage[latin9]{inputenc}
\usepackage{geometry}
\usepackage{babel}
\usepackage{float}
\usepackage{amsmath}
\usepackage{amsthm}
\usepackage{amssymb}
\usepackage{graphicx}
\usepackage[unicode=true,pdfusetitle,
 bookmarks=true,bookmarksnumbered=false,bookmarksopen=false,
 breaklinks=false,pdfborder={0 0 0},pdfborderstyle={},backref=false,colorlinks=false]
 {hyperref}

\makeatletter
\numberwithin{equation}{section}
\numberwithin{figure}{section}
\theoremstyle{plain}
\newtheorem{thm}{\protect\theoremname}[section]
\theoremstyle{definition}
\newtheorem{defn}[thm]{\protect\definitionname}
\theoremstyle{remark}
\newtheorem{rem}[thm]{\protect\remarkname}
\theoremstyle{plain}
\newtheorem{prop}[thm]{\protect\propositionname}
\theoremstyle{plain}
\newtheorem{lem}[thm]{\protect\lemmaname}

\@ifundefined{date}{}{\date{}}
\makeatother

\providecommand{\definitionname}{Definition}
\providecommand{\lemmaname}{Lemma}
\providecommand{\propositionname}{Proposition}
\providecommand{\remarkname}{Remark}
\providecommand{\theoremname}{Theorem}

\begin{document}
\title{A central limit theorem and its application to the limiting distribution
of volatility target index}
\author{Xuan Liu\thanks{Nomura Securities, Hong Kong SAR. Email: xuan.liu@nomura.com}\enskip{}and
Michel Gauthier\thanks{Nomura Securities, Tokyo, Japan. Email: michel.gauthier@nomura.com}}
\maketitle
\begin{abstract}
We study the limiting distribution of a volatility target index as
the discretisation time step converges to zero. Two limit theorems
(a strong law of large numbers and a central limit theorem) are established,
and as an application, the exact limiting distribution is derived.
We demonstrate that the volatility of the limiting distribution is
consistently larger than the target volatility, and converges to the
target volatility as the observation-window parameter $\lambda$ in
the definition of the realised variance converges to 1. Besides the
exact formula for the drift and the volatility of the limiting distribution,
their upper and lower bounds are derived. As a corollary of the exact
limiting distribution, we obtain a vega conversion formula which converts
the rho1 sensitivity of a financial derivative on the limiting diffusion
to the vega sensitivity of the same financial derivative on the underlying
of the volatility target index.
\end{abstract}

\section{Introduction\label{sec:-3}}

In financial markets, a type of financial index called volatility
target indices has been popular among insurance companies, mutual
funds, and investment banks for more than a decade, and is widely
used in industry for quantitative investment strategies (QIS) and
fixed indexed annuities (FIA). Volatility target indices replicate
the performance of volatility targeting strategies by dynamically
adjusting the weights of a risky asset and a risk free asset, so that
the portfolio volatility remains at a fixed level. From an investment
perspective, they provide a control of balance between return and
risk. There exist several major index providers, e.g. the S\&P and
the MSCI, issuing volatility target indices, and institutional investors
can invest in these indices just like they invest in exchange traded
funds. From derivative trading perspective, these indices reduce uncertainty
of volatility, which makes vanilla options written on volatility target
indices popular among investment banks. In recent years, the financial
industry has seen increasing popularity of intraday volatility target
indices which rebalance much more frequently than once a day. There
has been extensive empirical research showing that volatility target
indices are able to generate higher Sharpe ratio than a buy-and-hold
strategy (cf. \cite{HHKRSV,BS24}), especially under extreme market
conditions (cf. \cite{BKV20} for empirical results for performance
of volatility target indices under volatility cluster). Investors
investing in volatility target indices seek risk protection by means
of financial derivatives written on these indices. Various numerical
methods such as Monte Carlo simulation can be used to capture the
exact contractual terms of these indices. However, those methods tend
to become rapidly inconvenient when the underlying (of financial derivatives
or volatility target indices themselves) increases in complexity.
Another common and more efficient approach in practice is to approximate
a volatility target index as an ordinary risky asset (i.e. an Itô
diffusion) with volatility close to the target volatility, and apply
the Black--Scholes formula to this Itô diffusion. This approach however
suffers from uncertainty of the drift and the volatility of the approximating
diffusion, and fails to provide a direct quantification of vega sensitivity
on the underlying risky asset, which is often assumed to be negligible.
To the best of our knowledge, there is no existing result in the literature
on the exact distribution of the limiting process for vanishing discretisation
time step. The purpose of this paper, motivated by the tendency of
higher rebalancing frequency, is to derive the exact limiting distribution
of volatility target indices, and to study their asymptotic behaviours.
A crucial ingredient of our derivation is a central limit theorem
for dependent sequence, which can also be applied to the study of
other types of volatility based indices.

In this paper, we will assume that the dynamics of the risky asset
is given by
\begin{equation}
dS_{t}=\rho(t)S_{t}dt+\sigma(t)S_{t}dW_{t},\label{eq:}
\end{equation}
under the martingale measure, where 
\begin{itemize}
\item $\rho(t)\in L^{\infty}(0,T)$ is a deterministic function of $t$
representing the sum of the discount rate, the repo rate, and the
dividend yield; 
\item $\sigma(t)\in L^{\infty}(0,T)$ is a deterministic function of $t$
representing the volatility, and
\begin{equation}
0<\sigma_{\ast}<\sigma(t)<\sigma^{\ast}<\infty,\quad0\le t\le T,\label{eq:-80}
\end{equation}
for some constants $\sigma_{\ast}$ and $\sigma^{\ast}$.
\end{itemize}
The process $\{S_{t}\}_{t\ge0}$ given by (\ref{eq:}) has log-normal
marginal distributions with time dependent parameters, which corresponds
to the so-called \emph{term structure} in volatility modeling. It
is the degenerate case of more realistic models including volatility
smiles.
\begin{defn}
\label{def:-3}Let $r(t)\in L^{\infty}(0,T)$ be a deterministic function
representing a customized\emph{ risk free rate}. Let $T>0$ be a fixed
time, $N\in\mathbb{N}_{+}$ be the \emph{number of time steps}, and
$\Delta t=T/N$ be the \emph{rebalancing time step}. Denote $t_{n}=n\Delta t$,
$n=0,1,\dots,N$. The \emph{(discrete time) volatility target index}
$\tilde{I}_{t}$ is the adapted process defined by 
\begin{equation}
\begin{aligned}\frac{\tilde{I}_{t_{n}}}{\tilde{I}_{t_{n-1}}} & =1+(1-w_{n-1})\int_{t_{n-1}}^{t_{n}}r(t)dt+w_{n-1}\Big(\frac{S_{t_{n}}}{S_{t_{n-1}}}-1\Big),\quad n\ge1,\\
\tilde{I}_{t_{0}} & =I_{0},
\end{aligned}
\label{eq:-86}
\end{equation}
and $\tilde{I}_{t}=\tilde{I}_{t_{n}}$ for $t\in[t_{n},t_{n+1})$,
where $I_{0}>0$ is a constant for \emph{initial level}, and $\bar{\sigma}$
is a constant called the \emph{target volatility}, and
\begin{equation}
w_{n}=\frac{\bar{\sigma}}{\sqrt{v_{n}}}\label{eq:-98}
\end{equation}
is the \emph{risky asset leverage}, and $v_{n}$ is the \emph{realised
annual variance}
\begin{equation}
\begin{aligned}v_{n} & =\lambda v_{n-1}+\frac{1-\lambda}{\Delta t}\Big(\frac{S_{t_{n}}}{S_{t_{n-1}}}-1\Big)^{2},\quad n\ge1,\end{aligned}
\label{eq:-2}
\end{equation}
or equivalently,
\begin{equation}
v_{n}=\lambda^{n}v_{0}+\frac{1-\lambda}{\Delta t}\sum_{k=1}^{n}\lambda^{n-k}\Big[\exp\Big(\int_{t_{k-1}}^{t_{k}}\Big(\rho(t)-\frac{\sigma(t)^{2}}{2}\Big)dt+\int_{t_{k-1}}^{t_{k}}\sigma(t)dW_{t}\Big)-1\Big]^{2}.\label{eq:-36}
\end{equation}
with $v_{0}>0$ and $\lambda\in(0,1)$ being some constants. 
\end{defn}

\begin{rem}
\label{rem:-5}(i) Intuitively, the leverage $w_{n}$ scales the risky
asset return so that the volatility of the scaled return ``becomes''
the target volatility. And the remaining portion $1-w_{n}$ of capital
is allocated to the risk-free asset.

(ii) The parameters $\lambda$ and $\Delta t$ in Definition \ref{def:-3}
are contractual parameters specified in the index rule book by the
index publisher. The parameter $\lambda$ is typically selected from
a set of commonly used values, while smaller values of $\Delta t$
are becoming increasingly popular for intraday volatility target indices.
\end{rem}

Note that the definition (\ref{eq:-86}) is model-free. However, it
is not convenient for the study of its limiting distribution. This
is because the right hand side of (\ref{eq:-86}) can be negative.
The definition (\ref{eq:-1}) below is the continuous time version
of (\ref{eq:-86}). As we will show in Proposition \ref{prop:} in
\nameref{sec:-4}, when the risky asset dynamics is given by (\ref{eq:}),
the processes defined by (\ref{eq:-86}) and (\ref{eq:-1}) converge
to the same limiting distribution. Moreover, the process defined by
(\ref{eq:-1}) is always positive and more convenient for the study
of the limiting distribution. Therefore, in this paper, we will adopt
Definition \ref{def:-2} below, where the existence and uniqueness
of solution of the stochastic differential equation is an immediate
consequence of Itô's lemma.
\begin{defn}
\label{def:-2}The \emph{(continuous time) volatility target index
}$\{I_{t}\}_{t\ge0}$ is defined to be the process
\begin{equation}
\begin{aligned}\frac{dI_{t}}{I_{t}} & =(1-w_{n-1})r(t)dt+w_{n-1}\frac{dS_{t}}{S_{t}},\quad t\in[t_{n-1},t_{n}),\\
I_{t_{0}} & =I_{0}.
\end{aligned}
\label{eq:-1}
\end{equation}
\end{defn}

We would like to point out that the leverage process $w_{n}$ in (\ref{eq:-98})
is physically observable and model-free. Different versions of the
leverage $w_{n}$ have been adopted in the literature, some of which
are not observable. For example, in \cite{Tor18}, the leverage process
is defined to be $w^{h}(t)=\bar{\sigma}/\big(\int_{0}^{t}\sigma_{t-s}^{2}h(s)ds\big)^{1/2}$
with $h\in L^{2}([0,\infty))$, $\Vert h\Vert_{L^{2}}=1$, and $\sigma_{t}$
is a stochastic volatility process. With this definition, option pricing
for volatility target index is studied in \cite{Tor18} under stochastic
volatility models. 
\begin{rem}
(i) The realised variance $v_{n-1}$ in (\ref{eq:-1}) can be replaced
with $v_{n-l}$ for any fixed $l\ge1$. The results in this paper
remain the same, and their proof are similar to the case of $v_{n-1}$.

(ii) The definition of $v_{n}$ in (\ref{eq:-2}) is an exponential
moving average of annualised returns. There is an alternative definition
using simple moving average. The arguments used in this paper can
be easily adapted to this variant, and similar results hold. See Remark
\ref{rem:-1} (ii) below.

(iii) A popular variant of volatility target index is to replace the
risky asset leverage $w_{n}=v_{n}^{-1/2}\bar{\sigma}$ by 
\begin{equation}
w_{n}=\min\big[w^{\ast},\big(v_{n}^{(1)}\big)^{-1/2}\bar{\sigma},\big(v_{n}^{(2)}\big)^{-1/2}\bar{\sigma}\big],\label{eq:-85}
\end{equation}
where $w^{\ast}>0$ is a constant preventing large leverage of the
risky asset, and the realised variances $v_{n}^{(1)}$ and $v_{n}^{(2)}$
are defined by replacing $\lambda$ by different $\lambda_{1}$ and
$\lambda_{2}$ in (\ref{eq:-2}). For this variant, the approach used
in this paper is still applicable, and the limiting distribution is
again an Itô diffusion of the form (\ref{eq:-39}). We will give a
brief discussion on this direction in Remark \ref{rem:} in Section
\ref{subsec:-1}. We would like to point out that, in practice, the
values of $\lambda_{1}$ and $\lambda_{2}$ are close to each other,
and $w^{\ast}$ is away from $\bar{\sigma}/\sigma(t)$. Therefore,
(\ref{eq:-1}) provides a good approximation to the variant (\ref{eq:-85}).

(iv) From a finance point of view, it is desirable to also consider
dividend protection or a fee charge on the volatility target index,
that is,
\begin{equation}
\frac{\tilde{I}_{t_{n}}}{\tilde{I}_{t_{n-1}}}=1+(1-w_{n-1})\int_{t_{n-1}}^{t_{n}}r(t)dt+w_{n-1}\int_{t_{n-1}}^{t_{n}}a(t)dt+w_{n-1}\Big(\frac{S_{t_{n}}}{S_{t_{n-1}}}-1\Big),\quad n\ge1,\label{eq:-101}
\end{equation}
where $a(t)$ represents an adjustment rate, which could be the dividend
yield or a rate of fee charge. For (\ref{eq:-101}), results similar
to those in Theorem \ref{thm:-3} still hold (see Remark \ref{rem:-2}
(iv) below).
\end{rem}

Intuitively, the process $v_{n}$ provides an estimate of $\sigma(t)^{2}$
when $N$ is large. Such interpretation often leads to the intuition
that $\{I_{t}\}_{t\ge0}$ is approximately equal to a diffusion given
by $dI_{t}=\mu(t)I_{t}dt+\bar{\sigma}I_{t}dW_{t}$ when $\Delta t$
is small. However, this is only true when $\lambda$ is close to $1$
and $\Delta t$ becomes accordingly small, as we will see in Theorem
\ref{thm:-3}, and Lemma \ref{lem:-8}, Lemma \ref{lem:-9} in Section
\ref{sec:-6}. 

In the rest of this section, we summarise the main results of this
paper and make some comments.
\begin{thm}
\label{thm:-3}(i) For any fixed $\lambda\in(0,1)$, as $\Delta t\to0$,
the process $\{I_{t}\}_{t\ge0}$ converges in law to the Itô diffusion
\begin{equation}
dX_{t}=\big(r(t)+(\rho(t)-r(t))\sigma(t)^{-1}\bar{\sigma}U(\lambda)\big)X_{t}dt+\bar{\sigma}V(\lambda)^{1/2}X_{t}dW_{t},\label{eq:-39}
\end{equation}
where
\begin{equation}
U(\lambda)=\sqrt{\frac{2}{\pi(1-\lambda)}}\int_{0}^{\infty}\prod_{k=0}^{\infty}(1+t^{2}\lambda^{k})^{-1/2}dt,\label{eq:-16}
\end{equation}
and
\begin{equation}
V(\lambda)=\frac{1}{2(1-\lambda)}\int_{0}^{\infty}\prod_{k=0}^{\infty}(1+t\lambda^{k})^{-1/2}dt.\label{eq:-3}
\end{equation}

(ii) The multipliers $U(\lambda)$ and $V(\lambda)$ satisfy that
$U(\lambda)>1$, $V(\lambda)>1$ and $\lim_{\lambda\to1-}U(\lambda)=\lim_{\lambda\to1-}V(\lambda)=1$.
Moreover, for $\lambda\in(0.7,1)$, the following bounds for $U(\lambda)$
\begin{equation}
\sqrt{\frac{\lambda^{-1.2}\log\lambda^{-1}}{\lambda^{-1}-1}}\le U(\lambda)\le\frac{1}{1-2e^{-2\pi^{2}/\log(\lambda^{-1})}}\sqrt{\frac{\lambda^{-1.25}\log\lambda^{-1}}{\lambda^{-1}-1}},\label{eq:-70}
\end{equation}
and bounds for $V(\lambda)$
\begin{equation}
\frac{\lambda^{-1.45}\log(\lambda^{-1})}{\lambda^{-1}-1}\le V(\lambda)\le\frac{\lambda^{-1.5}\log(\lambda^{-1})}{\lambda^{-1}-1},\label{eq:-38}
\end{equation}
hold.
\end{thm}

\begin{rem}
\label{rem:-2}(i) The infinite product $\prod_{k=0}^{\infty}(1+t\lambda^{k})$
appearing in (\ref{eq:-3}) is recognized as the infinite $q$-Pochhammer
symbol $(-t;\lambda)_{\infty}$, which is a major building block of
the so called $q$-calculus; see \cite[Section 10.2]{AAR99} for more
details.

(ii) Figure \ref{fig:-5} and Figure \ref{fig:} in Section \ref{sec:-2}
below show the effectiveness of the bounds in (\ref{eq:-70}) and
(\ref{eq:-38}). As shown in these figures, these bounds, and their
arithmetic or geometric averages, are good approximations to $U(\lambda)$
and $V(\lambda)$ respectively\footnote{In practice, $\lambda$ is usually between 0.87 and 0.97, which correspond
to half-life being a week and a month respectively for daily rebalanced
volatility target indices.}. 

(iii) Suppose that $r(t)$, $\rho(t)$, and $\sigma(t)$ are constants,
and that $dS_{t}=\mu S_{t}dt+\sigma S_{t}dW_{t}$, $\mu\ge r$ is
the risky asset dynamics under the physical measure\footnote{The condition $\mu\ge r$ is due to the risk premium of the risky
asset.}. The annualised Sharpe ratio of $\{S_{t}\}_{t\ge0}$ is $(\mu-\sigma^{2}/2-r)/\sigma=(\mu-r)/\sigma-\sigma/2$.
By Theorem \ref{thm:-3}, the annualised Sharpe ratio of the limiting
process $\{X_{t}\}_{t\ge0}$ is 
\[
\frac{r+(\mu-r)\sigma^{-1}\bar{\sigma}U(\lambda)-\bar{\sigma}^{2}V(\lambda)/2-r}{\bar{\sigma}V(\lambda)^{1/2}}=\frac{(\mu-r)U(\lambda)}{\sigma V(\lambda)^{1/2}}-\bar{\sigma}V(\lambda)^{1/2}/2.
\]
When 
\[
\frac{\bar{\sigma}}{\sigma}<\Big[1-2\Big(1-\frac{U(\lambda)}{V(\lambda)^{1/2}}\Big)\frac{\mu-r}{\sigma^{2}}\Big]V(\lambda)^{-1/2},
\]
the Sharpe ratio of $\{X_{t}\}_{t\ge0}$ is higher than that of $\{S_{t}\}_{t\ge0}$.
By (\ref{eq:-70}) and (\ref{eq:-38}), the above inequality is valid
when $\lambda\ge0.9$ and $\bar{\sigma}/\sigma<\frac{9}{10}\cdot[1-\frac{1}{25}(\mu-r)/\sigma^{2}]$.
This is consistent with the empirically observed Sharpe ratio improvement
for small target volatility.

(iv) When the volatility target index has dividend protection or a
fee charge, i.e. the variant given by (\ref{eq:-101}), the limiting
process is given by
\begin{equation}
dX_{t}=\big(r(t)+(\rho(t)-r(t)+a(t))\sigma(t)^{-1}\bar{\sigma}U(\lambda)\big)X_{t}dt+\bar{\sigma}V(\lambda)^{1/2}X_{t}dW_{t}.\label{eq:-102}
\end{equation}
\end{rem}

When approximating the volatility target index with its limiting distribution,
it is important to know risk sensitivities with respect to the underlying
risky asset $\{S_{t}\}_{t\ge0}$. Since the spot of $I_{t}$ is a
simple function of the spot of $S_{t}$, spot sensitivity with respect
to $\{S_{t}\}_{t\ge0}$ can be obtained by the chain rule. However,
the volatility sensitivity with respect to $\{S_{t}\}_{t\ge0}$ is
less obvious. The following proposition gives the desired vega conversion
formula, which will be useful for vega hedging. 
\begin{prop}
\label{prop:-1}(i) Let $G(I_{T}^{(N)})$ be a payoff function on
the volatility target index, and let
\[
\psi_{S,\Delta t}(r,\rho,\sigma)=\mathbb{E}\Big[e^{-\int_{0}^{T}r^{\text{disc}}(t)dt}G(I_{T}^{(N)})\Big]
\]
be the price function of $G(I_{T}^{(N)})$, where $r^{\text{disc}}(t)$
is the discount rate, and we use the subscript $S$ in $\psi_{S,\Delta t}$
to indicate that the underlying diffusion is $\{S_{t}\}_{t\ge0}$.
For any $\mu\in L^{\infty}(0,T)$ and $\nu\in L^{\infty}(0,\infty)$,
let $\{Y_{t}^{\mu,\nu}\}_{t\ge0}$ be the diffusion
\[
dY_{t}^{\mu,\nu}=\mu(t)Y_{t}^{\mu,\nu}dt+\nu(t)Y_{t}^{\mu,\nu}dW_{t},
\]
and denote
\begin{equation}
\psi_{Y}(\mu,\nu)=\mathbb{E}\Big[e^{-\int_{0}^{T}r^{\text{disc}}(t)dt}G(Y_{T}^{\mu,\nu})\Big].\label{eq:-41}
\end{equation}
Then, for any fixed $\lambda\in(0,1)$,
\begin{equation}
\lim_{\Delta t\to0}\psi_{S,\Delta t}(r,\rho,\sigma)=\psi_{Y}\big(r+(\rho-r)\sigma^{-1}\bar{\sigma}U(\lambda),\bar{\sigma}V(\lambda)^{1/2}\big).\label{eq:-42}
\end{equation}

(ii) If, in addition, the coefficients $r(t)=r$, $\rho(t)=\rho$,
and $\sigma(t)=\sigma$ are constants, the following rho-vega conversion
holds
\begin{equation}
\lim_{\Delta t\to0}(\partial_{\sigma}\psi_{S,\Delta t})(r,\rho,\sigma)=(r-\rho)\sigma^{-2}\bar{\sigma}U(\lambda)(\partial_{\mu}\psi_{Y})\big(r+(\rho-r)\sigma^{-1}\bar{\sigma}U(\lambda),\bar{\sigma}V(\lambda)^{1/2}\big).\label{eq:-40}
\end{equation}
\end{prop}

\begin{rem}
(i) The formula (\ref{eq:-40}) is essentially a result on the interchangeability
of limiting and differentiation.

(ii) The rho-vega conversion formula (\ref{eq:-40}) shows that the
vega of the original payoff $G(I_{T}^{(N)})$ (i.e. with respect to
the risky asset volatility) is not related to the vega of the same
payoff on the approximating Itô diffusion $\{X_{t}\}_{t\ge0}$. Instead,
it is a constant multiple of the rho sensitivity of the payoff $G$
on the approximating Itô diffusion $\{X_{t}\}_{t\ge0}$. On the other
hand, the formula (\ref{eq:-40}) implies that the vega with respect
to $\{S_{t}\}_{t\ge0}$ is small when $\rho-r$ is small.
\end{rem}

The rest of this paper is organised as follows. Section \ref{sec:-1}
is devoted to the proof of the main results, i.e. Theorem \ref{thm:-3}
and Proposition \ref{prop:-1}. Section \ref{sec:} is devoted to
the study of properties of the functions $U(\lambda)$ and $V(\lambda)$,
and the derivation of the upper and lower bounds in (\ref{eq:-70})
and (\ref{eq:-38}). Limits of the volatility of $I_{T}$ when $(\Delta t,\lambda)$
approaches $(0,1)$ in different ways are discussed in Section \ref{sec:-6},
which provide quantification of the fact that, to achieve a preset
volatility targeting quality, a smaller value of $\Delta t$ is needed
when $\lambda$ is closer to $1$. In Section \ref{sec:-2}, we present
some numerical test results confirming the conclusions of Theorem
\ref{thm:-3} and the results in Section \ref{sec:-6}. 

\section{\label{sec:-1}Proof of the main results}

This section consists of four subsections. In Section \ref{subsec:},
we establish two limit theorems, Theorem \ref{thm:-4}, a strong law
of large numbers, and Theorem \ref{thm:-5}, a central limit theorem.
In Section \ref{subsec:-3}, we show that the convergence of $I_{T}$
can be reduced to a simpler process having a form for which the Theorem
\ref{thm:-4} and Theorem \ref{thm:-5} can be easily applied (see
Proposition \ref{prop:-4} below). In Section \ref{subsec:-1}, as
an application of Theorem \ref{thm:-4}, Theorem \ref{thm:-5}, and
Proposition \ref{prop:-4}, we establish Theorem \ref{thm:-3} under
the additional assumption that $r(t)$, $\rho(t)$, and $\sigma(t)$
are constants. In the last subsection, we prove the general case of
Theorem \ref{thm:-3} using the constant coefficient case and a continuity
argument. 

Throughout this section, we will adopt the convention that the notation
$c_{\alpha_{1},\dots,\alpha_{k}}$ denotes a constant depending only
on $\alpha_{1},\dots,\alpha_{k}$, and its value may vary at different
appearances.

\subsection{\label{subsec:}Two limit theorems}

The proof of the constant coefficient case of Theorem \ref{thm:-3}
relies on Theorem \ref{thm:-4} and Theorem \ref{thm:-5} below, which
are a law of large numbers and a central limit theorem for sequence
of dependent random variables. 
\begin{thm}
\label{thm:-4}Let $\{\lambda_{n}\}_{n\ge0}$ be a sequence of non-negative
numbers such that $\sum_{n=0}^{\infty}\lambda_{n}=1$ and $\sup_{n\ge0}\gamma^{-n}\lambda_{n}<\infty$
for some constant $\gamma\in(0,1)$. Let $\{\xi_{n}\}_{n\ge1}$ be
a sequence of i.i.d. non-negative random variables, $\omega_{0}>0$
be a constant, and
\begin{equation}
\omega_{n}=\Big(1-\sum_{k=1}^{n}\lambda_{n-k}\Big)\omega_{0}+\sum_{k=1}^{n}\lambda_{n-k}\xi_{k},\quad n\ge1.\label{eq:-57}
\end{equation}
Suppose that the Laplace transform of $\xi_{1}$
\begin{equation}
\phi(t)=\mathbb{E}(e^{-t\xi_{1}}),\quad t\ge0,\label{eq:-56}
\end{equation}
satisfies the property
\begin{equation}
1-t^{\alpha}\le\phi(t)\le c_{\phi}(1+t)^{-\beta},\label{eq:-44}
\end{equation}
for some constants $0<\alpha\le1,\beta\ge0$, and $c_{\phi}>0$. Let
$F(s)$ be a function on $(0,\infty)$ satisfying the following properties:
\begin{enumerate}
\item[(P.1)]  $0\le F(s)<\infty$ on $(0,\infty)$.
\item[(P.2)]  F(s) has non-negative inverse Laplace transform $f\ge0$ with at
most polynomial growth at infinity, that is
\begin{equation}
F(s)=\mathcal{L}(f)(s)=\int_{0}^{\infty}e^{-st}f(t)dt,\quad s>0,\label{eq:-50}
\end{equation}
and
\begin{equation}
f(t)\le c_{f}(1+t)^{m},\quad t\ge1,\label{eq:-43}
\end{equation}
for some constants $c_{f}>0$ and $m\ge0$. 
\end{enumerate}
Then
\begin{equation}
\lim_{N\to\infty}\frac{1}{N}\sum_{k=0}^{N-1}F(\omega_{k})=\int_{0}^{\infty}f(t)\prod_{k=0}^{\infty}\phi(\lambda_{k}t)dt<\infty\quad\text{a.s.}\label{eq:-55}
\end{equation}
\end{thm}

\begin{rem}
\label{rem:-1}(i) Below are some examples of the function $F(s)$
satisfying the property (P.2).
\begin{itemize}
\item $F(s)=s^{-p}$ for any $p>0$.
\item $F(s)=e^{-bs}(s+a)^{-n}$ for any $a,b>0$ and integer $n\ge1$.
\item $F(s)=s^{-1}(1-e^{-bs})$ for any $b>0$.
\item If $F(s)$ and $G(s)$ satisfy the property (P.2), then $F(s)+G(s)$
and $F(s)G(s)$ also satisfy the property (P.2).
\end{itemize}
(ii) An example of the sequence $\{\lambda_{n}\}_{n\ge0}$ is $\lambda_{n}=(1-\lambda)\lambda^{n}$,
for which $\omega_{n}=\lambda^{n}\omega_{0}+\sum_{k=1}^{n}\lambda^{n-k}\xi_{k}$.
Another example of $\{\lambda_{n}\}_{n\ge0}$ is $\lambda_{n}=L^{-1}$
when $0\le n\le L-1$, and $\lambda_{n}=0$ when $n\ge L$, for which
$\omega_{n}=\frac{1}{L}\sum_{k=0}^{L-1}\xi_{n-k}$, $n\ge L$. For
this case, the corresponding multiplier functions in (\ref{eq:-39})
are given by $U(L)=(L/2)^{1/2}\Gamma[(L-1)/2]/\Gamma(L/2)$ and $V(L)=L/(L-2)$,
with $\Gamma(z)$ being the Gamma function.
\end{rem}

\begin{proof}[Proof of Theorem \ref{thm:-4}]
We will prove the theorem in 3 steps. 

\textbf{Step 1.} We first show that
\begin{equation}
\lim_{n\to\infty}\mathbb{E}(F(\omega_{n}))=\int_{0}^{\infty}f(t)\prod_{k=0}^{\infty}\phi(\lambda_{k}t)dt<\infty,\label{eq:-48}
\end{equation}
and
\begin{equation}
\lim_{n\to\infty}\mathbb{E}\big[F(\omega_{n})^{2}\big]=\int_{0}^{\infty}(f\ast f)(t)\prod_{k=0}^{\infty}\phi(\lambda_{k}t)dt<\infty,\label{eq:-51}
\end{equation}

Let $r_{n}=\sum_{k=n}^{\infty}\lambda_{k}$. By (\ref{eq:-50}) and
Fubini's theorem,
\begin{equation}
\begin{aligned}\mathbb{E}(F(\omega_{n})) & =\mathbb{E}\Big(\int_{0}^{\infty}f(t)e^{-\omega_{n}t}dt\Big)=\int_{0}^{\infty}f(t)\mathbb{E}\big(e^{-\omega_{n}t}\big)dt\\
 & =\int_{0}^{\infty}f(t)e^{-r_{n}t\omega_{0}}\prod_{k=1}^{n}\mathbb{E}(e^{-\lambda_{n-k}t\xi_{k}})dt=\int_{0}^{\infty}f(t)e^{-r_{n}t\omega_{0}}\prod_{k=0}^{n-1}\phi(\lambda_{n-k}t)dt.
\end{aligned}
\label{eq:-46}
\end{equation}
By (\ref{eq:-44}) and (\ref{eq:-43}), for sufficiently large $K\in\mathbb{N}$,
we have $\int_{1}^{\infty}f(t)\prod_{k=0}^{K-1}\phi(\lambda_{k}t)<\infty$.
Moreover, since $F(s)<\infty$, we have 
\[
\int_{0}^{1}f(t)\prod_{k=0}^{K-1}\phi(\lambda_{k}t)dt\le e\int_{0}^{1}e^{-t}f(t)dt\le eF(1)<\infty.
\]
Therefore,
\[
\int_{0}^{\infty}f(t)\prod_{k=0}^{K-1}\phi(\lambda_{k}t)dt<\infty,
\]
which, together with $\phi(\lambda_{k}t)\le1$ and the dominated convergence
theorem, implies (\ref{eq:-48}).

For (\ref{eq:-51}), since $F^{2}$ also satisfies the property (P.2),
by an argument similar to the above, we obtain that
\begin{equation}
\mathbb{E}\big[F(\omega_{n})^{2}\big]=\int_{0}^{\infty}(f\ast f)(t)e^{-r_{n}t\omega_{0}}\prod_{k=0}^{n-1}\phi(\lambda_{n-k}t)dt.\label{eq:-45}
\end{equation}
Clearly, $f\ast f$ has at most polynomial growth at infinity. This,
together with (\ref{eq:-44}) and (\ref{eq:-45}), implies (\ref{eq:-51}).

\textbf{Step 2.} We next show that
\begin{equation}
\text{Cov}(F(\omega_{n}),F(\omega_{n+k}))\le c_{\gamma,f,\phi}\gamma^{\alpha k},\label{eq:-52}
\end{equation}
for all $n,k$.

Since $f\ge0$, we see that $F(s)$ is decreasing in $s$. Therefore,
\[
\begin{aligned}\mathbb{E}\big[F(\omega_{n})F(\omega_{n+k})\big] & =\mathbb{E}\Big[F(\omega_{n})F\Big(r_{n+k}\omega_{0}+\sum_{j=1}^{n+k}\lambda_{n+k-j}\xi_{j}\Big)\Big]\le\mathbb{E}\Big[F(\omega_{n})F\Big(\sum_{j=1}^{k}\lambda_{k-j}\xi_{n+j}\Big)\Big]\\
 & =\mathbb{E}\big[F(\omega_{n})\big]\mathbb{E}\Big[F\Big(\sum_{j=1}^{k}\lambda_{k-j}\xi_{n+j}\Big)\Big]=\mathbb{E}\big[F(\omega_{n})\big]\mathbb{E}\Big[F\Big(\sum_{j=1}^{k}\lambda_{k-j}\xi_{j}\Big)\Big].
\end{aligned}
\]
Applying (\ref{eq:-46}) to the case $\omega_{0}=0$ gives that
\[
\mathbb{E}\big[F(\omega_{n})F(\omega_{n+k})\big]\le\mathbb{E}\big[F(\omega_{n})\big]J_{k},
\]
where
\[
J_{k}=\mathbb{E}\Big[F\Big(\sum_{j=1}^{k}\lambda_{k-j}\xi_{j}\Big)\Big]=\int_{0}^{\infty}f(t)\prod_{j=0}^{k-1}\phi(\lambda_{k}t)dt.
\]
Hence,
\begin{equation}
\begin{aligned}\text{Cov}(F(\omega_{n}),F(\omega_{n+k})) & =\mathbb{E}\big[F(\omega_{n})F(\omega_{n+k})\big]-\mathbb{E}\big[F(\omega_{n})\big]\mathbb{E}\big[F(\omega_{n+k})\big]\\
 & \le\mathbb{E}\big[F(\omega_{n})\big](J_{k}-J_{n+k})+\mathbb{E}\big[F(\omega_{n})\big]\big(J_{n+k}-\mathbb{E}\big[F(\omega_{n+k})\big]\big).
\end{aligned}
\label{eq:-53}
\end{equation}

For the first term in the last inequality of (\ref{eq:-53}), we have
\[
\begin{aligned}J_{k}-J_{n+k} & =\int_{0}^{\infty}f(t)\Big(\prod_{j=0}^{k-1}\phi(\lambda_{j}t)\Big)\Big(1-\prod_{j=0}^{n-1}\phi(\lambda_{k+j}t)\Big)dt\le\int_{0}^{\infty}f(t)\Big(1-\prod_{j=0}^{n-1}\phi(\lambda_{k+j}t)\Big)dt.\end{aligned}
\]
Denote $A=\sup_{n\ge0}\gamma^{-n}\lambda_{n}<\infty$. By (\ref{eq:-44})
and $\lambda_{n}\le A\gamma^{n}$,
\begin{align*}
1-\prod_{j=0}^{n-1}\phi(\lambda_{k+j}t) & =\sum_{l=0}^{n-1}\Big(\prod_{j=l+1}^{n-1}\phi(\lambda_{k+j}t)-\prod_{j=l}^{n-1}\phi(\lambda_{k+j}t)\Big)\\
 & =\sum_{l=0}^{n-1}[1-\phi(\lambda_{k+l}t)]\prod_{j=l+1}^{n-1}\phi(\lambda_{k+j})\le\sum_{l=0}^{n-1}[1-\phi(\lambda_{k+l}t)]\le\frac{A^{\alpha}\gamma^{\alpha k}t^{\alpha}}{1-\gamma^{\alpha}},
\end{align*}
where we use the convention that $\prod_{j=l+1}^{n-1}\phi(\lambda_{k+j}t)=1$
when $l+1>n-1$. Let $K$ be an integer such that $K>(m+2)/\beta$.
By (\ref{eq:-44}) and (\ref{eq:-43}), $\int_{0}^{\infty}t^{\alpha}f(t)\prod_{j=0}^{K-1}\phi(\lambda_{j}t)dt<\infty$.
Therefore, for any $k>K$, 
\begin{align*}
J_{k}-J_{n+k} & \le\frac{A^{\alpha}\gamma^{\alpha k}}{1-\gamma^{\alpha}}\int_{0}^{\infty}t^{\alpha}f(t)\prod_{j=0}^{k-1}\phi(\lambda_{j}t)dt\\
 & \le\frac{A^{\alpha}\gamma^{\alpha k}}{1-\gamma^{\alpha}}\int_{0}^{\infty}t^{\alpha}f(t)\prod_{j=0}^{K-1}\phi(\lambda_{j}t)dt=c_{\gamma,f,\phi}\gamma^{\alpha k}.
\end{align*}
This, together with (\ref{eq:-48}), implies that
\begin{equation}
\mathbb{E}\big[F(\omega_{n})\big](J_{k}-J_{n+k})\le c_{\gamma,f,\phi}\gamma^{\alpha k}.\label{eq:-54}
\end{equation}

For the second term in the last inequality of (\ref{eq:-53}), let
$K$ be same integer as above. By (\ref{eq:-44}) and (\ref{eq:-43}),
$\int_{0}^{\infty}tf(t)\prod_{j=0}^{K-1}\phi(\lambda_{j}t)dt<\infty$.
Therefore, for $k\ge K$,
\begin{align*}
J_{n+k}-\mathbb{E}\big[F(\omega_{n+k})\big] & \le\int_{0}^{\infty}f(t)\big(1-e^{-r_{n+k}t\omega_{0}}\big)\prod_{j=0}^{n+k-1}\phi(\lambda_{j}t)dt\\
 & \le r_{n+k}\omega_{0}\int_{0}^{\infty}tf(t)\prod_{j=0}^{n+k-1}\phi(\lambda_{j}t)dt\le r_{n+k}\omega_{0}\int_{0}^{\infty}tf(t)\prod_{j=0}^{K-1}\phi(\lambda_{j}t)dt=c_{f,\phi}\omega_{0}r_{n+k}.
\end{align*}
By $r_{n+k}\le A\sum_{j=n+k}^{\infty}\gamma^{j}\le c_{\gamma}\gamma^{n+k}$
and (\ref{eq:-48}), we obtain that
\begin{equation}
\mathbb{E}\big[F(\omega_{n})\big]\big(J_{n+k}-\mathbb{E}\big[F(\omega_{n+k})\big]\big)\le c_{\gamma,f,\phi}\gamma^{n+k}.\label{eq:-49}
\end{equation}
Combining (\ref{eq:-54}) and (\ref{eq:-49}) gives (\ref{eq:-52}).

\textbf{Step 3.} We now prove the a.s. convergence (\ref{eq:-55}).

By (\ref{eq:-51}), we have $L=\sup_{n}\text{Var}(F(\omega_{n}))<\infty$.
Denote $\psi_{n}=\frac{1}{n}\sum_{k=0}^{n-1}F(\omega_{k})$. By (\ref{eq:-52}),
\begin{equation}
\begin{aligned}\text{Var}(\psi_{n}) & =n^{-2}\sum_{k=0}^{n-1}\text{Var}(F(\omega_{k}))+2n^{-2}\sum_{1\le k<l\le n}\text{Cov}(F(\omega_{k}),F(\omega_{l}))\\
 & \le n^{-1}L+2c_{\gamma,f,\phi}n^{-2}\sum_{k=1}^{n}\sum_{l=k+1}^{n}\gamma^{\alpha(l-k)}\le n^{-1}\Big(L+2c_{\gamma,f,\phi}\frac{\gamma^{\alpha}}{1-\gamma^{\alpha}}\Big).
\end{aligned}
\label{eq:-75}
\end{equation}
To derive the a.s. convergence of $\lim_{n\to\infty}\psi_{n}$ from
the variance estimate (\ref{eq:-75}), we use an argument similar
to that in \cite{Ete81}. For any $\theta>1$, let $\theta_{n}=\lfloor\theta^{n}\rfloor$
be the integer part of $\theta^{n}$. Then $\theta_{n}\ge\theta^{n}/2$.
Hence, for any $\epsilon>0$, 
\[
\begin{aligned} & \sum_{n=1}^{\infty}\mathbb{P}\big(|\psi_{\theta_{n}}-\mathbb{E}(\psi_{\theta_{n}})|>\epsilon\big)\le\epsilon^{-2}\sum_{n=1}^{\infty}\text{Var}(\psi_{\theta_{n}})\\
 & \hspace{3em}\le\epsilon^{-1}\Big(\frac{L}{2}+c_{\gamma,f,\phi}\frac{\gamma^{\alpha}}{1-\gamma^{\alpha}}\Big)\sum_{n=1}^{\infty}\theta_{n}^{-1}\le\epsilon^{-2}\Big(\frac{L}{2}+c_{\gamma,f,\phi}\frac{\gamma^{\alpha}}{1-\gamma^{\alpha}}\Big)\sum_{n=1}^{\infty}\theta^{-n}<\infty.
\end{aligned}
\]
By the Borel--Cantelli lemma, $\lim_{n\to\infty}\big[\psi_{\theta_{n}}-\mathbb{E}(\psi_{\theta_{n}})\big]=0$
a.s. By (\ref{eq:-48}), we see that
\[
\lim_{n\to\infty}\psi_{\theta_{n}}=J_{\infty}=\int_{0}^{\infty}f(t)\prod_{k=0}^{\infty}\phi(\lambda_{k}t)dt\quad\text{a.s.}
\]
For any $n>0$, let $k(n)$ be such that $\theta_{k(n)}<n\le\theta_{k(n)+1}$.
Then $\theta_{k(n)+1}/\theta_{k(n)}\le\theta$. Clearly $n\psi_{n}$
is increasing in $n$. Therefore,
\[
\frac{\theta_{k(n)}}{\theta_{k(n)+1}}\psi_{\theta_{k(n)}}\le\frac{n}{\theta_{k(n)+1}}\psi_{n}\le\psi_{n}\le\frac{n}{\theta_{k(n)}}\psi_{n}\le\frac{\theta_{k(n)+1}}{\theta_{k(n)}}\psi_{\theta_{k(n)+1}}.
\]
Setting $n\to\infty$ gives that
\[
\theta^{-1}J_{\infty}\le\liminf_{n\to\infty}\psi_{n}\le\limsup_{n\to\infty}\psi_{n}\le\theta J_{\infty}\quad\text{a.s.}
\]
for any $\theta>1$, which implies that
\[
\lim_{n\to\infty}\psi_{n}=J_{\infty}=\int_{0}^{\infty}f(t)\prod_{k=0}^{\infty}\phi(\lambda_{k}t)dt\quad\text{a.s.}
\]
This completes the proof.
\end{proof}
\begin{thm}
\label{thm:-5}Let $\{\lambda_{n}\}_{n\ge0}$ be a sequence of non-negative
numbers satisfying the conditions $\sum_{n=0}^{\infty}\lambda_{n}=1$
and $\sup_{n\ge0}\gamma^{-n}\lambda_{n}<\infty$ for some $\gamma\in(0,1)$.
Let $\{\xi_{n}\}_{n\ge1}$ be a sequence of i.i.d. random variables
of which the Laplace transform $\phi(t)=\mathbb{E}(e^{-t\xi_{1}})$
satisfies the property (\ref{eq:-44}), and $F(s)$ be a completely
monotone function satisfying the properties (P.1) and (P.2) in Theorem
\ref{thm:-4}. Let $\omega_{0}>0$ be a constant and $\{\omega_{n}\}_{n\ge1}$
be the sequence defined by (\ref{eq:-57}). Suppose that $\{\eta_{n}\}_{n\ge1}$
is a sequence of random variables such that $\eta_{n+1}$ is independent
of $\{(\xi_{k},\eta_{k}):1\le k\le n\}$, and 
\[
\mathbb{E}(\eta_{n})=0,
\]
and
\[
\lim_{n\to\infty}\mathbb{E}(\eta_{n}^{2})=1.
\]
Suppose, in addition, that $\sup_{n\ge1}\mathbb{E}(|\eta_{n}|^{p})<\infty$
for some $p>2$. Then
\begin{equation}
\frac{1}{\sqrt{N}}\sum_{k=1}^{N}F(\omega_{k-1})\eta_{k}\to\mathcal{N}\Big(0,\int_{0}^{\infty}(f\ast f)(t)\prod_{k=0}^{\infty}\phi(\lambda_{k}t)dt\Big),\label{eq:-58}
\end{equation}
in distribution as $N\to\infty$, where $\mathcal{N}(\mu,\sigma^{2})$
is the normal distribution with mean $\mu$ and variance $\sigma^{2}$,
$f=\mathcal{L}^{-1}(F)$ is the inverse Laplace transform of $F$,
and the convolution $f\ast f$ is defined by extending $f|_{(-\infty,0)}=0$.
\end{thm}

\begin{proof}
Let $M_{n}=\sum_{k=1}^{n}F(\omega_{k-1})\eta_{k},$ $M_{0}=0$. Clearly,
$\{M_{n}\}_{n\ge1}$ is a martingale in view of $\mathbb{E}(\eta_{n})=0$.
Let $\mathcal{F}_{n}=\sigma(\xi_{k},\eta_{k}:1\le k\le n)$ be the
filtration generated by $\{(\xi_{n},\eta_{n})\}_{n\ge1}$, and define
\begin{align*}
s_{n} & =\sum_{k=1}^{n}\text{Var}(M_{k}-M_{k-1})=\sum_{k=1}^{n}\text{Var}\big(F(\omega_{k-1})\eta_{k}\big),\\
v_{n} & =\sum_{k=1}^{n}\text{Var}(M_{k}-M_{k-1}|\mathcal{F}_{k-1})=\sum_{k=1}^{n}\text{Var}\big(F(\omega_{k-1})\eta_{k}|\mathcal{F}_{k-1}\big).
\end{align*}
By independence of $\eta_{k}$ and $\mathcal{F}_{k-1}$, we have $s_{n}=\sum_{k=1}^{n}\mathbb{E}\big(F(\omega_{k-1})^{2}\big)$
and $v_{n}=\sum_{k=1}^{n}F(\omega_{k-1})^{2}$. We will show that
the Lindeberg condition
\begin{equation}
\lim_{n\to\infty}\frac{1}{n}\sum_{k=1}^{n}\mathbb{E}\big[F(\omega_{k-1})^{2}\eta_{k}^{2}\chi_{\{|F(\omega_{k-1})\eta_{k}|>\epsilon\sqrt{n}\}}\big]=0,\label{eq:-74}
\end{equation}
holds for any $\epsilon>0$, and that
\begin{equation}
\lim_{n\to\infty}\frac{s_{n}}{v_{n}}=1\quad\text{a.s.}\label{eq:-59}
\end{equation}
Once (\ref{eq:-74}) and (\ref{eq:-59}) are proved, by the martingale
central limit theorem (see for example \cite[page 9--10]{HLBH80}),
we may deduce that
\begin{equation}
\frac{1}{\sqrt{s_{N}}}\sum_{k=1}^{N}F(\omega_{k-1})\eta_{k}\to\mathcal{N}(0,1),\label{eq:-60}
\end{equation}
in distribution. Moreover, by (\ref{eq:-51}), we have
\begin{equation}
\lim_{n\to\infty}\frac{s_{n}}{n}=\lim_{n\to\infty}\mathbb{E}\big(F(\omega_{n})^{2}\big)=\int_{0}^{\infty}(f\ast f)(t)\prod_{k=0}^{\infty}\phi(\lambda_{k}t)dt\quad\text{a.s.}\label{eq:-61}
\end{equation}
Combining (\ref{eq:-60}) and (\ref{eq:-61}) proves the desired convergence
(\ref{eq:-58}). 

To prove that the Lindeberg condition (\ref{eq:-74}) holds, for any
integer $m\ge1$, note that by the properties of completely monotone
functions (see for example \cite[Chapter 1]{SSV12}), $F^{p}$ satisfies
the properties (P.1) and (P.2). Applying (\ref{eq:-48}) to $F^{p}$
gives that $\sup_{n\ge1}\mathbb{E}(F(\omega_{k-1})^{p})\le c_{p,\gamma,f,\phi}<\infty$.
Let $A=\sup_{n\ge1}\mathbb{E}(|\eta_{n}|^{p})$. Then
\[
\begin{aligned}\mathbb{E}\big[F(\omega_{k-1})^{2}\eta_{k}^{2}\chi_{\{|F(\omega_{k-1})\eta_{k}|>\epsilon\sqrt{n}\}}\big] & \le\frac{1}{\epsilon^{p-2}n^{p/2-1}}\mathbb{E}\big[F(\omega_{k-1})^{p}|\eta_{k}|^{p}\big]\\
 & =\frac{1}{\epsilon^{p-2}n^{p/2-1}}\mathbb{E}\big[F(\omega_{k-1})^{p}\big]\mathbb{E}\big[|\eta_{k}|^{p}\big]\le\frac{c_{p,\gamma,f,\phi}A}{\epsilon^{p-2}n^{p/2-1}}.
\end{aligned}
\]
Therefore,
\[
\frac{1}{n}\sum_{k=1}^{n}\mathbb{E}\big[F(\omega_{k-1})^{2}\eta_{k}^{2}\chi_{\{|F(\omega_{k-1})\eta_{k}|>\epsilon\sqrt{n}\}}\big]\le\frac{c_{p,\gamma,f,\phi}A}{\epsilon^{p-2}n^{p/2-1}}\to0,
\]
as $n\to\infty$.

It remains to prove (\ref{eq:-59}). Note that $F^{2}$ also satisfies
the properties (P.1) and (P.2). By Theorem \ref{thm:-4}, we obtain
that 
\begin{equation}
\lim_{n\to\infty}\frac{v_{n}}{n}=\int_{0}^{\infty}(f\ast f)(t)\prod_{k=0}^{\infty}\phi(\lambda_{k}t)dt\quad\text{a.s.},\label{eq:-63}
\end{equation}
which, together with (\ref{eq:-61}) implies (\ref{eq:-59}). This
completes the proof of Theorem \ref{thm:-5}.
\end{proof}

\subsection{\label{subsec:-3}Reduction to an equivalent convergence}

For this subsection and the rest of this paper, we will denote processes
such as $I_{t}$ and $v_{n}$ by $I_{t}^{(N)}$ and $v_{n}^{(N)}$
when it is necessary to indicate their dependence on $\Delta t$.
Then
\begin{equation}
\begin{aligned}\log\frac{I_{T}^{(N)}}{I_{0}} & =\int_{0}^{T}r(t)dt+\sum_{k=1}^{N}\dfrac{\bar{\sigma}}{\sqrt{v_{k-1}^{(N)}}}\int_{t_{k-1}}^{t_{k}}[\rho(t)-r(t)]dt\\
 & -\frac{1}{2}\sum_{k=1}^{N}\dfrac{\bar{\sigma}^{2}}{v_{k-1}^{(N)}}\int_{t_{k-1}}^{t_{k}}\sigma(t)^{2}dt+\sum_{k=1}^{N}\dfrac{\bar{\sigma}}{\sqrt{v_{k-1}^{(N)}}}\int_{t_{k-1}}^{t_{k}}\sigma(t)dW_{t},
\end{aligned}
\label{eq:-64}
\end{equation}
We shall show that the convergence of $I_{T}^{(N)}$ in distribution
is equivalent to that of a simpler random variable $X_{T}^{(N)}$
defined by
\begin{equation}
\begin{aligned}\frac{dX_{t}^{(N)}}{X_{t}^{(N)}} & =\Big(1-\frac{\bar{\sigma}}{\sqrt{u_{t_{n-1}}^{(N)}}}\Big)r(t)dt+\frac{\bar{\sigma}}{\sqrt{u_{t_{n-1}}^{(N)}}}\frac{dS_{t}}{S_{t}},\quad t\in[t_{n-1},t_{n}),\\
X_{0}^{(N)} & =I_{0},
\end{aligned}
\label{eq:-125}
\end{equation}
where
\[
\begin{aligned}u_{n}^{(N)} & =\lambda u_{n-1}^{(N)}+\frac{1-\lambda}{\Delta t}\Big(\int_{t_{k-1}}^{t_{k}}\sigma(t)dW_{t}\Big)^{2},\quad u_{0}^{(N)}=v_{0},\end{aligned}
\]
or equivalently,
\begin{equation}
\begin{aligned}\log\frac{X_{T}^{(N)}}{I_{0}} & =\int_{0}^{T}r(t)dt+\sum_{k=1}^{N}\dfrac{\bar{\sigma}}{\sqrt{u_{k-1}^{(N)}}}\int_{t_{k-1}}^{t_{k}}[\rho(t)-r(t)]dt\\
 & -\frac{1}{2}\sum_{k=1}^{N}\dfrac{\bar{\sigma}^{2}}{u_{k-1}^{(N)}}\int_{t_{k-1}}^{t_{k}}\sigma(t)^{2}dt+\sum_{k=1}^{N}\dfrac{\bar{\sigma}}{\sqrt{u_{k-1}^{(N)}}}\int_{t_{k-1}}^{t_{k}}\sigma(t)dW_{t}
\end{aligned}
\label{eq:-127}
\end{equation}
and
\begin{equation}
u_{n}^{(N)}=\lambda^{n}v_{0}+\frac{1-\lambda}{\Delta t}\sum_{k=1}^{n}\lambda^{n-k}\Big(\int_{t_{k-1}}^{t_{k}}\sigma(t)dW_{t}\Big)^{2}.\label{eq:-14}
\end{equation}

We start with the following lemma.
\begin{lem}
\label{lem:-5}For any $h(x)=x^{-\alpha}$, $\alpha>0$, we have
\begin{equation}
\sup_{N\ge1}\sup_{n<N}\mathbb{E}\big[h(u_{n}^{(N)})\big]\le c_{\alpha,\lambda,v_{0},\sigma_{\ast}}<\infty,\label{eq:-79}
\end{equation}
\begin{equation}
\sup_{N\ge1}\sup_{n<N}\mathbb{E}\big[h(v_{n}^{(N)})\big]\le c_{\alpha,\lambda,v_{0},\sigma_{\ast}}<\infty,\label{eq:-33}
\end{equation}
for some constant $c_{\alpha,\lambda,v_{0},\sigma_{\ast}}$ depending
only on $\alpha,\lambda,v_{0},\sigma_{\ast}$, and for any $p\ge1$,
\begin{equation}
\lim_{N\to\infty}\sup_{n<N}\mathbb{E}\big[|h(u_{n}^{(N)})-h(v_{n}^{(N)})|^{p}\big]=0.\label{eq:-31}
\end{equation}
\end{lem}

\begin{proof}
We first prove (\ref{eq:-79}). Note that the inverse Laplace transform
of $h(x)$ is $\mathcal{L}^{-1}(h)(t)=t^{\alpha-1}/\Gamma(\alpha).$
Therefore, by Fubini's theorem,
\begin{equation}
\begin{aligned}\mathbb{E}\big[h(u_{n}^{(N)})\big] & =\mathbb{E}\Big[\frac{1}{\Gamma(\alpha)}\int_{0}^{\infty}t^{\alpha-1}e^{-u_{n}^{(N)}t}dt\Big]\\
 & =\frac{1}{\Gamma(\alpha)}\int_{0}^{\infty}t^{\alpha-1}\mathbb{E}\big(e^{-u_{n}^{(N)}t}\big)dt=\frac{1}{\Gamma(\alpha)}\int_{0}^{\infty}t^{\alpha-1}e^{-\lambda^{n}v_{0}t}\prod_{k=1}^{n-1}\mathbb{E}(\mathcal{E}_{k})dt,
\end{aligned}
\label{eq:-35}
\end{equation}
where $\mathcal{E}_{k}=\exp\Big(-\frac{1-\lambda}{\Delta t}t\lambda^{n-k}\Big(\int_{t_{k-1}}^{t_{k}}\sigma(t)dW_{t}\Big)^{2}\Big)$.
Recall that for a standard normal distribution $Z$ and any $a,\sigma>0$,
$\mu\in\mathbb{R}$, we have
\begin{equation}
\mathbb{E}[e^{-a(\mu+\sigma Z)^{2}}]=(1+2a\sigma^{2})^{-1/2}\exp\Big(-\frac{a\mu^{2}}{1+2a\sigma^{2}}\Big).\label{eq:-133}
\end{equation}
Since $\int_{t_{k-1}}^{t_{k}}\sigma(t)dW_{t}$ is a normal distribution
with mean zero and variance $\int_{t_{k-1}}^{t_{k}}\sigma(t)^{2}dt$,
by (\ref{eq:-133}),
\[
\mathbb{E}(\mathcal{E}_{k})=\Big[1+2\frac{1-\lambda}{\Delta t}t\lambda^{n-k}\Big(\int_{t_{k-1}}^{t_{k}}\sigma(t)^{2}dt\Big)\Big]^{-1/2}\le[1+2t(1-\lambda)\lambda^{n-k}\sigma_{\ast}^{2}]^{-1/2}.
\]
Therefore, by the above and (\ref{eq:-35}),
\begin{equation}
\mathbb{E}\big[h(u_{n}^{(N)})\big]\le\frac{1}{\Gamma(\alpha)}\int_{0}^{\infty}t^{\alpha-1}e^{-\lambda^{n}v_{0}t}\prod_{k=0}^{n-1}[1+2t(1-\lambda)\lambda^{k}\sigma_{\ast}^{2}]^{-1/2}dt.\label{eq:-66}
\end{equation}
Let $m$ be an integer such that $m>2\alpha$. For $n\ge m$, by (\ref{eq:-66}),
\begin{equation}
\begin{aligned}\mathbb{E}\big[h(u_{n}^{(N)})\big] & \le\frac{1}{\Gamma(\alpha)}\int_{0}^{\infty}t^{\alpha-1}\prod_{k=0}^{m-1}[1+2t(1-\lambda)\lambda^{k}\sigma_{\ast}^{2}]^{-1/2}dt\end{aligned}
.\label{eq:-78}
\end{equation}
For $n<m$, by (\ref{eq:-66}) again,
\begin{equation}
\mathbb{E}\big[h(u_{n}^{(N)})\big]\le\frac{1}{\Gamma(\alpha)}\int_{0}^{\infty}t^{\alpha-1}e^{-\lambda^{m}v_{0}t}dt.\label{eq:-114}
\end{equation}
The uniform $L^{1}$ bound (\ref{eq:-79}) follows from (\ref{eq:-78})
and (\ref{eq:-114}).

We next prove (\ref{eq:-33}). Denote $\theta_{k}=\int_{t_{k-1}}^{t_{k}}(\rho(t)-\sigma(t)^{2}/2)dt$,
$Y_{k}=\int_{t_{k-1}}^{t_{k}}\sigma(t)dW_{t}$, and $R_{k}=(\Delta t)^{-1/2}(\theta_{k}+Y_{k})$.
Then 
\[
v_{n}^{(N)}=\lambda^{n}v_{0}+\frac{1-\lambda}{\Delta t}\sum_{k=1}^{n}\lambda^{n-k}\big[e^{\sqrt{\Delta t}R_{k}}-1\big]^{2}.
\]
Note that
\begin{equation}
\frac{1}{\sqrt{\Delta t}}\big|e^{\sqrt{\Delta t}x}-1\big|\ge e^{-\sqrt{\Delta t}}|x|\chi_{\{x\ge-1\}}+\frac{1-e^{-\sqrt{T}}}{\sqrt{T}}\chi_{\{x<-1\}}.\label{eq:-115}
\end{equation}
Indeed, for $x\ge0$, we have $|e^{\sqrt{\Delta t}x}-1|/\sqrt{\Delta t}\ge|x|$.
When $-1\le x<0$, we have $|e^{\sqrt{\Delta t}x}-1|/\sqrt{\Delta t}=e^{-\sqrt{\Delta t}x}\big(e^{\sqrt{\Delta t}x}-1\big)/\sqrt{\Delta t}\ge e^{-\sqrt{\Delta t}}|x|$.
When $x<-1$, we have $|e^{\sqrt{\Delta t}x}-1|/\sqrt{\Delta t}\ge(1-e^{-\sqrt{\Delta t}})/\sqrt{\Delta t}\ge(1-e^{-\sqrt{T}})/\sqrt{T}$.
By (\ref{eq:-115}),
\begin{equation}
\begin{aligned}\mathbb{E}\Big[\exp\Big(-t(1-\lambda)\lambda^{n-k}\frac{(e^{\sqrt{\Delta t}R_{k}}-1)^{2}}{\Delta t}\Big)\Big] & \le\mathbb{E}\Big[e^{-t(1-\lambda)\lambda^{n-k}R_{k}^{2}}+\exp\Big(-t(1-\lambda)\lambda^{n-k}\frac{(1-e^{-\sqrt{\Delta t}})^{2}}{\Delta t}\Big)\Big]\\
 & \le\mathbb{E}\big[e^{-t(1-\lambda)\lambda^{n-k}R_{k}^{2}}\big]+e^{-c_{T}t(1-\lambda)\lambda^{n-k}}.
\end{aligned}
\label{eq:-34}
\end{equation}
Since $R_{k}$ is a normal distribution with mean $(\Delta t)^{-1/2}\theta_{k}$
and variance $\sigma_{k}^{2}=(\Delta t)^{-1}\int_{t_{k-1}}^{t_{k}}\sigma(t)^{2}dt$,
by (\ref{eq:-133}),
\begin{align*}
\mathbb{E}\big[e^{-t(1-\lambda)\lambda^{n-k}R_{k}^{2}}\big] & =[1+2t(1-\lambda)\lambda^{n-k}\sigma_{k}^{2}]^{-1/2}\exp\Big(-\frac{t(1-\lambda)\lambda^{n-k}\theta_{k}^{2}(\Delta t)^{-1}}{1+2t(1-\lambda)\lambda^{n-k}\sigma_{k}^{2}}\Big).
\end{align*}
By $\sigma(t)\ge\sigma_{\ast}$, we obtain that $\mathbb{E}\big[e^{-t(1-\lambda)\lambda^{n-k}R_{k}^{2}}\big]\le[1+2t(1-\lambda)\lambda^{n-k}\sigma_{\ast}^{2}]^{-1/2}$,
and therefore, by (\ref{eq:-34})
\[
\begin{aligned}\mathbb{E}\Big[\exp\Big(-t(1-\lambda)\lambda^{n-k}\frac{(e^{\sqrt{\Delta t}R_{k}}-1)^{2}}{\Delta t}\Big)\Big] & \le[1+2t(1-\lambda)\lambda^{n-k}\sigma_{\ast}^{2}]^{-1/2}+e^{-c_{T}t(1-\lambda)\lambda^{n-k}}\\
 & \le c_{\sigma_{\ast},T}[1+t(1-\lambda)\lambda^{n-k}]^{-1/2}.
\end{aligned}
\]
Moreover, 
\[
\begin{aligned}\mathbb{E}\big[h(v_{n}^{(N)})\big] & =\frac{1}{\Gamma(\alpha)}\int_{0}^{\infty}t^{\alpha-1}e^{-\lambda^{n}v_{0}t}\prod_{k=1}^{n-1}\mathbb{E}\Big[\exp\Big(-t(1-\lambda)\lambda^{n-k}\frac{(e^{\sqrt{\Delta t}R_{k}}-1)^{2}}{\Delta t}\Big)\Big]dt\\
 & \le\frac{c_{\sigma_{\ast},T}}{\Gamma(\alpha)}\int_{0}^{\infty}t^{\alpha-1}e^{-\lambda^{n}v_{0}t}\prod_{k=1}^{n-1}[1+t(1-\lambda)\lambda^{n-k}]^{-1/2}dt,
\end{aligned}
\]
which, similar to the proof of (\ref{eq:-79}), implies the uniform
$L^{1}$-bound (\ref{eq:-33}).

We now prove the uniform convergence (\ref{eq:-31}). It suffices
to prove this for integer-valued $p$. By Hölder's inequality,
\begin{equation}
\begin{aligned}\mathbb{E}\big(|u_{n}^{(N)}-v_{n}^{(N)}|^{2p}\big) & \le\mathbb{E}\Big[\Big(\frac{1-\lambda}{\Delta t}\sum_{k=1}^{n}\lambda^{n-k}\big|(e^{\theta_{k}+Y_{k}}-1)^{2}-Y_{k}^{2}\big|\Big)^{2p}\Big]\\
 & \le\mathbb{E}\Big(\frac{1-\lambda}{(\Delta t)^{2p}}\sum_{k=1}^{n}\lambda^{n-k}\big|(e^{\theta_{k}+Y_{k}}-1)^{2}-Y_{k}^{2}\big|^{2p}\bigg)\\
 & \le\frac{1-\lambda}{(\Delta t)^{2p}}\sum_{k=1}^{n}\lambda^{n-k}\mathbb{E}\big[(e^{\theta_{k}+Y_{k}}-1-Y_{k})^{4p}\big]^{1/2}\mathbb{E}\big[(e^{\theta_{k}+Y_{k}}-1+Y_{k})^{4p}\big]^{1/2}.
\end{aligned}
\label{eq:-123}
\end{equation}
Let $M>(\sigma^{\ast})^{2}$ be such that $\Vert\rho\Vert_{L^{\infty}(0,T)}\le M$.
By 
\[
\begin{aligned}(e^{\theta_{k}+Y_{k}}-1+Y_{k})^{4p} & =[e^{\theta_{k}}(e^{Y_{k}}-1)+(e^{\theta_{k}}-1)+Y_{k}]^{4p}\le c_{p}[e^{4p\theta_{k}}(e^{Y_{k}}-1)^{4p}+(e^{\theta_{k}}-1)^{4p}+Y_{k}^{4p}],\end{aligned}
\]
we have
\begin{equation}
\begin{aligned}\mathbb{E}\big[(e^{\theta_{k}+Y_{k}}-1+Y_{k})^{4p}\big] & \le c_{p}\big[e^{4p\theta_{k}}\mathbb{E}[(e^{Y_{k}}-1)^{4p}]+(e^{\theta_{k}}-1)^{4p}+\mathbb{E}(Y_{k}^{4p})\big]\\
 & \le c_{p,M,T}\big[\mathbb{E}[(e^{Y_{k}}-1)^{4p}]+(e^{c_{M,T}\Delta t}-1)^{4p}+(\Delta t)^{2p}\big]\\
 & \le c_{p,M,T}\big[\mathbb{E}[(|Y_{k}|e^{|Y_{k}|})^{4p}]+(e^{c_{M,T}\Delta t}-1)^{4p}+(\Delta t)^{2p}\big]\\
 & \le c_{p,M,T}\big[\mathbb{E}\big(|Y_{k}|^{8p}\big)^{1/2}\mathbb{E}\big(e^{8p|Y_{k}|}\big)^{1/2}+(e^{c_{p,M,T}\Delta t}-1)^{4p}+(\Delta t)^{2p}\big]\\
 & \le c_{p,M,T}\big[(e^{c_{M,T}\Delta t}-1)^{4p}+(\Delta t)^{2p}\big],
\end{aligned}
\label{eq:-121}
\end{equation}
Similarly,
\begin{equation}
\begin{aligned}\mathbb{E}\big[(e^{\theta_{k}+Y_{k}}-1-Y_{k})^{4p}\big] & \le c_{p}\mathbb{E}\big[e^{4p\theta_{k}}(e^{Y_{k}}-1-Y_{k})^{4p}+(e^{\theta_{k}}-1)^{4p}(1+Y_{k})^{4p}\big]\\
 & \le c_{p,M,T}\big[\mathbb{E}[(e^{Y_{k}}-1-Y_{k})^{4p}]+(e^{\theta_{k}}-1)^{4p}(1+\mathbb{E}(Y_{k}^{4p}))\big]\\
 & \le c_{p,M,T}\big[\mathbb{E}[(Y_{k}^{2}e^{|Y_{k}|})^{4p}]+(e^{c_{M,T}\Delta t}-1)^{4p}\big]\\
 & \le c_{p,M,T}\big[\mathbb{E}(Y_{k}^{16p})^{1/2}\mathbb{E}(e^{8p|Y_{k}|})^{1/2}+(e^{c_{M,T}\Delta t}-1)^{4p}\big]\\
 & \le c_{p,M,T}\big[(\Delta t)^{4p}+(e^{c_{M,T}\Delta t}-1)^{4p}\big].
\end{aligned}
\label{eq:-122}
\end{equation}
By (\ref{eq:-123}), (\ref{eq:-121}), and (\ref{eq:-122}), we deduce
that
\begin{equation}
\sup_{n<N}\mathbb{E}\big(|u_{n}^{(N)}-v_{n}^{(N)}|^{2p}\big)\le c_{p,M,T}\Big[1+\Big(\frac{e^{c_{M,T}\Delta t}-1}{\Delta t}\Big)^{2}\Big]^{1/2}\big[(e^{c_{M,T}\Delta t}-1)^{4p}+(\Delta t)^{2p}\big]^{1/2}\to0,\label{eq:-124}
\end{equation}
as $N\to\infty$. By the monotonicity of $h^{\prime}$,
\[
\begin{aligned}\mathbb{E}\big[|h(u_{n}^{(N)})-h(v_{n}^{(N)})|^{p}\big] & \le\mathbb{E}\big[\big(|h^{\prime}(u_{n}^{(N)})|^{p}+|h^{\prime}(v_{n}^{(N)})|^{p}\big)\cdot|u_{n}^{(N)}-v_{n}^{(N)}|^{p}\big]\\
 & \le\big[\mathbb{E}\big(|h^{\prime}(u_{n}^{(N)})|^{2p}\big)^{1/2}+\mathbb{E}\big(|h^{\prime}(v_{n}^{(N)})|^{2p}\big)^{1/2}\big]\mathbb{E}\big(|u_{n}^{(N)}-v_{n}^{(N)}|^{2p}\big)^{1/2}.
\end{aligned}
\]
Applying (\ref{eq:-79}) and (\ref{eq:-33}) to $|h^{\prime}|^{2p}$
gives that
\[
\sup_{N\ge1}\sup_{n<N}\mathbb{E}\big(|h^{\prime}(u_{n}^{(N)})|^{2p}\big)^{1/2}+\mathbb{E}\big(|h^{\prime}(v_{n}^{(N)})|^{2p}\big)^{1/2}\le c_{p,\lambda,v_{0},\sigma_{\ast}}<\infty,
\]
which, together with (\ref{eq:-124}) implies (\ref{eq:-31}).
\end{proof}
The following proposition shows that $I_{T}^{(N)}$ converges in distribution
if and only if $X_{T}^{(N)}$ converges in distribution.
\begin{prop}
\label{prop:-4}Let $u_{T}^{(N)}$ be the random variable defined
in (\ref{eq:-14}). Then the random variable
\begin{equation}
\begin{aligned} & \sum_{k=1}^{N}\Big(\dfrac{\bar{\sigma}}{\sqrt{v_{k-1}^{(N)}}}-\frac{\bar{\sigma}}{\sqrt{u_{k-1}^{(N)}}}\Big)\int_{t_{k-1}}^{t_{k}}[\rho(t)-r(t)]dt\\
 & +\frac{1}{2}\sum_{k=1}^{N}\Big(\dfrac{\bar{\sigma}^{2}}{v_{k-1}^{(N)}}-\dfrac{\bar{\sigma}^{2}}{u_{k-1}^{(N)}}\Big)\int_{t_{k-1}}^{t_{k}}\sigma(t)^{2}dt+\sum_{k=1}^{N}\Big(\dfrac{\bar{\sigma}}{\sqrt{v_{k-1}^{(N)}}}-\frac{\bar{\sigma}}{\sqrt{u_{k-1}^{(N)}}}\Big)\int_{t_{k-1}}^{t_{k}}\sigma(t)dW_{t}
\end{aligned}
\label{eq:-20}
\end{equation}
 converges in distribution to zero as $N\to\infty$. Therefore, $I_{T}^{(N)}$
converges in distribution if and only if $X_{T}^{(N)}$ converges
in distribution to the same limiting distribution.
\end{prop}

\begin{proof}
For the first term in (\ref{eq:-20}), applying (\ref{eq:-31}) to
$h(x)=x^{-1/2}$ and $p=2$ gives
\begin{equation}
\lim_{N\to\infty}\sup_{n<N}\mathbb{E}\big[\big|\big(v_{n}^{(N)}\big)^{-1/2}-\big(u_{n}^{(N)}\big)^{-1/2}\big|^{2}\big]=0.\label{eq:-126}
\end{equation}
Therefore,
\begin{equation}
\mathbb{E}\Big[\Big(N^{-1}\sum_{k=1}^{N}\big[\big(v_{k-1}^{(N)}\big)^{-1/2}-\big(u_{k-1}^{(N)}\big)^{-1/2}\big]\Big)^{2}\Big]\le N^{-1}\sum_{k=1}^{N}\mathbb{E}\big[\big|\big(v_{k-1}^{(N)}\big)^{-1/2}-\big(u_{k-1}^{(N)}\big)^{-1/2}\big|^{2}\big]\to0\label{eq:-32}
\end{equation}
as $N\to\infty$, which, together with the boundedness of $\rho(t)$
and $r(t)$, implies that 
\[
\sum_{k=1}^{N}\Big(\dfrac{\bar{\sigma}}{\sqrt{v_{k-1}^{(N)}}}-\frac{\bar{\sigma}}{\sqrt{u_{k-1}^{(N)}}}\Big)\int_{t_{k-1}}^{t_{k}}[\rho(t)-r(t)]dt\to0
\]
in distribution as $N\to\infty$.

For the second term in (\ref{eq:-20}), applying (\ref{eq:-31}) to
$h(x)=x^{-1}$ and $p=2$ gives that $\sup_{n<N}\mathbb{E}[|(v_{n}^{(N)})^{-1}-(u_{n}^{(N)})^{-1}|^{2}]\to0$
as $N\to\infty$. Note that $\int_{t_{k-1}}^{t_{k}}\sigma(t)^{2}dt\le c_{\sigma^{\ast}}T/N$.
Therefore,
\begin{align*}
\mathbb{E}\Big[\Big(\sum_{k=1}^{N}\big[\big(v_{k-1}^{(N)}\big)^{-1}-\big(u_{k-1}^{(N)}\big)^{-1}\big]\int_{t_{k-1}}^{t_{k}}\sigma(t)^{2}dt\Big)^{2}\Big] & \le c_{\sigma^{\ast}}\mathbb{E}\Big[\Big(N^{-1}\sum_{k=1}^{N}\big|\big(v_{k-1}^{(N)}\big)^{-1}-\big(u_{k-1}^{(N)}\big)^{-1}\big|\Big)^{2}\Big]\to0
\end{align*}
as $N\to\infty$. Hence, $\sum_{k=1}^{N}\Big(\dfrac{\bar{\sigma}^{2}}{v_{k-1}^{(N)}}-\dfrac{\bar{\sigma}^{2}}{u_{k-1}^{(N)}}\Big)\int_{t_{k-1}}^{t_{k}}\sigma(t)^{2}dt\to0$
in distribution as $N\to\infty$.

For the last term, let
\[
M_{n}=\sum_{k=1}^{n}\Big(\dfrac{\bar{\sigma}}{\sqrt{v_{k-1}^{(N)}}}-\frac{\bar{\sigma}}{\sqrt{u_{k-1}^{(N)}}}\Big)\int_{t_{k-1}}^{t_{k}}\sigma(t)dW_{t}.
\]
Then $\{M_{n}\}_{n\ge1}$ is a martingale. Moreover,
\begin{align*}
\mathbb{E}(M_{N}^{2}) & =\sum_{k=1}^{N}\mathbb{E}\Big[\Big(\dfrac{\bar{\sigma}}{\sqrt{v_{k-1}^{(N)}}}-\frac{\bar{\sigma}}{\sqrt{u_{k-1}^{(N)}}}\Big)^{2}\Big(\int_{t_{k-1}}^{t_{k}}\sigma(t)dW_{t}\Big)^{2}\Big]\\
 & =\bar{\sigma}^{2}\sum_{k=1}^{N}\mathbb{E}\big[\big|\big(v_{k-1}^{(N)}\big)^{-1/2}-\big(u_{k-1}^{(N)}\big)^{-1/2}\big|^{2}\big]\mathbb{E}\Big[\Big(\int_{t_{k-1}}^{t_{k}}\sigma(t)dW_{t}\Big)^{2}\Big]\\
 & \le\frac{(\bar{\sigma}\sigma^{\ast})^{2}T}{N}\sum_{k=1}^{N}\mathbb{E}\big[\big|\big(v_{k-1}^{(N)}\big)^{-1/2}-\big(u_{k-1}^{(N)}\big)^{-1/2}\big|^{2}\big].
\end{align*}
By the above inequality and (\ref{eq:-126}), we deduce that 
\[
\sum_{k=1}^{n}\Big(\dfrac{\bar{\sigma}}{\sqrt{v_{k-1}^{(N)}}}-\frac{\bar{\sigma}}{\sqrt{u_{k-1}^{(N)}}}\Big)\int_{t_{k-1}}^{t_{k}}\sigma(t)dW_{t}\to0
\]
in distribution as $N\to\infty$. This completes the proof.
\end{proof}

\subsection{\label{subsec:-1}Proof of Theorem \ref{thm:-3}: the constant coefficient
case}

In this subsection, we prove part (i) of Theorem \ref{thm:-3} under
the assumption that $r(t)=r$, $\rho(t)=\rho$, and $\sigma(t)=\sigma$
are constants. The derivation of the upper and lower bounds in part
(ii) is given in Section \ref{sec:}. Under the constant coefficient
assumption, the random variable $\log(X_{T}^{(N)}/I_{0})$ can be
written as
\begin{equation}
\begin{aligned}\log\frac{X_{T}^{(N)}}{I_{0}} & =rT+\frac{(\rho-r)T}{N}\sum_{k=1}^{N}\dfrac{\bar{\sigma}\sigma^{-1}}{\sqrt{\sigma^{-2}u_{k-1}^{(N)}}}\\
 & -\frac{T}{2N}\sum_{k=1}^{N}\dfrac{\bar{\sigma}^{2}}{\sigma^{-2}u_{k-1}^{(N)}}+\sqrt{\frac{T}{N}}\sum_{k=1}^{N}\dfrac{\bar{\sigma}}{\sqrt{\sigma^{-2}u_{k-1}^{(N)}}}Z_{k},
\end{aligned}
\label{eq:-22}
\end{equation}
and $u_{n}^{(N)}$ can be written as 
\begin{equation}
\sigma^{-2}u_{n}^{(N)}=\lambda^{n}\sigma^{-2}v_{0}+(1-\lambda)\sum_{k=1}^{n}\lambda^{n-k}Z_{k}^{2},\label{eq:-23}
\end{equation}
where $Z_{n}=\Delta W_{t_{n}}/\sqrt{\Delta t}$, $n=1,\dots,N$ are
i.i.d. standard random normal distributions.
\begin{prop}
\label{prop:-5}Let $\{u_{n}^{(N)}\}_{n\ge0}$ be the sequence of
random variables defined in (\ref{eq:-14}). Then
\begin{equation}
\lim_{N\to\infty}\frac{1}{N}\sum_{k=1}^{N}\frac{1}{\sqrt{\sigma^{-2}u_{k-1}^{(N)}}}=U(\lambda)\quad\text{a.s.},\label{eq:-4}
\end{equation}
\begin{equation}
\lim_{N\to\infty}\frac{1}{N}\sum_{k=1}^{N}\dfrac{1}{\sigma^{-2}u_{k-1}^{(N)}}=V(\lambda)\quad\text{a.s.},\label{eq:-21}
\end{equation}
and 
\begin{equation}
\frac{1}{\sqrt{N}}\sum_{k=1}^{N}\frac{Z_{k}}{\sqrt{\sigma^{-2}u_{k-1}^{(N)}}}\to\mathcal{N}\big(0,V(\lambda)\big)\label{eq:-5}
\end{equation}
in distribution, where $U(\lambda)$ and $V(\lambda)$ are defined
in (\ref{eq:-16}) and (\ref{eq:-3}) respectively. Therefore, 
\begin{equation}
\log X_{T}^{(N)}\to\mathcal{N}\big((r+(\rho-r)\bar{\sigma}\sigma^{-1}-\bar{\sigma}^{2}V(\lambda)/2)T,\bar{\sigma}^{2}T\Big)\label{eq:-15}
\end{equation}
in distribution, where $X_{T}^{(N)}$ is the random variable defined
in (\ref{eq:-35}).
\end{prop}

\begin{proof}
Let $F(s)=s^{-1/2}$, $\xi_{n}=Z_{n}^{2}$. Then $\mathcal{L}(F)(t)=\sqrt{\pi}t^{-1/2}=\sqrt{\pi}F(t)$,
and therefore $f(t)=\mathcal{L}^{-1}(F)(t)=\frac{1}{\sqrt{\pi t}}$.
Clearly $f(t)$ satisfies the property \ref{eq:-43}. Moreover, $\phi(t)=\mathbb{E}(e^{-t\xi_{1}})=(1+2t)^{-1/2}$
satisfies the condition (\ref{eq:-44}). Applying Theorem \ref{thm:-4}
to $F$, $\xi_{n}$, and $\lambda_{n}=(1-\lambda)\lambda^{n}$ gives
that 
\[
\begin{aligned}\lim_{N\to\infty}\frac{1}{N}\sum_{k=1}^{N}\frac{1}{\sqrt{\sigma^{-2}u_{k-1}^{(N)}}} & =\lim_{N\to\infty}\frac{1}{N}\sum_{k=0}^{N-1}F\big(\sigma^{-2}u_{k-1}^{(N)}\big)\\
 & =\int_{0}^{\infty}\frac{1}{\sqrt{\pi t}}\prod_{k=0}^{\infty}\big[1+2(1-\lambda)\lambda^{k}t\big]^{-1/2}dt=U(\lambda)\quad\text{a.s.},
\end{aligned}
\]
where the last equation is obtained by a simple change of variable. 

Furthermore, since $\mathcal{L}(1)=t^{-1}=F(t)^{2}$, we see that
$f\ast f=\mathcal{L}^{-1}(F^{2})=1$. Applying Theorem \ref{thm:-4}
to $F^{2}$, $\xi_{n}$, and $\lambda_{n}=(1-\lambda)\lambda^{n}$
gives that
\[
\begin{aligned}\lim_{N\to\infty}\frac{1}{N}\sum_{k=1}^{N}\frac{1}{\sigma^{-2}u_{k-1}^{(N)}} & =\lim_{N\to\infty}\frac{1}{N}\sum_{k=0}^{N-1}F\big(\sigma^{-2}u_{k-1}^{(N)}\big)^{2}\\
 & =\int_{0}^{\infty}\prod_{k=0}^{\infty}\big[1+2(1-\lambda)\lambda^{k}t\big]^{-1/2}dt=V(\lambda)\quad\text{a.s.},
\end{aligned}
\]
where the last equation is obtained by a simple change of variable.
Moreover, by Theorem \ref{thm:-5}, we deduce that
\[
\frac{1}{\sqrt{N}}\sum_{k=1}^{N}\frac{Z_{k}}{\sigma_{k-1}}=\frac{1}{\sqrt{N}}\sum_{k=1}^{N}F(\sigma_{k-1}^{2})Z_{k}\to\mathcal{N}\big(0,V(\lambda)\big).
\]
This completes the proof.

The convergence (\ref{eq:-15}) follows immediately from (\ref{eq:-4})
and (\ref{eq:-5}).
\end{proof}
The convergence of marginal distributions of $\log(I_{T}^{(N)})$
follows easily from Proposition \ref{prop:-5}. To deduce convergence
in law, the tightness of the family $\log(I^{(N)})$ is required.
\begin{lem}
\label{lem:-10}Let $a_{1},\dots,a_{n}\in\mathbb{R}$ and $\xi_{1},\dots,\xi_{n}$
be a sequence of random variables. Then, for any $1<p<\infty$,
\begin{equation}
\mathbb{E}\Big(\Big|\sum_{k=1}^{n}a_{k}\xi_{k}\Big|^{p}\Big)\le\Big(\sum_{k=1}^{n}|a_{k}|\Big)^{p}\max_{1\le k\le n}\mathbb{E}(|\xi_{k}|^{p}).\label{eq:-134}
\end{equation}
\end{lem}

\begin{proof}
We may assume $A=\sum_{k=1}^{n}|a_{k}|>0$. By H\"{o}lder's inequality
applied to the discrete probability measure $p_{k}=A^{-1}|a_{k}|$
on $\{1,2,\dots,n\}$, we have
\[
A^{-p}\Big(\sum_{k=1}^{n}|a_{k}\xi_{k}|\Big)^{p}=\Big(\sum_{k=1}^{n}|\xi_{k}|p_{k}\Big)^{p}\le\sum_{k=1}^{n}|\xi_{k}|^{p}p_{k}.
\]
Therefore,
\[
A^{-p}\mathbb{E}\Big(\Big|\sum_{k=1}^{n}a_{k}\xi_{k}\Big|^{p}\Big)\le\sum_{k=1}^{n}\mathbb{E}(|\xi_{k}|^{p})p_{k}\le\max_{1\le k\le n}\mathbb{E}(|\xi_{k}|^{p}).
\]
This completes the proof.
\end{proof}
\begin{lem}
\label{lem:-11}Let $s,t\in[0,T]$ and $1<p<\infty$. Then
\begin{equation}
\mathbb{E}\big(\big|\log I_{t}^{(N)}-\log I_{s}^{(N)}\big|^{p}\big)\le L_{p,\lambda,v_{0},\bar{\sigma},r,\rho,\sigma,T}|t-s|^{p/2},\label{eq:-135}
\end{equation}
where $L_{p,\lambda,v_{0},\bar{\sigma},r,\rho,\sigma,T}$ is a constant
depending only on $p,\,\lambda,\,v_{0},\,\bar{\sigma},\,\Vert r\Vert_{L^{\infty}},\,\Vert\rho\Vert_{L^{\infty}},\,\Vert\sigma\Vert_{L^{\infty}},\,\sigma_{\ast},\,T$.
\end{lem}

\begin{proof}
We may assume $s<t$. We first prove (\ref{eq:-135}) for the case
when $s$ and $t$ are grid points, i.e. $s=t_{\ell}=\ell\Delta t$
and $t=t_{m}=m\Delta t$ for some $0\le\ell<m\le N$. By (\ref{eq:-64}),
\[
\begin{aligned}\log I_{t_{m}}^{(N)}-\log I_{t_{\ell}}^{(N)} & =\int_{t_{\ell}}^{t_{m}}r(u)du+\sum_{k=\ell+1}^{m}\frac{\bar{\sigma}}{\sqrt{v_{k-1}^{(N)}}}\int_{t_{k-1}}^{t_{k}}[\rho(u)-r(u)]du\\
 & \quad-\frac{1}{2}\sum_{k=\ell+1}^{m}\frac{\bar{\sigma}^{2}}{v_{k-1}^{(N)}}\int_{t_{k-1}}^{t_{k}}\sigma(u)^{2}du+\sum_{k=\ell+1}^{m}\frac{\bar{\sigma}}{\sqrt{v_{k-1}^{(N)}}}\int_{t_{k-1}}^{t_{k}}\sigma(u)dW_{u}.
\end{aligned}
\]
Therefore,
\begin{equation}
\begin{aligned} & \mathbb{E}\big(\big|\log I_{t_{m}}^{(N)}-\log I_{t_{\ell}}^{(N)}\big|^{p}\big)\\
 & \le c_{p}\Big[\Big|\int_{t_{\ell}}^{t_{m}}r(u)du\Big|^{p}+\mathbb{E}\Big(\Big|\sum_{k=\ell+1}^{m}\frac{\bar{\sigma}}{\sqrt{v_{k-1}^{(N)}}}\int_{t_{k-1}}^{t_{k}}[\rho(u)-r(u)]du\Big|^{p}\Big)\\
 & \quad+\mathbb{E}\Big(\Big|\sum_{k=\ell+1}^{m}\frac{\bar{\sigma}^{2}}{v_{k-1}^{(N)}}\int_{t_{k-1}}^{t_{k}}\sigma(u)^{2}du\Big|^{p}\Big)+\mathbb{E}\Big(\Big|\sum_{k=\ell+1}^{m}\frac{\bar{\sigma}}{\sqrt{v_{k-1}^{(N)}}}\int_{t_{k-1}}^{t_{k}}\sigma(u)dW_{u}\Big|^{p}\Big)\\
 & =c_{p}\Big[\Big|\int_{t_{\ell}}^{t_{m}}r(u)du\Big|^{p}+E_{1}+E_{2}+E_{3}\Big].
\end{aligned}
\label{eq:-136}
\end{equation}
Clearly,
\begin{equation}
\Big|\int_{t_{\ell}}^{t_{m}}r(u)du\Big|^{p}\le\Vert r\Vert_{L^{\infty}}^{p}|t_{m}-t_{\ell}|^{p}.\label{eq:-137}
\end{equation}
By Lemma \ref{lem:-10}, we deduce that
\begin{equation}
\begin{aligned}E_{1} & \le\bar{\sigma}^{p}\Big|\int_{t_{\ell}}^{t_{m}}[\rho(u)-r(u)]du\Big|^{p}\max_{\ell\le k<m}\mathbb{E}\big[\big(v_{k}^{(N)}\big)^{-p/2}\big]\\
 & \le\bar{\sigma}^{p}\big(\Vert\rho\Vert_{L^{\infty}}+\Vert r\Vert_{L^{\infty}}\big)^{p}|t_{m}-t_{\ell}|^{p}\max_{\ell\le k<m}\mathbb{E}\big[\big(v_{k}^{(N)}\big)^{-p/2}\big],
\end{aligned}
\label{eq:-138}
\end{equation}
and
\begin{equation}
\begin{aligned}E_{2} & \le\bar{\sigma}^{2p}\Big|\int_{t_{\ell}}^{t_{m}}\sigma(u)^{2}du\Big|^{p}\max_{\ell\le k<m}\mathbb{E}\big[\big(v_{k}^{(N)}\big)^{-p}\big]\\
 & \le\bar{\sigma}^{2p}\Vert\sigma\Vert_{L^{\infty}}^{2p}|t_{m}-t_{\ell}|^{p}\max_{\ell\le k<m}\mathbb{E}\big[\big(v_{k}^{(N)}\big)^{-p}\big].
\end{aligned}
\label{eq:-139}
\end{equation}
For $E_{3}$, by the Burkholder--Davis--Gundy inequality,
\begin{equation}
E_{3}\le c_{p}\mathbb{E}\Big(\Big|\sum_{k=\ell+1}^{m}\frac{\bar{\sigma}^{2}}{v_{k-1}^{(N)}}\int_{t_{k-1}}^{t_{k}}\sigma(u)^{2}du\Big|^{p/2}\Big).\label{eq:-140}
\end{equation}
Similar to (\ref{eq:-139}), we obtain from (\ref{eq:-140}) that
\begin{equation}
E_{3}\le c_{p}\bar{\sigma}^{p}\Vert\sigma\Vert_{L^{\infty}}^{p}|t_{m}-t_{\ell}|^{p/2}\max_{\ell\le k<m}\mathbb{E}\big[\big(v_{k}^{(N)}\big)^{-p/2}\big].\label{eq:-141}
\end{equation}
By Lemma \ref{lem:-5},
\begin{equation}
\sup_{N\ge1}\max_{k<N}\mathbb{E}\big[\big(v_{k}^{(N)}\big)^{-p/2}\big]\le\sup_{N\ge1}\max_{k<N}\mathbb{E}\big[\big(v_{k}^{(N)}\big)^{-p}\big]^{1/2}\le c_{p,\lambda,v_{0},\sigma_{\ast}}<\infty\label{eq:-142}
\end{equation}
for some $c_{p,\lambda,v_{0},\sigma_{\ast}}$ depending only on $p,\lambda,v_{0},\sigma_{\ast}$.
Combining (\ref{eq:-136}), (\ref{eq:-137}), (\ref{eq:-138}), (\ref{eq:-139}),
(\ref{eq:-141}) and (\ref{eq:-142}) gives
\begin{equation}
\begin{aligned} & \mathbb{E}\big(\big|\log I_{t_{m}}^{(N)}-\log I_{t_{\ell}}^{(N)}\big|^{p}\big)\\
 & \le c_{p,\lambda,v_{0},\sigma_{\ast}}\big[\big(\Vert r\Vert_{L^{\infty}}+\bar{\sigma}\Vert\rho\Vert_{L^{\infty}}+\bar{\sigma}\Vert r\Vert_{L^{\infty}}\big)^{p}|t_{m}-t_{\ell}|^{p}+\bar{\sigma}^{p}\Vert\sigma\Vert_{L^{\infty}}^{p}|t_{m}-t_{\ell}|^{p/2}\big]\\
 & \le c_{p,\lambda,v_{0},\sigma_{\ast}}\big[\big(\Vert r\Vert_{L^{\infty}}+\bar{\sigma}\Vert\rho\Vert_{L^{\infty}}+\bar{\sigma}\Vert r\Vert_{L^{\infty}}\big)^{p}T^{p/2}+\bar{\sigma}^{p}\Vert\sigma\Vert_{L^{\infty}}^{p}\big]|t_{m}-t_{\ell}|^{p/2}\\
 & =L_{p,\lambda,v_{0},\bar{\sigma},r,\rho,\sigma,T}|t_{m}-t_{\ell}|^{p/2}.
\end{aligned}
\label{eq:-143}
\end{equation}
This proves (\ref{eq:-135}) for the case when $s$, $t$ are grid
points.

For the general case, we consider the following two cases.

\textbf{Case 1}: When there is no grid point in $(s,t)$; that is,
$t_{m}\le s<t\le t_{m+1}$ for some $m$.

By (\ref{eq:-64}), we have
\begin{equation}
\begin{aligned}\log I_{t}^{(N)}-\log I_{s}^{(N)} & =\int_{s}^{t}r(u)du+\frac{\bar{\sigma}}{\sqrt{v_{m}^{(N)}}}\int_{s}^{t}[\rho(u)-r(u)]du\\
 & \quad-\frac{1}{2}\frac{\bar{\sigma}^{2}}{v_{m}^{(N)}}\int_{s}^{t}\sigma(u)^{2}du+\frac{\bar{\sigma}}{\sqrt{v_{m}^{(N)}}}\int_{s}^{t}\sigma(u)dW_{u}.
\end{aligned}
\label{eq:-144}
\end{equation}
Now (\ref{eq:-135}) follows easily from (\ref{eq:-144}), (\ref{eq:-142})
and the Burkholder--Davis--Gundy inequality.

\textbf{Case 2}: When there is at least one grid points in $(s,t)$;
that is, $t_{\ell-1}\le s<t_{\ell}\le t_{m}<t\le t_{m+1}$ for some
$\ell\le m$.

Applying the result of Case 1 to $(s,t_{\ell})$ and $(t_{m},t)$
gives
\[
\mathbb{E}\big(\big|\log I_{t_{\ell}}^{(N)}-\log I_{s}^{(N)}\big|^{p}\big)\le L_{p,\lambda,v_{0},\bar{\sigma},r,\rho,\sigma,T}|t_{\ell}-s|^{p/2},
\]
and
\[
\mathbb{E}\big(\big|\log I_{t}^{(N)}-\log I_{t_{m}}^{(N)}\big|^{p}\big)\le L_{p,\lambda,v_{0},\bar{\sigma},r,\rho,\sigma,T}|t-t_{m}|^{p/2},
\]
which, together with (\ref{eq:-143}), imply that
\[
\mathbb{E}\big(\big|\log I_{t}^{(N)}-\log I_{s}^{(N)}\big|^{p}\big)\le L_{p,\lambda,v_{0},\bar{\sigma},r,\rho,\sigma,T}\big(|t_{\ell}-s|^{p/2}+|t_{m}-t_{\ell}|^{p/2}+|t-t_{m}|^{p/2}\big)\le L_{p,\lambda,v_{0},\bar{\sigma},r,\rho,\sigma,T}|t-s|^{p/2}.
\]
This completes the proof of Lemma \ref{lem:-11}.
\end{proof}
\begin{prop}
\label{prop:-2}The results in part (i) of Theorem \ref{thm:-3} hold
under the addition assumption that $r(t)=r$, $\rho(t)=\rho$ and
$\sigma(t)=\sigma$ are constants.
\end{prop}

\begin{proof}
By (\ref{eq:-22}) and Proposition \ref{prop:-5}, we see that $\log(I_{T}^{(N)}/I_{t_{0}})$
converges in distribution to
\[
\mathcal{N}\big((r+(\rho-r)\bar{\sigma}\sigma^{-1}U(\lambda)-\bar{\sigma}^{2}V(\lambda)/2)T,\bar{\sigma}^{2}T\Big)
\]
as $\Delta t\to0$. To deduce convergence in law from the convergence
of marginal distribution, let $\{\mathcal{F}_{t}\}_{t\ge0}$ be the
filtration generated by the Brownian motion $\{W_{t}\}_{t\ge0}$,
and let $0<s<t\le T$. Conditional on $\mathcal{F}_{s}$, by the above
convergence of marginal distribution, we see that
\begin{equation}
\log\Big(\frac{I_{t}^{(N)}}{I_{s}^{(N)}}\Big)\to\mathcal{N}\big((r+(\rho-r)\bar{\sigma}\sigma^{-1}U(\lambda)-\bar{\sigma}^{2}V(\lambda)/2)(t-s),\bar{\sigma}^{2}(t-s)\Big).\label{eq:-37}
\end{equation}
Note that the limiting distribution in (\ref{eq:-37}) is independent
of $\mathcal{F}_{s}$, which implies the convergence of finite dimensional
distributions of $\{\log I_{t}^{(N)}\}_{t\ge0}$. By Lemma \ref{lem:-11}
and the Kolmogorov's tightness criterion, the sequence $\{\log I_{t}^{(N)}\}_{t\ge0}$
is tight in $C([0,T])$. Therefore, the sequence $\{\log I_{t}^{(N)}\}$
converges in law to a continuous process $\{\log X_{t}\}_{t\ge0}$
which has i.i.d. Wiener increments. Hence $\{\log X_{t}\}_{t\ge0}$
is a scaled Wiener process with drift, which completes the proof.
\end{proof}
\begin{rem}
\label{rem:}We end this subsection with some comments on how the
result of Proposition \ref{prop:-5} can be extended to the variant
of volatility target index defined by (\ref{eq:-85}). The analogues
of Proposition \ref{prop:-5} for this variant are the three convergence
\begin{align*}
\lim_{N\to\infty}\frac{1}{N}\sum_{n=1}^{N}\sigma w_{n-1}^{(N)} & =U(\lambda_{1},\lambda_{2},w^{\ast}),\\
\lim_{N\to\infty}\frac{1}{N}\sum_{n=1}^{N}\sigma^{2}\big(w_{n-1}^{(N)}\big)^{2} & =V(\lambda_{1},\lambda_{2},w^{\ast}),\\
\frac{1}{\sqrt{N}}\sum_{n=1}^{N}\sigma w_{n-1}^{(N)}Z_{n} & \to\mathcal{N}\big(0,V(\lambda_{1},\lambda_{2},w^{\ast})\big).
\end{align*}
As we see from the proof of Theorem \ref{thm:-4} and Theorem \ref{thm:-5},
the derivation of these convergence lies on the following key ingredients:
(i) convergence of $\mathbb{E}(w_{n}^{(N)})$ as $N\to\infty,n\to\infty$;
(ii) convergence of $\mathbb{E}[(w_{n}^{(N)})^{2}]$ as $N\to\infty,n\to\infty$;
and (iii) exponential decay of covariance $\text{Cov}(w_{n}^{(N)},w_{n+k}^{(N)})\le c\lambda^{k}$
for some $c$ and $\lambda$. The convergence of $\mathbb{E}(w_{n}^{(N)})$
and $\mathbb{E}[(w_{n}^{(N)})^{2}]$ can be established once convergence
in distribution of $w_{n}^{(N)}$ is obtained. To see the convergence
in distribution of $w_{n}^{(N)}$, let $F(x_{1},x_{2})=\min(x_{1},x_{2},w^{\ast})$.
Then $w_{n}^{(N)}=F(v_{n}^{(1,N)},v_{n}^{(2,N)})$. Note that the
characteristic functions of the random vector $(v_{n}^{(1,N)},v_{n}^{(2,N)})$
is
\[
\mathbb{E}\big[\exp\big(\text{i}\xi_{1}v_{n}^{(1,N)}+\text{i}\xi_{2}v_{n}^{(2,N)}\big)\big]=\prod_{k=0}^{n-1}[1-2\text{i}\xi_{1}(1-\lambda_{1})\lambda_{1}^{k}-2\text{i}\xi_{2}(1-\lambda_{2})\lambda_{2}^{k}]^{-1/2},
\]
which converges as $n\to\infty$. This implies that $(v_{n}^{(1,N)},v_{n}^{(2,N)})$,
and therefore $w_{n}^{(N)}$, converge in distribution. Moreover,
the expectations $\mathbb{E}(w_{n}^{(N)})$ and $\mathbb{E}[(w_{n}^{(N)})^{2}]$
can be expressed using the inverse Laplace transform of $F(x_{1},x_{2})$.
For the exponential covariance decay, note that $v_{n+k}^{(j,N)}=\lambda_{j}^{k}v_{n}^{(j,N)}+(1-\lambda_{j})\sum_{l=1}^{k}\lambda_{j}^{k-l}Z_{n+l}^{2}$
and therefore
\begin{align*}
\mathbb{E}\big[F(v_{n}^{(j,N)})F(v_{n+k}^{(j,N)})\big] & \le\mathbb{E}\Big[F(v_{n}^{(j,N)})F\Big((1-\lambda_{j})\sum_{l=1}^{k}\lambda_{j}^{k-l}Z_{n+l}^{2}\Big)\Big]\\
 & =\mathbb{E}\big[F(v_{n}^{(j,N)})\big]\mathbb{E}\Big[F\Big((1-\lambda_{j})\sum_{l=1}^{k}\lambda_{j}^{k-l}Z_{l}^{2}\Big)\Big]=\mathbb{E}\big[F(v_{n}^{(j,N)})\big]\mathbb{E}\big[F(v_{k}^{(j,N)}-\lambda_{j}^{k}v_{0})\big],
\end{align*}
where we have used the notation $F(v_{n}^{(j,N)})=F(v_{n}^{(1,N)},v_{n}^{(2,N)})$.
It is routinary to verify that $\mathbb{E}\big[F(v_{k}^{(j,N)})\big]-\mathbb{E}\big[F(v_{k}^{(j,N)}-\lambda_{j}^{k}v_{0})\big]$
has exponential decay in $k$. Combining the above gives the desired
exponential decay of covariance. Thus, all key ingredients can be
established.
\end{rem}

\subsection{\label{subsec:-4}Proof of Proposition \ref{prop:-1}}

We now turn to the proof of the vega conversion formula in Proposition
\ref{prop:-1}.
\begin{proof}[Proof of Proposition \ref{prop:-1}]
Let $J_{T}^{(N)}=\partial_{\sigma}\log(I_{T}^{(N)})$ be the path
wise derivative of $\log(I_{T}^{(N)})$ with respect to $\sigma$.
Then
\[
\partial_{\sigma}\psi_{S,\Delta t}(r,\rho,\sigma)=\mathbb{E}\Big[e^{-\int_{0}^{T}r^{\text{disc}}(t)dt}G^{\prime}(I_{T}^{(N)})I_{T}^{(N)}J_{T}^{(N)}\Big].
\]
Similarly, let $L_{T}^{(N)}=\partial_{\sigma}\log(X_{T}^{(N)})$ be
the path wise derivative of $\log(X_{T}^{(N)})$ with respect to $\sigma$.
By an argument similar to the proof of Proposition \ref{prop:-4},
it can be shown that $\lim_{N\to\infty}(J_{T}^{(N)}-L_{T}^{(N)})=0$
in distribution. Therefore, by the boundedness of $G^{\prime}$ and
the dominated convergence theorem,
\begin{equation}
\lim_{\Delta t\to0}\partial_{\sigma}\psi_{S,\Delta t}(r,\rho,\sigma)=\lim_{N\to\infty}\mathbb{E}\Big[e^{-\int_{0}^{T}r^{\text{disc}}(t)dt}G^{\prime}(X_{T}^{(N)})X_{T}^{(N)}L_{T}^{(N)}\Big],\label{eq:-24}
\end{equation}
provided that the limit on the right hand side exists. We now turn
to the limiting distribution of $L_{T}^{(N)}$. Let $\beta_{T}^{(N)}$
be the path wise derivative of $\sigma^{-2}u_{n}^{(N)}$ with respect
to $\sigma$, that is $\beta_{n}^{(N)}=-2\lambda^{n}\sigma^{-3}v_{0}$.
Then
\[
\begin{aligned}L_{T}^{(N)} & =\frac{(r-\rho)T}{N}\sum_{k=1}^{N}\dfrac{\bar{\sigma}\sigma^{-2}}{\sqrt{\sigma^{-2}u_{k-1}^{(N)}}}+\frac{(r-\rho)T}{2N}\sum_{k=1}^{N}\dfrac{\bar{\sigma}\sigma^{-1}\beta_{k-1}^{(N)}}{\big(\sigma^{-2}u_{k-1}^{(N)}\big)^{3/2}}\\
 & \quad+\frac{T}{2N}\sum_{k=1}^{N}\dfrac{\bar{\sigma}^{2}\beta_{k-1}^{(N)}}{\big(\sigma^{-2}u_{k-1}^{(N)}\big)^{2}}-\sqrt{\frac{T}{4N}}\sum_{k=1}^{N}\dfrac{\bar{\sigma}\beta_{k-1}^{(N)}}{\big(\sigma^{-2}u_{k-1}^{(N)}\big)^{3/2}}Z_{k}\\
 & =(r-\rho)T\bar{\sigma}\sigma^{-2}H_{1}+\frac{(r-\rho)T\bar{\sigma}\sigma^{-1}}{2}H_{2}+\frac{T\bar{\sigma}^{2}}{2}H_{3}-\sqrt{\frac{T\bar{\sigma}^{2}}{4}}H_{4}.
\end{aligned}
\]

For $H_{1}$, by (\ref{eq:-4}), 
\[
\lim_{N\to\infty}H_{1}=\frac{1}{N}\sum_{k=1}^{N}\dfrac{1}{\sqrt{\sigma^{-2}u_{k-1}^{(N)}}}=U(\lambda)\quad\text{a.s.}
\]

For $H_{2}$ and $H_{3}$, by applying (\ref{eq:-79}) to $h(x)=x^{-3/2}$,
we see that $A_{1}=\sup_{N\ge1}\sup_{n<N}\mathbb{E}\big[\big(\sigma^{-2}u_{n}^{(N)}\big)^{-3/2}\big]<\infty$.
Therefore,
\[
\begin{aligned}\mathbb{E}(|H_{2}|) & \le\frac{(r-\rho)Tv_{0}\bar{\sigma}\sigma^{-4}}{N}\sum_{k=1}^{N}\lambda^{k-1}\mathbb{E}\big[\big(\sigma^{-2}u_{k-1}^{(N)}\big)^{3/2}\big]\le\frac{c_{r,\rho,T,\sigma,\bar{\sigma},\lambda}A_{1}}{N}\to0.\end{aligned}
\]
Therefore, $H_{2}\to0$ in distribution as $N\to\infty$. By a similar
argument, it is easy to show that $H_{3}\to0$ in distribution as
$N\to\infty$.

For $H_{4}$, let
\[
M_{n}=-\frac{2\sigma^{-3}v_{0}}{\sqrt{N}}\sum_{k=1}^{n}\dfrac{\lambda^{k-1}}{\big(\sigma^{-2}u_{k-1}^{(N)}\big)^{3/2}}Z_{k}.
\]
Then $\{M_{n}\}_{n\ge1}$ is a martingale. Therefore,
\begin{equation}
\begin{aligned}\mathbb{E}(M_{N}^{2}) & =\frac{4\sigma^{-6}v_{0}^{2}}{N}\sum_{k=1}^{N}\mathbb{E}\Big[\dfrac{\lambda^{2k-2}}{\big(\sigma^{-2}u_{k-1}^{(N)}\big)^{3}}Z_{k}^{2}\Big]=\frac{4\sigma^{-6}v_{0}^{2}}{N}\sum_{k=1}^{N}\lambda^{2k-2}\mathbb{E}\big[\big(\sigma^{-2}u_{k-1}^{(N)}\big)^{-3}\big].\end{aligned}
\label{eq:-25}
\end{equation}
Applying (\ref{eq:-79}) to $h(x)=x^{-3}$ shows that $A_{2}=\sup_{N\ge1}\sup_{n<N}\mathbb{E}\big[\big(\sigma^{-2}u_{n}^{(N)}\big)^{-3}\big]<\infty$,
which, together with (\ref{eq:-25}) implies that $H_{4}=M_{N}\to0$
in distribution as $N\to\infty$.

Therefore, we have shown that $L_{T}^{(N)}\to(r-\rho)T\bar{\sigma}\sigma^{-2}U(\lambda)$
in distribution. Moreover, by (\ref{eq:-24}),
\begin{equation}
\lim_{\Delta t\to0}\partial_{\sigma}\psi_{S,\Delta t}(r,q,\sigma)=(r-\rho)T\bar{\sigma}\sigma^{-2}U(\lambda)\mathbb{E}\Big[e^{-\int_{0}^{T}r^{\text{disc}}(t)dt}G^{\prime}(X_{T})X_{T}\Big].\label{eq:-27}
\end{equation}
On the other hand, by
\[
d\log(Y_{t}^{\mu,\nu})=(\mu(t)-\nu(t)^{2}/2)dt+\nu(t)dW_{t},
\]
the path wise derivative of $\log(Y_{T}^{\mu,\nu})$ with respect
to $\mu$ is $\partial_{\mu}\log(Y_{T}^{\mu,\nu})=T$. Therefore,
\[
\partial_{\mu}\psi_{Y}(\mu,\nu)=T\mathbb{E}\Big[e^{-\int_{0}^{T}r^{\text{disc}}(t)dt}G^{\prime}(Y_{T}^{\mu,\nu})Y_{T}^{\mu,\nu}\Big].
\]
Since $Y_{t}^{r+(\rho-r)\sigma^{-1}\bar{\sigma}U(\lambda),\bar{\sigma}V(\lambda)^{1/2}}=X_{t}$,
we deduce that
\[
\lim_{\Delta t\to0}\partial_{\sigma}\psi_{S,\Delta t}(r,q,\sigma)=(r-\rho)\bar{\sigma}\sigma^{-2}U(\lambda)\partial_{\mu}\psi_{Y}\big(r+(\rho-r)\sigma^{-1}\bar{\sigma}U(\lambda),\bar{\sigma}V(\lambda)^{1/2}\big).
\]
\end{proof}

\subsection{\label{subsec:-2}Proof of Theorem \ref{thm:-3}: the general case}

In this section, we consider the general case when $r(t)$, $\rho(t)$
and $\sigma(t)$ are deterministic functions of $t$. Let $L^{2}(\mathbb{P};L^{\infty}(0,T))$
be space of stochastic processes $\{Q_{t}\}_{t\ge0}$ such that
\[
\Vert Q\Vert_{L^{2}(\mathbb{P};L^{\infty}(0,T))}=\Big[\mathbb{E}\Big(\sup_{t\in[0,T]}|Q_{t}|\Big)^{2}\Big]^{1/2}<\infty.
\]
Let $L_{+}^{2}(0,T)$ be the convex subset of $L^{2}(0,T)$ consisting
of functions $\sigma\in L^{2}(0,T)$ such that $\inf_{t\in[0,T]}\sigma(t)>0$.
For any integer $N\ge1$ and any $(r,\rho,\sigma)\in L^{1}(0,T)\times L^{1}(0,T)\times L_{+}^{2}(0,T)$.
Let $\mathfrak{I}^{(N)}$ be the mapping which maps $(r,\rho,\sigma)$
to the process (\ref{eq:-125}) , that is,
\begin{equation}
\begin{aligned}\frac{d\mathfrak{I}^{(N)}(r,\rho,\sigma)_{t}}{\mathfrak{I}^{(N)}(r,\rho,\sigma)_{t}} & =\Big(1-\dfrac{\bar{\sigma}}{\sqrt{u_{n-1}^{(N)}}}\Big)r(t)dt+\dfrac{\bar{\sigma}}{\sqrt{u_{n-1}^{(N)}}}\frac{dS_{t}}{S_{t}},\quad t\in[t_{n-1},t_{n}),\\
\mathfrak{I}^{(N)}(r,\rho,\sigma)_{0} & =I_{0}.
\end{aligned}
\label{eq:-68}
\end{equation}
Similarly, define $\mathfrak{I}$ to be the mapping which maps $(r,\rho,\sigma)$
to the process (\ref{eq:-39}), that is,
\begin{equation}
d\mathfrak{I}(r,\rho,\sigma)_{t}=\big[r(t)+(\rho(t)-r(t))\sigma(t)^{-1}\bar{\sigma}U(\lambda)\big]\mathfrak{J}(r,\rho,\sigma)_{t}dt+\bar{\sigma}V(\lambda)^{1/2}\mathfrak{J}(r,\rho,\sigma)_{t}dW_{t}.\label{eq:-77}
\end{equation}
 By Proposition \ref{prop:-4}, $I_{t}^{(N)}\to X_{t}=\mathfrak{I}(r,\rho,\sigma)_{t}$
in law if and only if $\mathfrak{I}^{(N)}(r,\rho,\sigma)_{t}\to\mathfrak{I}(r,\rho,\sigma)_{t}$
in law. We now turn to prove the general case of part (i) of Theorem
\ref{thm:-3}.
\begin{proof}[Proof of Theorem \ref{thm:-3}, (i)]
Let $L^{1}(0,T)\times L^{1}(0,T)\times L_{+}^{2}(0,T)$ be the metric
space induced by the metric $\text{dist}((r,\rho,\sigma),(\tilde{r},\tilde{\rho},\tilde{\sigma}))=\Vert r-\tilde{r}\Vert_{L^{1}(0,T)}+\Vert\rho-\tilde{\rho}\Vert_{L^{1}(0,T)}+\Vert\sigma-\tilde{\sigma}\Vert_{L^{2}(0,T)}$,
and let
\[
\mathfrak{I}^{(N)}:L^{1}(0,T)\times L^{1}(0,T)\times L_{+}^{2}(0,T)\to L^{2}(\mathbb{P};L^{\infty}(0,T)),
\]
and 
\[
\mathfrak{I}:L^{1}(0,T)\times L^{1}(0,T)\times L_{+}^{2}(0,T)\to L^{2}(\mathbb{P};L^{\infty}(0,T)),
\]
be the mappings defined by (\ref{eq:-68}) and (\ref{eq:-77}) respectively.
We need to prove that $\lim_{N\to\infty}\mathfrak{I}^{(N)}(r,\rho,\sigma)=\mathfrak{I}(r,\rho,\sigma)$
in law. 

Let $\{\mathcal{F}_{t}\}_{t\ge0}$ be the filtration generated by
$\{W_{t}\}_{t\ge0}$. For any simple functions $r(t)=\sum_{j=1}^{m}r_{j}\chi_{[s_{j},s_{j+1})}(t)$,
$\rho(t)=\sum_{j=1}^{m}\rho_{j}\chi_{[s_{j},s_{j+1})}(t)$, and $\sigma(t)=\sum_{j=1}^{m}\sigma_{j}\chi_{[s_{j},s_{j+1})}(t)$,
with $\sigma_{j}>0$. Denote $\mathfrak{I}=\mathfrak{I}(r,\rho,\sigma)$
and $\mathfrak{J}^{(N)}=\mathfrak{J}^{(N)}(r,\rho,\sigma)$. Since
$r(t)$, $\rho(t)$, $\sigma(t)$ are constants on $[s_{j},s_{j+1})$,
by Proposition \ref{prop:-2}, conditional on $\mathcal{F}_{s_{j}}$,
\[
\log\frac{\mathfrak{I}_{t}^{(N)}}{\mathfrak{I}_{s_{j}}^{(N)}}\to\log\frac{\mathfrak{I}_{t}}{\mathfrak{I}_{s_{j}}},\quad t\in[s_{j},s_{j+1}),
\]
in distribution as $N\to\infty$, and the limiting distribution is
independent of $\mathcal{F}_{s_{j}}$. Therefore, with $\tau(t)=\sup\{j:s_{j}\le t\}$
\begin{align*}
\lim_{N\to\infty}\log\frac{\mathfrak{I}_{t}^{(N)}}{\mathfrak{I}_{s}^{(N)}} & =\sum_{j=\tau(s)+1}^{\tau(t)}\lim_{N\to\infty}\log\frac{\mathfrak{I}_{s_{j}}^{(N)}}{\mathfrak{I}_{s_{j-1}}^{(N)}}+\lim_{N\to\infty}\log\frac{\mathfrak{I}_{t}^{(N)}}{\mathfrak{I}_{s_{\tau(t)}}^{(N)}}\\
 & =\sum_{j=\tau(s)+1}^{\tau(t)}\log\frac{\mathfrak{I}_{s_{j}}}{\mathfrak{I}_{s_{j-1}}}+\log\frac{\mathfrak{I}_{t}}{\mathfrak{I}_{s_{\tau(t)}}}\\
 & =\log\frac{\mathfrak{I}_{t}}{\mathfrak{I}_{s}},
\end{align*}
in distribution. Therefore, $\lim_{N\to\infty}\mathfrak{I}^{(N)}(r,\rho,\sigma)=\mathfrak{I}(r,\rho,\sigma)$
in law for simple functions $r$, $\rho$, $\sigma$.

We will show that, for any $0<\delta<M<\infty$, $\log\mathfrak{I}^{(N)}(r,\rho,\sigma)$
is uniformly continuous in $(r,\rho,\sigma)$ on the following closed
subset of $L^{1}(0,T)\times L^{1}(0,T)\times L_{+}^{2}(0,T)$,
\begin{equation}
A_{\delta,M}=\Big\{(r,\rho,\sigma):\sup_{t\in[0,T]}(|r(t)|+|\rho(t)|)\le M,\delta\le\inf_{t\in[0,T]}\sigma(t)\le\sup_{t\in[0,T]}\sigma(t)\le M\Big\}.\label{eq:-92}
\end{equation}
Assume the uniform continuity of $\log\mathfrak{I}^{(N)}$ on $A_{\delta,M}$
for the moment. By Itô's formula, $\log\mathfrak{J}(r,\rho,\sigma)$
is continuous in $(r,\rho,\sigma)$ on $A_{\delta,M}$. For any $\epsilon>0$,
since simple functions are dense in $L^{1}(0,T)$ and $L_{+}^{2}(0,T)$,
we may choose simple functions $\tilde{r}$, $\tilde{\rho}$ and $\tilde{\sigma}$
such that $(\tilde{r},\tilde{\rho},\tilde{\sigma})\in A_{\delta,M}$
and $\sup_{N\ge1}\Vert\log\mathfrak{I}^{(N)}-\log\tilde{\mathfrak{I}}^{(N)}\Vert_{L^{2}(\mathbb{P};L^{\infty}(0,T))}\le\epsilon$
and $\Vert\log\mathfrak{I}-\log\tilde{\mathfrak{I}}\Vert_{L^{2}(\mathbb{P};L^{\infty}(0,T))}\le\epsilon$,
where, to simplify notation, we denote $\mathfrak{I}=\mathfrak{I}(r,\rho,\sigma)$,
$\tilde{\mathfrak{I}}=\mathfrak{I}(\tilde{r},\tilde{\rho},\tilde{\sigma})$,
and similarly for $\mathfrak{J}^{(N)}$ and $\tilde{\mathfrak{J}}^{(N)}$.
For any smooth function $f$ with bounded derivative, by
\begin{align}
\Big|\mathbb{E}\Big[f\Big(\log\frac{\mathfrak{I}_{t}^{(N)}}{\mathfrak{I}_{s}^{(N)}}\Big)\Big]-\mathbb{E}\Big[f\Big(\log\frac{\mathfrak{I}_{t}}{\mathfrak{I}_{s}}\Big)\Big]\Big| & \le\mathbb{E}\Big[\Big|f\Big(\log\frac{\mathfrak{I}_{t}^{(N)}}{\mathfrak{I}_{s}^{(N)}}\Big)-f\Big(\log\frac{\tilde{\mathfrak{I}}_{t}^{(N)}}{\tilde{\mathfrak{I}}_{s}^{(N)}}\Big)\Big|\Big]\nonumber \\
 & \quad+\mathbb{E}\Big[\Big|f\Big(\log\frac{\mathfrak{I}_{t}}{\mathfrak{I}_{s}}\Big)-f\Big(\log\frac{\tilde{\mathfrak{I}}_{t}}{\tilde{\mathfrak{I}}_{s}}\Big)\Big|\Big]+\Big|\mathbb{E}\Big[f\Big(\log\frac{\tilde{\mathfrak{I}}_{t}^{(N)}}{\tilde{\mathfrak{I}}_{s}^{(N)}}\Big)-f\Big(\log\frac{\tilde{\mathfrak{I}}_{t}}{\tilde{\mathfrak{I}}_{s}}\Big)\Big]\Big|,\label{eq:-112}
\end{align}
and
\[
\Big|f\Big(\log\frac{\mathfrak{I}_{t}^{(N)}}{\mathfrak{I}_{s}^{(N)}}\Big)-f\Big(\log\frac{\tilde{\mathfrak{I}}_{t}^{(N)}}{\tilde{\mathfrak{I}}_{s}^{(N)}}\Big)\Big|\le2\Vert f^{\prime}\Vert_{L^{\infty}}\Vert\log\mathfrak{I}^{(N)}-\log\tilde{\mathfrak{I}}^{(N)}\Vert_{L^{\infty}(0,T)},
\]
\[
\Big|f\Big(\log\frac{\mathfrak{I}_{t}}{\mathfrak{I}_{s}}\Big)-f\Big(\log\frac{\tilde{\mathfrak{I}}_{t}}{\tilde{\mathfrak{I}}_{s}}\Big)\Big|\le2\Vert f^{\prime}\Vert_{L^{\infty}}\Vert\log\mathfrak{I}-\log\tilde{\mathfrak{I}}\Vert_{L^{\infty}(0,T)},
\]
and that $\lim_{N\to\infty}\tilde{\mathfrak{I}}^{(N)}=\tilde{\mathfrak{I}}$
in distribution, we deduce that
\[
\limsup_{N\to\infty}\Big|\mathbb{E}\Big[f\Big(\log\frac{\mathfrak{I}_{t}^{(N)}}{\mathfrak{I}_{s}^{(N)}}\Big)\Big]-\mathbb{E}\Big[f\Big(\log\frac{\mathfrak{I}_{t}}{\mathfrak{I}_{s}}\Big)\Big]\Big|\le4\Vert f^{\prime}\Vert_{L^{\infty}}\epsilon.
\]
This implies that 
\[
\lim_{N\to\infty}\log\frac{\mathfrak{I}_{t}^{(N)}}{\mathfrak{I}_{s}^{(N)}}=\log\frac{\mathfrak{I}_{t}}{\mathfrak{I}_{s}},
\]
in distribution, and the limiting distribution is independent of $\mathcal{F}_{s}$.
This implies that $\lim_{N\to\infty}\mathfrak{I}^{(N)}(r,\rho,\sigma)=\mathfrak{I}(r,\rho,\sigma)$
in law.

It suffices to prove the uniform continuity of $\log\mathfrak{I}^{(N)}(r,\rho,\sigma)$
in $(r,\rho,\sigma)$ on $A_{\delta,M}$. Suppose that $(r,\rho,\sigma),(\tilde{r},\tilde{\rho},\tilde{\sigma})\in A_{\delta,M}$.
Let $q=r-\rho$ and $\tilde{q}=\tilde{r}-\tilde{\rho}$. Denote $\tilde{u}_{n}^{(N)}$
be random variables defined in (\ref{eq:-14}) with $(r,\rho,\sigma)$
replaced by $(\tilde{r},\tilde{\rho},\tilde{\sigma})$. Clearly,
\begin{align*}
d[\log\mathfrak{J}^{(N)}(r,\rho,\sigma)-\log\mathfrak{J}^{(N)}(\tilde{r},\tilde{\rho},\tilde{\sigma})]_{t} & =[r(t)-\tilde{r}(t)]dt-\Big(\dfrac{\bar{\sigma}}{\sqrt{u_{n-1}^{(N)}}}q(t)-\dfrac{\bar{\sigma}}{\sqrt{\tilde{u}_{n-1}^{(N)}}}\tilde{q}(t)\Big)dt\\
 & \quad-\frac{1}{2}\Big(\dfrac{\bar{\sigma}^{2}\sigma(t)^{2}}{u_{n-1}^{(N)}}-\dfrac{\bar{\sigma}^{2}\tilde{\sigma}(t)^{2}}{\tilde{u}_{n-1}^{(N)}}\Big)dt+\Big(\dfrac{\bar{\sigma}\sigma(t)}{\sqrt{u_{n-1}^{(N)}}}-\dfrac{\bar{\sigma}\tilde{\sigma}(t)}{\sqrt{\tilde{u}_{n-1}^{(N)}}}\Big)dW_{t}
\end{align*}
for $t\in[t_{n-1},t_{n})$. Therefore, 
\begin{equation}
\begin{aligned} & \mathbb{E}\big(\sup_{t\in[0,T]}|\log\mathfrak{I}^{(N)}(r,\rho,\sigma)_{t}-\log\mathfrak{I}^{(N)}(\tilde{r},\tilde{\rho},\tilde{\sigma})_{t}|^{2}\big)\\
 & \le4\Big(\int_{0}^{T}|r(t)-\tilde{r}(t)|dt\Big)^{2}+4\mathbb{E}\Big[\Big(\sum_{n=1}^{N}\Big|\dfrac{\bar{\sigma}}{\sqrt{u_{n-1}^{(N)}}}\int_{t_{n-1}}^{t_{n}}q(t)dt-\dfrac{\bar{\sigma}}{\sqrt{\tilde{u}_{n-1}^{(N)}}}\int_{t_{n-1}}^{t_{n}}\tilde{q}(t)dt\Big|\Big)^{2}\Big]\\
 & +\mathbb{E}\Big[\Big(\sum_{n=1}^{N}\Big|\dfrac{\bar{\sigma}^{2}}{u_{n-1}^{(N)}}\int_{t_{n-1}}^{t_{n}}\sigma(t)^{2}dt-\dfrac{\bar{\sigma}^{2}}{\tilde{u}_{n-1}^{(N)}}\int_{t_{n-1}}^{t_{n}}\tilde{\sigma}(t)^{2}dt\Big|\Big)^{2}\Big]\\
 & \quad+4\mathbb{E}\Big[\Big(\sum_{n=1}^{N}\Big(\dfrac{\bar{\sigma}}{\sqrt{u_{n-1}^{(N)}}}\int_{t_{n-1}}^{t_{n}}\sigma(t)dW_{t}-\dfrac{\bar{\sigma}}{\sqrt{\tilde{u}_{n-1}^{(N)}}}\int_{t_{n-1}}^{t_{n}}\tilde{\sigma}(t)dW_{t}\Big)\Big)^{2}\Big]\\
 & =4\Vert r-\tilde{r}\Vert_{L^{1}(0,T)}^{2}+4E_{1}+E_{2}+4E_{3}.
\end{aligned}
\label{eq:-84}
\end{equation}

To estimate the expectations $E_{1}$, $E_{2}$, $E_{3}$, we first
derive an $L^{p}$ bound of $|h(u_{n-1}^{(N)})-h(\tilde{u}_{n-1}^{(N)})|$
for any $h(x)=x^{-\alpha}$ with $p\ge1$ and $\alpha>0$. By (\ref{eq:-14})
and Hölder's inequality,
\[
\begin{aligned}\mathbb{E}\big(\big|u_{n}^{(N)}-\tilde{u}_{n}^{(N)}\big|^{2p}\big) & \le\mathbb{E}\Big[\Big(\frac{1-\lambda}{\Delta t}\sum_{k=1}^{n}\lambda^{n-k}\Big|\Big(\int_{t_{k-1}}^{t_{k}}\sigma(t)dW_{t}\Big)^{2}-\Big(\int_{t_{k-1}}^{t_{k}}\tilde{\sigma}(t)dW_{t}\Big)^{2}\Big|\Big)^{2p}\Big]\\
 & \le\mathbb{E}\Big[\frac{1-\lambda}{(\Delta t)^{2p}}\sum_{k=1}^{n}\lambda^{n-k}\Big|\Big(\int_{t_{k-1}}^{t_{k}}\sigma(t)dW_{t}\Big)^{2}-\Big(\int_{t_{k-1}}^{t_{k}}\tilde{\sigma}(t)dW_{t}\Big)^{2}\Big|^{2p}\Big]\\
 & \le\frac{1-\lambda}{(\Delta t)^{2p}}\sum_{k=1}^{n}\lambda^{n-k}\mathbb{E}\Big(\Big|\int_{t_{k-1}}^{t_{k}}[\sigma(t)-\tilde{\sigma}(t)]dW_{t}\Big|^{2p}\Big|\int_{t_{k-1}}^{t_{k}}[\sigma(t)+\tilde{\sigma}(t)]dW_{t}\Big|^{2p}\Big)\\
 & \le\frac{1-\lambda}{(\Delta t)^{2p}}\sum_{k=1}^{n}\lambda^{n-k}\mathbb{E}\Big(\Big|\int_{t_{k-1}}^{t_{k}}[\sigma(t)-\tilde{\sigma}(t)]dW_{t}\Big|^{4p}\Big)^{1/2}\mathbb{E}\Big(\Big|\int_{t_{k-1}}^{t_{k}}[\sigma(t)+\tilde{\sigma}(t)]dW_{t}\Big|^{4p}\Big)^{1/2}.
\end{aligned}
\]
By the Burkholder--Davis--Gundy inequality,
\[
\begin{aligned}\mathbb{E}\big(\big|u_{n}^{(N)}-\tilde{u}_{n}^{(N)}\big|^{2p}\big) & \le c_{p}\frac{1-\lambda}{(\Delta t)^{2p}}\sum_{k=1}^{n}\lambda^{n-k}\mathbb{E}\Big(\Big|\int_{t_{k-1}}^{t_{k}}[\sigma(t)-\tilde{\sigma}(t)]^{2}dt\Big|^{2p}\Big)^{1/2}\mathbb{E}\Big(\Big|\int_{t_{k-1}}^{t_{k}}[\sigma(t)+\tilde{\sigma}(t)]^{2}dt\Big|^{2p}\Big)^{1/2}\\
 & \le c_{p,M}\frac{1-\lambda}{(\Delta t)^{p}}\sum_{k=1}^{n}\lambda^{n-k}\Big(\int_{t_{k-1}}^{t_{k}}[\sigma(t)-\tilde{\sigma}(t)]^{2}dt\Big)^{p}
\end{aligned}
\]
Therefore,
\[
\mathbb{E}\big(\big|u_{n}^{(N)}-\tilde{u}_{n}^{(N)}\big|^{2p}\big)^{1/p}\le c_{p,M}\frac{(1-\lambda)^{1/p}}{\Delta t}\sum_{k=1}^{n}\lambda^{(n-k)/p}\int_{t_{k-1}}^{t_{k}}[\sigma(t)-\tilde{\sigma}(t)]^{2}dt,
\]
and hence,
\begin{equation}
\begin{aligned}\sum_{n=1}^{N}\mathbb{E}\big(\big|u_{n}^{(N)}-\tilde{u}_{n}^{(N)}\big|^{2p}\big)^{1/p} & \le c_{p,M}\frac{(1-\lambda)^{1/p}}{\Delta t}\sum_{n=1}^{N}\sum_{k=1}^{n}\lambda^{(n-k)/p}\int_{t_{k-1}}^{t_{k}}[\sigma(t)-\tilde{\sigma}(t)]^{2}dt,\\
 & =c_{p,M}\frac{(1-\lambda)^{1/p}}{\Delta t}\sum_{k=1}^{N}\sum_{n=k}^{N}\lambda^{(n-k)/p}\int_{t_{k-1}}^{t_{k}}[\sigma(t)-\tilde{\sigma}(t)]^{2}dt\\
 & =c_{p,M}\frac{(1-\lambda)^{1/p}}{(1-\lambda^{1/p})\Delta t}\sum_{k=1}^{N}\int_{t_{k-1}}^{t_{k}}[\sigma(t)-\tilde{\sigma}(t)]^{2}dt\\
 & \le c_{p,\lambda,M}(\Delta t)^{-1}\Vert\sigma-\tilde{\sigma}\Vert_{L^{2}(0,T)}^{2}.
\end{aligned}
\label{eq:-26}
\end{equation}
By the monotonicity of $h^{\prime}(x)$, 
\[
\begin{aligned}\mathbb{E}\big(\big||h(u_{n-1}^{(N)})-h(\tilde{u}_{n-1}^{(N)})|\big|^{p}\big) & \le\mathbb{E}\big[\big(|h^{\prime}(u_{n-1}^{(N)})|^{p}+|h^{\prime}(\tilde{u}_{n-1}^{(N)})|^{p}\big)|u_{n}^{(N)}-\tilde{u}_{n}^{(N)}|^{p}\big]\\
 & \le c_{p}\big[\mathbb{E}\big(|h^{\prime}(u_{n-1}^{(N)})|^{2p}\big)^{1/2}+\mathbb{E}\big(|h^{\prime}(\tilde{u}_{n-1}^{(N)})|^{2p}\big)^{1/2}\big]\mathbb{E}\big(|u_{n}^{(N)}-\tilde{u}_{n}^{(N)}|^{2p}\big)^{1/2}.
\end{aligned}
\]
By applying (\ref{eq:-79}) to $h^{\prime}$, we deduce that 
\[
\sup_{N\ge}\sup_{n<N}\mathbb{E}\big(|h^{\prime}(u_{n-1}^{(N)})|^{2p}\big)<c_{\alpha,p,\delta,M}<\infty.
\]
Therefore, by (\ref{eq:-26}),
\begin{equation}
\begin{aligned}\sum_{n=1}^{N}\mathbb{E}\big(\big||h(u_{n-1}^{(N)})-h(\tilde{u}_{n-1}^{(N)})|\big|^{p}\big)^{2/p} & \le c_{\alpha,p,\delta,M}\sum_{n=1}^{N}\mathbb{E}\big(|u_{n}^{(N)}-\tilde{u}_{n}^{(N)}|^{2p}\big)^{1/p}\le c_{\alpha,p,\lambda,\delta,M}(\Delta t)^{-1}\Vert\sigma-\tilde{\sigma}\Vert_{L^{2}(0,T)}^{2}.\end{aligned}
\label{eq:-28}
\end{equation}

We now revert to the estimates of $E_{1}$, $E_{2}$, and $E_{3}$.
For $E_{1}$, by the Cauchy--Schwartz inequality,
\begin{equation}
\begin{aligned}E_{1} & \le\mathbb{E}\Big[\Big(\sum_{n=1}^{N}\Big|\dfrac{\bar{\sigma}}{\sqrt{u_{n-1}^{(N)}}}-\dfrac{\bar{\sigma}}{\sqrt{\tilde{u}_{n-1}^{(N)}}}\Big|\int_{t_{n-1}}^{t_{n}}|\tilde{q}(t)|dt\Big)^{2}\Big]+\mathbb{E}\Big[\Big(\sum_{n=1}^{N}\dfrac{\bar{\sigma}}{\sqrt{u_{n-1}^{(N)}}}\int_{t_{n-1}}^{t_{n}}|q(t)-\tilde{q}(t)|dt\Big)^{2}\Big]\\
 & \le\mathbb{E}\Big[\Big(\sum_{n=1}^{N}\Big(\dfrac{\bar{\sigma}}{\sqrt{u_{n-1}^{(N)}}}-\dfrac{\bar{\sigma}}{\sqrt{\tilde{u}_{n-1}^{(N)}}}\Big)^{2}\int_{t_{n-1}}^{t_{n}}|\tilde{q}(t)|dt\Big)\Big(\sum_{n=1}^{N}\int_{t_{n-1}}^{t_{n}}|\tilde{q}(t)|dt\Big)\Big]\\
 & \quad+\mathbb{E}\Big[\Big(\sum_{n=1}^{N}\dfrac{\bar{\sigma}^{2}}{u_{n-1}^{(N)}}\int_{t_{n-1}}^{t_{n}}|q(t)-\tilde{q}(t)|dt\Big)\Big(\sum_{n=1}^{N}\int_{t_{n-1}}^{t_{n}}|q(t)-\tilde{q}(t)|dt\Big)\Big]\\
 & \le c_{M,T}\Delta t\sum_{n=1}^{N}\mathbb{E}\Big[\Big(\dfrac{\bar{\sigma}}{\sqrt{u_{n-1}^{(N)}}}-\dfrac{\bar{\sigma}}{\sqrt{\tilde{u}_{n-1}^{(N)}}}\Big)^{2}\Big]+\Vert q-\tilde{q}\Vert_{L^{1}(0,T)}\sum_{n=1}^{N}\mathbb{E}\Big(\dfrac{\bar{\sigma}^{2}}{u_{n-1}^{(N)}}\Big)\int_{t_{n-1}}^{t_{n}}|q(t)-\tilde{q}(t)|dt.
\end{aligned}
\label{eq:-94}
\end{equation}
By applying (\ref{eq:-79}) to $h(x)=x^{-1}$ and $h(x)=x^{-2}$,
we have
\begin{equation}
\sup_{N\ge1}\sup_{n<N}\mathbb{E}\Big[\dfrac{\bar{\sigma}^{2}}{u_{n-1}^{(N)}}\Big]<c_{\delta,M}<\infty,\label{eq:-93}
\end{equation}
and
\begin{equation}
\sup_{N\ge1}\sup_{n<N}\mathbb{E}\Big[\Big(\dfrac{\bar{\sigma}^{2}}{u_{n-1}^{(N)}}\Big)^{2}\Big]<c_{\delta,M}<\infty.\label{eq:-109}
\end{equation}
Moreover, applying (\ref{eq:-28}) to $h(x)=x^{-1/2}$ and $p=2$
gives
\begin{equation}
\sum_{n=1}^{N}\mathbb{E}\Big[\Big(\dfrac{\bar{\sigma}}{\sqrt{u_{n-1}^{(N)}}}-\dfrac{\bar{\sigma}}{\sqrt{\tilde{u}_{n-1}^{(N)}}}\Big)^{2}\Big]\le c_{\lambda,\delta,M}(\Delta t)^{-1}\Vert\sigma-\tilde{\sigma}\Vert_{L^{2}(0,T)}^{2}.\label{eq:-82}
\end{equation}
By (\ref{eq:-94}), (\ref{eq:-93}), and (\ref{eq:-82}), we deduce
that
\begin{equation}
E_{1}\le c_{\lambda,\delta,M,T}\big(\Vert q-\tilde{q}\Vert_{L^{1}(0,T)}^{2}+\Vert\sigma-\tilde{\sigma}\Vert_{L^{2}(0,T)}^{2}\big).\label{eq:-95}
\end{equation}

For $E_{2}$, similar to (\ref{eq:-94}), we have
\begin{equation}
\begin{aligned}E_{2} & \le c_{M,T}\Delta t\sum_{n=1}^{N}\mathbb{E}\Big[\Big(\dfrac{\bar{\sigma}^{2}}{u_{n-1}^{(N)}}-\dfrac{\bar{\sigma}^{2}}{\tilde{u}_{n-1}^{(N)}}\Big)^{2}\Big]\\
 & \quad+c_{M,T}\Vert\sigma-\tilde{\sigma}\Vert_{L^{2}(0,T)}\sum_{n=1}^{N}\mathbb{E}\Big[\Big(\dfrac{\bar{\sigma}^{2}}{u_{n-1}^{(N)}}\Big)^{2}\Big]\Big(\int_{t_{n-1}}^{t_{n}}|\sigma(t)-\tilde{\sigma}(t)|dt\Big)
\end{aligned}
\label{eq:-88}
\end{equation}
By (\ref{eq:-109}),
\begin{equation}
\sum_{n=1}^{N}\mathbb{E}\Big[\Big(\dfrac{\bar{\sigma}^{2}}{u_{n-1}^{(N)}}\Big)^{2}\Big]\Big(\int_{t_{n-1}}^{t_{n}}|\sigma(t)-\tilde{\sigma}(t)|dt\Big)\le c_{\delta,M,T}\Vert\sigma-\tilde{\sigma}\Vert_{L^{2}(0,T)}.\label{eq:-81}
\end{equation}
Combining (\ref{eq:-88}), (\ref{eq:-81}), and (\ref{eq:-82}) gives
that
\begin{equation}
E_{2}\le c_{\lambda,\delta,M,T}\Vert\sigma-\tilde{\sigma}\Vert_{L^{2}(0,T)}^{2}.\label{eq:-91}
\end{equation}

For $E_{3}$, by the martingale property, we have
\begin{align}
E_{3} & \le\sum_{n=1}^{N}\mathbb{E}\Big[\Big(\dfrac{\bar{\sigma}}{\sqrt{u_{n-1}^{(N)}}}-\dfrac{\bar{\sigma}}{\sqrt{\tilde{u}_{n-1}^{(N)}}}\Big)^{2}\Big(\int_{t_{n-1}}^{t_{n}}\tilde{\sigma}(t)dW_{t}\Big)^{2}\Big]+\sum_{n=1}^{N}\mathbb{E}\Big[\dfrac{\bar{\sigma}^{2}}{u_{n-1}^{(N)}}\Big(\int_{t_{n-1}}^{t_{n}}[\sigma(t)-\tilde{\sigma}(t)]dW_{t}\Big)^{2}\Big]\label{eq:-89}\\
 & =\sum_{n=1}^{N}\mathbb{E}\Big[\Big(\dfrac{\bar{\sigma}}{\sqrt{u_{n-1}^{(N)}}}-\dfrac{\bar{\sigma}}{\sqrt{\tilde{u}_{n-1}^{(N)}}}\Big)^{2}\Big]\int_{t_{n-1}}^{t_{n}}\tilde{\sigma}(t)^{2}dt+\sum_{n=1}^{N}\mathbb{E}\Big[\dfrac{\bar{\sigma}^{2}}{u_{n-1}^{(N)}}\Big]\int_{t_{n-1}}^{t_{n}}[\sigma(t)-\tilde{\sigma}(t)]^{2}dt\nonumber \\
 & \le C_{M}\Delta t\sum_{n=1}^{N}\mathbb{E}\Big[\Big(\dfrac{\bar{\sigma}}{\sqrt{u_{n-1}^{(N)}}}-\dfrac{\bar{\sigma}}{\sqrt{\tilde{u}_{n-1}^{(N)}}}\Big)^{2}\Big]+\sum_{n=1}^{N}\mathbb{E}\Big[\dfrac{\bar{\sigma}^{2}}{u_{n-1}^{(N)}}\Big]\Big(\int_{t_{n-1}}^{t_{n}}[\sigma(t)-\tilde{\sigma}(t)]^{2}dt\Big).\nonumber 
\end{align}
By (\ref{eq:-81}) and (\ref{eq:-82}), we obtain that
\begin{equation}
E_{3}\le c_{\lambda,\delta,M,T}\Vert\sigma-\tilde{\sigma}\Vert_{L^{2}(0,T)}^{2}.\label{eq:-83}
\end{equation}

By (\ref{eq:-84}), (\ref{eq:-95}), (\ref{eq:-91}), and (\ref{eq:-83}),
we deduce that
\begin{align}
\Vert\log\mathfrak{I}^{(N)}(r,\rho,\sigma) & -\log\mathfrak{I}^{(N)}(\tilde{r},\tilde{\rho},\tilde{\sigma})\Vert_{L^{2}(\mathbb{P};L^{\infty}(0,T))}\nonumber \\
 & \le c_{\lambda,\delta,M,T}\big(\Vert r-\tilde{r}\Vert_{L^{1}(0,T)}+\Vert q-\tilde{q}\Vert_{L^{1}(0,T)}+\Vert\sigma-\tilde{\sigma}\Vert_{L^{2}(0,T)}\big)\label{eq:-47}\\
 & \le c_{\lambda,\delta,M,T}\big(\Vert r-\tilde{r}\Vert_{L^{1}(0,T)}+\Vert\rho-\tilde{\rho}\Vert_{L^{1}(0,T)}+\Vert\sigma-\tilde{\sigma}\Vert_{L^{2}(0,T)}\big),\nonumber 
\end{align}
where $c_{\lambda,\delta,M,T}>0$ is a constant independent of $N$.
This proves the uniform continuity of $\log\mathfrak{I}^{(N)}(r,\rho,\sigma)$
on $A_{\delta,M}$, which completes the proof.
\end{proof}

\section{\label{sec:}Asymptotic properties of the functions $U(\lambda)$
and $V(\lambda)$}

In this section, we study the asymptotic properties and derive upper
and lower bounds of the functions $U(\lambda)$ and $V(\lambda)$
defined in (\ref{eq:-16}) and (\ref{eq:-3}) respectively. The exponent
$-1/2$ in the terms $(1+t^{2}\lambda^{k})^{-1/2}$ and $(1+t\lambda^{k})^{-1/2}$
makes direct computation of $U(\lambda)$ and $V(\lambda)$ difficult.
We will start with studying the following two closely related integrals
$\int_{0}^{\infty}\prod_{k=0}^{\infty}\big(1+t^{2}\lambda^{k}\big)^{-1}dt$
and $\int_{0}^{\infty}\prod_{k=0}^{\infty}(1+t\lambda^{k})^{-1}dt$,
for which the integral of corresponding partial products can be computed
explicitly. 

Some results regarding the $q$-binomial coefficients and the $q$-gamma
function will be needed. 
\begin{defn}
\label{def:}Let $q\in\mathbb{R},q\not=1$. For any integers $n\ge k\ge0$,
the \emph{$q$-binomial coefficient} is defined by
\[
\left(\begin{array}{c}
n\\
k
\end{array}\right)_{q}=\frac{\prod_{j=1}^{n}(1-q^{j})}{\big[\prod_{j=1}^{k}(1-q^{j})\big]\cdot\big[\prod_{j=1}^{n-k}(1-q^{j})\big]}.
\]
\end{defn}

\begin{rem}
Since $\lim_{q\to1}(1-q^{j})/(1-q)=j$, it is easily seen that the
$q$-binomial coefficient $\left(\begin{array}{c}
n\\
k
\end{array}\right)_{q}$ converges to the ordinary binomial coefficient $\left(\begin{array}{c}
n\\
k
\end{array}\right)=\frac{n!}{k!(n-k)!}$ as $q\to1$.
\end{rem}

Theorem \ref{thm:-2} below is the $q$-analogue of the binomial theorem.
A proof of Theorem \ref{thm:-2} can be found in \cite[Corollary 10.2.2, page 490]{AAR99}.
\begin{thm}[The $q$-Binomial Theorem]
\label{thm:-2}For any $q,t\in\mathbb{R}$, and any integer $n\ge1$,
the following holds
\[
\prod_{k=0}^{n-1}(1+tq^{k})=\sum_{k=0}^{n}\left(\begin{array}{c}
n\\
k
\end{array}\right)_{q}t^{k}q^{k(k-1)/2}.
\]
\end{thm}

\begin{defn}
\label{def:-1}Let $q\in\mathbb{R},|q|<1$. The \emph{$q$-gamma function}
is defined by
\[
\Gamma_{q}(x)=(1-q)^{1-x}\prod_{n=0}^{\infty}\frac{1-q^{n+1}}{1-q^{n+x}}.
\]
Lemma \ref{lem:-3} below shows that the $q$-gamma function converges
to the ordinary gamma function. A proof of Lemma \ref{lem:-3} can
be found in \cite[Corollary 10.3.4, page 495]{AAR99}.
\end{defn}

\begin{lem}
\label{lem:-3}Let $\Gamma(x)=\int_{0}^{\infty}t^{x-1}e^{-t}dt$ be
the ordinary gamma function. Then 
\[
\lim_{q\to1-}\Gamma_{q}(x)=\Gamma(x).
\]
\end{lem}

We now turn to the computation of the integrals $\int_{0}^{\infty}\prod_{k=0}^{\infty}\big(1+t^{2}\lambda^{k}\big)^{-1}dt$
and $\int_{0}^{\infty}\prod_{k=0}^{\infty}(1+t\lambda^{k})^{-1}dt$.
\begin{lem}
\label{lem:-6}For any $n\ge1$,
\begin{equation}
\int_{0}^{\infty}\prod_{k=0}^{n}(1+t^{2}\lambda^{k})^{-1}dt=\frac{\pi}{2}\prod_{k=0}^{n-1}\frac{1-\lambda^{k+1/2}}{1-\lambda^{k+1}},\label{eq:-18}
\end{equation}
and therefore,
\begin{equation}
\int_{0}^{\infty}\prod_{k=0}^{\infty}(1+t^{2}\lambda^{k})^{-1}dt=\frac{\pi}{2}\frac{(1-\lambda)^{1/2}}{\Gamma_{\lambda}(1/2)}.\label{eq:-19}
\end{equation}
\end{lem}

\begin{proof}
Let $f(z)=\log(z)\prod_{k=0}^{n}(1+z^{2}\lambda^{k})^{-1}$, $z\in\mathbb{C}\backslash[0,\infty)$.
We select the analytic branch of $\log(z)$ such that $\text{Im}[\log(z)]\in(0,2\pi)$.
Let $0<\epsilon<R$ with $R$ sufficiently large. Consider the contour
which goes from $\epsilon$ to $R$ along the upper side of the $x$-axis,
then anti-clock wise along the circle $|z|=R$, then from $R$ to
$\epsilon$ along the lower side of the $x$-axis (see Figure \ref{fig:-6}),
and lastly from $\epsilon-\text{i}\epsilon$ to $\epsilon+\text{i}\epsilon$
vertically. By the residue theorem, integrating $f(z)$ along this
contour and setting $\epsilon\to0$ gives
\begin{equation}
\begin{aligned}\int_{0}^{R}\log(t) & \prod_{k=0}^{n}(1+t^{2}\lambda^{k})^{-1}dt-\int_{0}^{R}\big[\log(t)+\text{i}2\pi\big]\prod_{k=0}^{n}(1+t^{2}\lambda^{k})^{-1}dt\\
 & +\int_{|z|=R}\log(z)\prod_{k=0}^{n}(1+z^{2}\lambda^{k})^{-1}dz=\text{i}2\pi\sum_{k=0}^{n}\big(\text{Res}(f,\text{i}\lambda^{-k/2})+\text{Res}(f,-\text{i}\lambda^{-k/2})\big),
\end{aligned}
\label{eq:-17}
\end{equation}
where the integration along the vertical line from $\epsilon-\text{i}\epsilon$
to $\epsilon+\text{i}\epsilon$ vanishes due to the boundedness of
$f(z)$ on this line. Note that
\[
\Big|\int_{|z|=R}\log(z)\prod_{k=0}^{n}(1+z^{2}\lambda^{k})^{-1}dz\Big|\le\frac{2\pi R\log(R)}{\lambda^{n+1}R^{2n+2}-1}\to0,
\]
as $R\to\infty$. Setting $R\to\infty$ in (\ref{eq:-17}) gives
\[
\int_{0}^{\infty}\prod_{k=0}^{n}(1+t^{2}\lambda^{k})^{-1}dt=-\sum_{k=0}^{n}\big(\text{Res}(f,\text{i}\lambda^{-k/2})+\text{Res}(f,-\text{i}\lambda^{-k/2})\big).
\]
\begin{figure}[h]
\centering{}\includegraphics[scale=0.75]{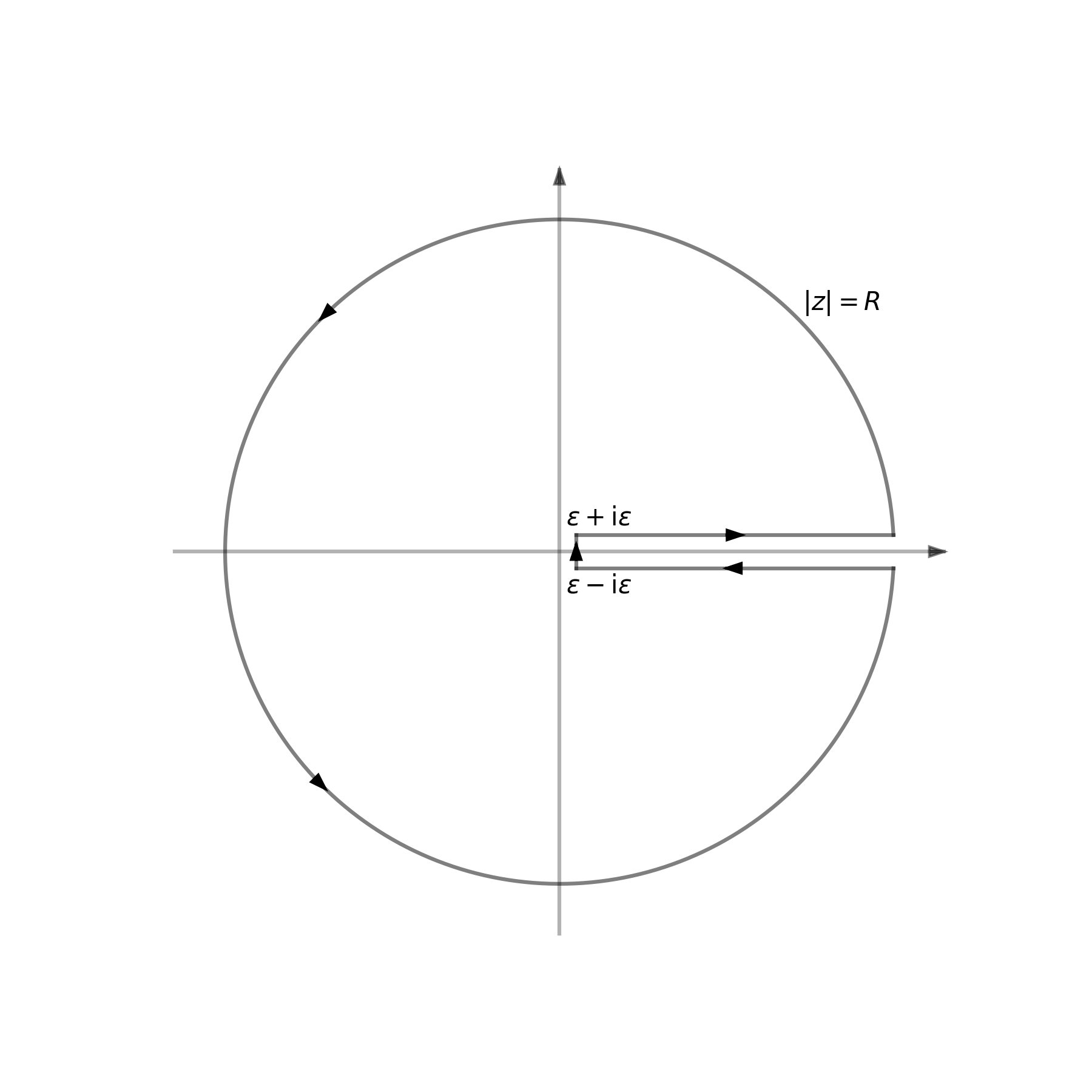}\caption{Integration contour}
\label{fig:-6}
\end{figure}
It is easy to see that
\begin{align*}
\text{Res}(f,\text{i}\lambda^{-k/2})+\text{Res}(f,-\text{i}\lambda^{-k/2}) & =-\frac{\pi}{2}\frac{\lambda^{-k/2}}{\prod_{0\le j\le n,j\not=k}^{n}(1-\lambda^{j-k})}=\frac{\pi}{2\prod_{j=1}^{n}(1-\lambda^{j})}\left(\begin{array}{c}
n\\
k
\end{array}\right)_{\lambda}(-1)^{k-1}\lambda^{k^{2}/2}.
\end{align*}
Therefore, by Theorem \ref{thm:-2},
\[
\begin{aligned}\int_{0}^{\infty}\prod_{k=0}^{n}(1+t^{2}\lambda^{k})^{-1}dt & =\frac{\pi}{2\prod_{j=1}^{n}(1-\lambda^{j})}\sum_{k=0}^{n}\left(\begin{array}{c}
n\\
k
\end{array}\right)_{\lambda}(-1)^{k}\lambda^{k^{2}/2}\\
 & =\frac{\pi}{2\prod_{j=1}^{n}(1-\lambda^{j})}\sum_{k=0}^{n}\left(\begin{array}{c}
n\\
k
\end{array}\right)_{\lambda}(-\lambda^{1/2})^{k}\lambda^{k(k-1)/2}=\frac{\pi}{2}\frac{\prod_{k=0}^{n-1}(1-\lambda^{k+1/2})}{\prod_{k=0}^{n-1}(1-\lambda^{k+1})}.
\end{aligned}
\]
This completes the proof of (\ref{eq:-18}). The equality (\ref{eq:-19})
follows from (\ref{eq:-18}) and Definition \ref{def:-1}.
\end{proof}
\begin{lem}
\label{lem:-2}For any $n\ge1$,
\begin{equation}
\int_{0}^{\infty}\prod_{k=0}^{n}(1+t\lambda^{k})^{-1}dt=\frac{\log(\lambda^{-1})}{1-\lambda^{n}},\label{eq:-8}
\end{equation}
and therefore,
\begin{equation}
\int_{0}^{\infty}\prod_{k=0}^{\infty}(1+t\lambda^{k})^{-1}dt=\log(\lambda^{-1}).\label{eq:-12}
\end{equation}
\end{lem}

\begin{proof}
Let $f(z)=\log(z)\prod_{k=0}^{n}(1+z\lambda^{k})^{-1}$, $z\in\mathbb{C}\backslash[0,\infty)$.
Again we select the analytic branch of $\log(z)$ such that $\text{Im}[\log(z)]\in(0,2\pi)$.
Consider the same contour as in the proof of Lemma \ref{lem:-6}.
By the residue theorem, integrating $f(z)$ along this contour gives
\begin{equation}
\begin{aligned}\int_{0}^{R}\log(t)\prod_{k=0}^{n}(1+t\lambda^{k})^{-1}dt & -\int_{0}^{R}\big[\log(t)+\text{i}2\pi\big]\prod_{k=0}^{n}(1+t\lambda^{k})^{-1}dt\\
 & +\int_{|z|=R}\log(z)\prod_{k=0}^{n}(1+z\lambda^{k})^{-1}dz=\text{i}2\pi\sum_{k=0}^{n}\text{Res}(f,-\lambda^{-k}).
\end{aligned}
\label{eq:-6}
\end{equation}
It is easily seen that $\Big|\int_{|z|=R}\log(z)\prod_{k=0}^{n}(1+z\lambda^{k})^{-1}dz\Big|\le\frac{2\pi R\log(R)}{\lambda^{n}R^{n+1}-1}\to0$
as $R\to\infty$. Therefore, setting $R\to\infty$ in (\ref{eq:-6})
gives that
\[
\int_{0}^{\infty}\prod_{k=0}^{n}(1+t\lambda^{k})^{-1}dt=-\sum_{k=0}^{n}\text{Res}(f,-\lambda^{-k}).
\]
Clearly,
\[
\begin{aligned}\text{Res}(f,-\lambda^{-k}) & =\frac{\lambda^{-k}\log(-\lambda^{-k})}{\prod_{0\le j\le n,j\not=k}(1-\lambda^{j-k})}=\frac{\lambda^{-k}[k\log(\lambda^{-1})+\text{i}\pi]}{\big[\prod_{j=1}^{k}(1-\lambda^{-j})\big]\big[\prod_{j=1}^{n-k}(1-\lambda^{j})\big]}\\
 & =\frac{1}{\prod_{j=1}^{n}(1-\lambda^{j})}\left(\begin{array}{c}
n\\
k
\end{array}\right)_{\lambda}(-1)^{k}\lambda^{k(k-1)/2}[k\log(\lambda^{-1})+\text{i}\pi].
\end{aligned}
\]
Therefore,
\begin{equation}
\int_{0}^{\infty}\prod_{k=0}^{n}(1+t\lambda^{k})^{-1}dt=\frac{1}{\prod_{j=1}^{n}(1-\lambda^{j})}\sum_{k=0}^{n}\left(\begin{array}{c}
n\\
k
\end{array}\right)_{\lambda}(-1)^{k-1}\lambda^{k(k-1)/2}[k\log(\lambda^{-1})+\text{i}\pi].\label{eq:-9}
\end{equation}
By Theorem \ref{thm:-2}, 
\begin{equation}
\sum_{k=0}^{n}\left(\begin{array}{c}
n\\
k
\end{array}\right)_{\lambda}s^{k}\lambda^{k(k-1)/2}=\prod_{k=0}^{n-1}(1+s\lambda^{k}).\label{eq:-7}
\end{equation}
Hence
\begin{equation}
\sum_{k=0}^{n}\left(\begin{array}{c}
n\\
k
\end{array}\right)_{\lambda}(-1)^{k}\lambda^{k(k-1)/2}=\prod_{k=0}^{n-1}(1-\lambda^{k})=0.\label{eq:-10}
\end{equation}
Differentiating (\ref{eq:-7}) at $s=-1$ gives
\begin{equation}
\sum_{k=0}^{n}\left(\begin{array}{c}
n\\
k
\end{array}\right)_{\lambda}k(-1)^{k-1}\lambda^{k(k-1)/2}=\prod_{k=1}^{n-1}(1-\lambda^{k}).\label{eq:-11}
\end{equation}
Combining (\ref{eq:-9}), (\ref{eq:-10}), and (\ref{eq:-11}) proves
(\ref{eq:-8}).
\end{proof}
We can now derive the upper and lower bounds for $U(\lambda)$ and
$V(\lambda)$.
\begin{lem}
\label{lem:-4}(i) For any $\lambda\in(0,1)$,
\[
\frac{\lambda^{-1}\log(\lambda^{-1})}{\lambda^{-1}-1}\le V(\lambda)\le\frac{\lambda^{-2}\log(\lambda^{-1})}{\lambda^{-1}-1},
\]
and therefore
\[
\lim_{\lambda\to1-}V(\lambda)=1.
\]

(ii) If, in addition, $\lambda\in(0.7,1)$, then
\[
\frac{\lambda^{-1.45}\log(\lambda^{-1})}{\lambda^{-1}-1}\le V(\lambda)\le\frac{\lambda^{-1.5}\log(\lambda^{-1})}{\lambda^{-1}-1}.
\]
\end{lem}

\begin{proof}
(i) Note that
\[
\begin{aligned}\prod_{k=0}^{\infty}(1+t\lambda^{k})^{-1/2} & =\prod_{k=0}^{\infty}\big[(1+t\lambda^{2k})(1+t\lambda^{2k+1})\big]^{-1/2}\ge\prod_{k=0}^{\infty}(1+t\lambda^{2k})^{-1}.\end{aligned}
\]
Therefore, by Lemma \ref{lem:-2},
\[
V(\lambda)\ge\frac{1}{2(1-\lambda)}\int_{0}^{\infty}\prod_{k=0}^{\infty}(1+t\lambda^{2k})^{-1}dt=\frac{\log(\lambda^{-2})}{2(1-\lambda)}=\frac{\lambda^{-1}\log(\lambda^{-1})}{\lambda^{-1}-1}.
\]
Similarly, we have
\[
\prod_{k=0}^{\infty}(1+t\lambda^{k})^{-1/2}=\prod_{k=0}^{\infty}\big[(1+t\lambda^{2k})(1+t\lambda^{2k+1})\big]^{-1/2}\le\prod_{k=0}^{\infty}(1+t\lambda^{2k+1})^{-1},
\]
and therefore, 
\begin{align*}
V(\lambda) & \le\frac{1}{2(1-\lambda)}\int_{0}^{\infty}\prod_{k=0}^{\infty}(1+t\lambda^{2k+1})^{-1}dt\\
 & =\frac{1}{2\lambda(1-\lambda)}\int_{0}^{\infty}\prod_{k=0}^{\infty}(1+t\lambda^{2k})^{-1}dt=\frac{\lambda^{-2}\log(\lambda^{-1})}{\lambda^{-1}-1}.
\end{align*}
This proves (i).

(ii) The argument of (i) can be slightly modified to prove (ii). By
\begin{equation}
(1+t\lambda^{2k})(1+t\lambda^{2k+1})\ge(1+t\lambda^{2k+1/2})^{2},\label{eq:-145}
\end{equation}
we have
\[
\begin{aligned}V(\lambda) & \le\frac{1}{2(1-\lambda)}\int_{0}^{\infty}\prod_{k=0}^{\infty}(1+t\lambda^{2k+1/2})^{-1}dt=\frac{\lambda^{-3/2}\log(\lambda^{-1})}{\lambda^{-1}-1}.\end{aligned}
\]
On the other hand, we choose $\alpha>0$ such that
\begin{equation}
(1+t\lambda^{2k})(1+t\lambda^{2k+1})\le(1+t\lambda^{2k+\alpha})^{2},\label{eq:-13}
\end{equation}
Once such an $\alpha$ is chosen, we obtain that
\[
V(\lambda)\ge\frac{1}{2(1-\lambda)}\int_{0}^{\infty}\prod_{k=0}^{\infty}(1+t\lambda^{2k+\alpha})^{-1}dt=\frac{\lambda^{-1-\alpha}\log(\lambda^{-1})}{\lambda^{-1}-1}.
\]
The inequality (\ref{eq:-13}) holds if and only if $\lambda$ satisfies
$\lambda^{4k+1}\le\lambda^{4k+2\alpha}$ and $\lambda^{2k}+\lambda^{2k+1}\le2\lambda^{2k+\alpha}$.
These inequalities can be combined to the inequality
\[
\alpha\le\min\Big\{\frac{1}{2},\frac{\log(2)-\log(1+\lambda)}{\log(\lambda^{-1})}\Big\}.
\]
When $\lambda\in(0.7,1)$,
\[
\frac{\log(2)-\log(1+\lambda)}{\log(\lambda^{-1})}\ge\frac{\log(2)-\log(1+0.7)}{\log(1/0.7)}>0.45.
\]
Therefore, the inequality (\ref{eq:-13}) holds for $\alpha=0.45$
and $\lambda\in(0.7,1)$, which completes the proof of (ii).
\end{proof}
\begin{lem}
\label{lem:-1}(i) For any $\lambda\in(0,1)$,
\begin{equation}
\sqrt{\frac{\pi}{2}}\frac{(1+\lambda)^{1/2}}{\Gamma_{\lambda^{2}}(1/2)}\le U(\lambda)\le\lambda^{-1/2}\sqrt{\frac{\pi}{2}}\frac{(1+\lambda)^{1/2}}{\Gamma_{\lambda^{2}}(1/2)},\label{eq:-29}
\end{equation}
and therefore
\begin{equation}
\lim_{\lambda\to1-}U(\lambda)=1.\label{eq:-30}
\end{equation}
 (ii) If, in addition, $\lambda\in(0.7,1)$, then
\begin{equation}
\lambda^{-0.1}\sqrt{\frac{\lambda^{-1}\log\lambda^{-1}}{\lambda^{-1}-1}}\le U(\lambda)\le\frac{\lambda^{-0.125}}{1-2e^{-2\pi^{2}/\log(\lambda^{-1})}}\sqrt{\frac{\lambda^{-1}\log\lambda^{-1}}{\lambda^{-1}-1}}.\label{eq:-67}
\end{equation}
\end{lem}

\begin{proof}
(i) The inequality (\ref{eq:-29}) can be obtained by an argument
similar to the proof of Lemma \ref{lem:-4}. The limit (\ref{eq:-30})
follows from (\ref{eq:-29}) and Lemma \ref{lem:-3}. 

(ii) For $\lambda\in(0.7,1)$, applying (\ref{eq:-145}) and (\ref{eq:-13})
to $t^{2}$ gives
\begin{equation}
(1+t^{2}\lambda^{2k+1/2})^{2}\le(1+t^{2}\lambda^{2k})(1+t^{2}\lambda^{2k+1})\le(1+t^{2}\lambda^{2k+\alpha})^{2},\label{eq:-147}
\end{equation}
where $\alpha=0.45$. Therefore, by
\[
\prod_{k=0}^{\infty}(1+t^{2}\lambda^{k})^{-1/2}=\prod_{k=0}^{\infty}[(1+t^{2}\lambda^{2k})(1+t^{2}\lambda^{2k+1})]^{-1/2},
\]
(\ref{eq:-147}) and (\ref{eq:-16}), we deduce that
\begin{equation}
\int_{0}^{\infty}\prod_{k=0}^{\infty}(1+t^{2}\lambda^{2k+\alpha})^{-1}dt\le\sqrt{\frac{\pi(1-\lambda)}{2}}U(\lambda)\le\int_{0}^{\infty}\prod_{k=0}^{\infty}(1+t^{2}\lambda^{2k+1/2})^{-1}dt.\label{eq:-148}
\end{equation}
By simple changes of variables, (\ref{eq:-148}) can be written as
\[
\lambda^{-\alpha/2}\int_{0}^{\infty}\prod_{k=0}^{\infty}(1+t^{2}\lambda^{2k})^{-1}dt\le\sqrt{\frac{\pi(1-\lambda)}{2}}U(\lambda)\le\lambda^{-1/2}\int_{0}^{\infty}\prod_{k=0}^{\infty}(1+t^{2}\lambda^{2k})^{-1}dt.
\]
It follows from the above and Lemma \ref{lem:-6} that
\begin{equation}
\lambda^{-\alpha/2}\sqrt{\frac{\pi}{2}}\frac{(1+\lambda)^{1/2}}{\Gamma_{\lambda^{2}}(1/2)}\le U(\lambda)\le\lambda^{-1/4}\sqrt{\frac{\pi}{2}}\frac{(1+\lambda)^{1/2}}{\Gamma_{\lambda^{2}}(1/2)},\label{eq:-73}
\end{equation}

We now derive upper and lower bounds of $\Gamma_{\lambda^{2}}(1/2)$.
By the following Gauss's formula (see \cite[Section 10.9, page 523]{AAR99}),
\[
\sum_{n=0}^{\infty}\lambda^{n(n+1)/2}=\prod_{n=1}^{\infty}\frac{1-\lambda^{2n}}{1-\lambda^{2n-1}}.
\]
Denote $h=\sqrt{\log\lambda^{-1}}$. Then
\begin{equation}
\begin{aligned}\Gamma_{\lambda^{2}}\Big(\frac{1}{2}\Big) & =(1-\lambda^{2})^{1/2}\prod_{n=1}^{\infty}\frac{1-\lambda^{2n}}{1-\lambda^{2n-1}}=\frac{(1-\lambda^{2})^{1/2}}{h}\sum_{n=0}^{\infty}e^{-n(n+1)h^{2}/2}h\\
 & =e^{h^{2}/8}\sqrt{\frac{1-\lambda^{2}}{\log\lambda^{-1}}}\sum_{n=0}^{\infty}e^{-(n+1/2)^{2}h^{2}/2}h=\frac{e^{h^{2}/8}}{2}\sqrt{\frac{1-\lambda^{2}}{\log\lambda^{-1}}}\sum_{n=-\infty}^{\infty}e^{-(n+1/2)^{2}h^{2}/2}h.
\end{aligned}
\label{eq:-69}
\end{equation}
Note that $\sum_{n=-\infty}^{\infty}e^{-(n+1/2)^{2}h^{2}/2}h$ is
a Riemann sum converging to $\sqrt{\pi/2}$. To derive bounds of this
Riemann sum for small $h>0$, let $g(x)=e^{-h^{2}x^{2}/2}$. The Fourier
transform of $g$ is given by
\[
\hat{g}(\xi)=\int_{-\infty}^{\infty}g(x)e^{-\text{i}2\pi\xi x}dx=\sqrt{2\pi}h^{-1}e^{-2\pi^{2}\xi^{2}/h^{2}}.
\]
By the Poisson summation formula (see \cite[page 252]{SW71}), we
have
\begin{equation}
\sum_{n=-\infty}^{\infty}g(n+x)=\sum_{n=-\infty}^{\infty}\hat{g}(n)e^{\text{i}2\pi nx},\quad x\in\mathbb{R}.\label{eq:-76}
\end{equation}
Setting $x=1/2$ in (\ref{eq:-76}) gives that
\begin{align*}
\sum_{n=-\infty}^{\infty}e^{-(n+1/2)^{2}h^{2}/2}h=\sqrt{2\pi}\sum_{n=-\infty}^{\infty}e^{-2\pi^{2}n^{2}/h^{2}}e^{\text{i}\pi n}=\sqrt{2\pi}\Big(1+2\sum_{n=1}^{\infty}(-1)^{n}e^{-2\pi^{2}n^{2}/h^{2}}\Big),
\end{align*}
which implies that
\begin{equation}
\sqrt{2\pi}(1-2e^{-2\pi^{2}/h^{2}})\le\sum_{n=-\infty}^{\infty}e^{-(n+1/2)^{2}h^{2}/2}h\le\sqrt{2\pi}.\label{eq:-71}
\end{equation}
Combining (\ref{eq:-69}) and (\ref{eq:-71}) and $h=\sqrt{\log\lambda^{-1}}$
gives that
\begin{equation}
\lambda^{-1/8}\sqrt{\frac{\pi(1-\lambda^{2})}{2\log\lambda^{-1}}}(1-2e^{-2\pi^{2}/\log(\lambda^{-1})})\le\Gamma_{\lambda^{2}}\Big(\frac{1}{2}\Big)\le\lambda^{-1/8}\sqrt{\frac{\pi(1-\lambda^{2})}{2\log\lambda^{-1}}}.\label{eq:-72}
\end{equation}
The estimate (\ref{eq:-67}) follows from (\ref{eq:-73}) and (\ref{eq:-72})
readily.
\end{proof}

\section{\label{sec:-6}Limits of $\text{Var}[\log(I_{T})]$}

In this section, we study the behavior of $\text{Var}[\log(I_{T})]$
when $(N,\lambda)$ approaches $(\infty,1)$ in different manners
for the case when $\sigma(t)=\sigma$ is a constant. Proposition \ref{prop:-3}
below shows that the simultaneous limit $\lim_{N\to\infty,\lambda\to1-}\text{Var}[\log(I_{T})]$
exists if and only if the initial variance $v_{0}$ is equal to $\sigma^{2}$.
Lemma \ref{lem:-8} shows that, when $v_{0}\not=\sigma^{2}$, there
exist paths $\lambda=\lambda_{N}$ along which $(N,\lambda)$ approaches
$(\infty,1)$ and $\text{Var}[\log(I_{T})]$ converges to a limit
different from the desired $\bar{\sigma}^{2}T$. This implies the
fact that when $\lambda\to1$, $\Delta t$ has to converge to $0$
faster than $1-\lambda$, in order to achieve 
\[
\bigg|\sqrt{\frac{1}{T}\text{Var}[\log(I_{T})]}-\bar{\sigma}\bigg|<\epsilon
\]
for a preset volatility targeting tolerance $\epsilon$>0. Moreover,
Lemma \ref{lem:-9} gives a sufficient condition on the path $\lambda=\lambda_{N}$
such that $\text{Var}[\log(I_{T})]$ indeed converges to $\bar{\sigma}^{2}T$.

For the rest of this section, we will denote $I_{T}$ by $I_{T}^{(N,\lambda)}$
to indicate its dependence on $N$ and $\lambda$.
\begin{prop}
\label{prop:-3}(i) When $v_{0}\not=\sigma^{2}$, the order of limits
are not interchangeable for $\mathrm{Var}[\log(I_{T}^{(N,\lambda)})]$,
i.e.
\begin{equation}
\lim_{N\to\infty}\lim_{\lambda\to1-}\mathrm{Var}[\log(I_{T}^{(N,\lambda)})]\not=\lim_{\lambda\to1-}\lim_{N\to\infty}\mathrm{Var}[\log(I_{T}^{(N,\lambda)})].\label{eq:-107}
\end{equation}

(ii) When $v_{0}=\sigma^{2}$, the simultaneous limit $\lim_{N\to\infty,\lambda\to1-}\mathrm{Var}[\log(I_{T}^{(N,\lambda)})]$
exists and
\begin{equation}
\lim_{N\to\infty,\lambda\to1-}\frac{1}{\bar{\sigma}^{2}T}\mathrm{Var}[\log(I_{T}^{(N,\lambda)})]=1.\label{eq:-113}
\end{equation}
Therefore, the simultaneous limit $\lim_{N\to\infty,\lambda\to1-}\mathrm{Var}[\log(I_{T}^{(N,\lambda)})]$
exists if and only if $v_{0}=\sigma^{2}$.
\end{prop}

\begin{proof}
(i) Denote $V_{N,\lambda}=\text{Var}[\log(I_{T}^{(N,\lambda)})]/(\bar{\sigma}^{2}T)$.
By (\ref{eq:-45}), it is easily seen that
\begin{equation}
V_{N,\lambda}=\frac{1}{N}\sum_{k=0}^{N-1}\int_{0}^{\infty}e^{-v_{0}\sigma^{-2}\lambda^{k}t}\prod_{j=0}^{k-1}(1+2t(1-\lambda)\lambda^{j})^{-1/2}dt.\label{eq:-108}
\end{equation}
By the dominated convergence theorem,
\[
\lim_{\lambda\to1-}V_{N,\lambda}=\frac{1}{N}\sum_{k=0}^{N-1}\int_{0}^{\infty}e^{-v_{0}\sigma^{-2}t}dt=\frac{\sigma^{2}}{v_{0}}.
\]
Therefore, 
\begin{equation}
\lim_{N\to\infty}\lim_{\lambda\to1-}V_{N,\lambda}=\frac{\sigma^{2}}{v_{0}}.\label{eq:-62}
\end{equation}
On the other hand, recall that $\lim_{N\to\infty}V_{N,\lambda}=V(\lambda)$.
Therefore, by (\ref{eq:-38}),
\begin{equation}
\lim_{\lambda\to1-}\lim_{N\to\infty}V_{N,\lambda}=\lim_{\lambda\to1-}V(\lambda)=1.\label{eq:-65}
\end{equation}
The non-interchangeability (\ref{eq:-107}) for $v_{0}\not=\sigma^{2}$
follows from (\ref{eq:-62}) and (\ref{eq:-65}) readily.

(ii) Suppose that $v_{0}=\sigma^{2}$. We need to show that $\lim_{N\to\infty,\lambda\to1-}V_{N,\lambda}=1$.
By (\ref{eq:-108}),
\begin{equation}
\begin{aligned}V_{N,\lambda} & \ge\frac{1}{N}\sum_{k=0}^{N-1}\int_{0}^{\infty}e^{-\lambda^{k}t}\prod_{j=0}^{k-1}e^{-(1-\lambda)\lambda^{j}t}dt=\frac{1}{N}\sum_{k=0}^{N-1}\int_{0}^{\infty}e^{-t}dt=1.\end{aligned}
\label{eq:-116}
\end{equation}
Hence, $\liminf_{N\to\infty,\lambda\to1-}V_{N,\lambda}\ge1$. Conversely,
for any $a>1$ and $\lambda>1-(2a)^{-1}$, we have $2a(1-\lambda)\lambda^{j}<1$
and therefore,
\[
(1+a^{-1}t)^{2a(1-\lambda)\lambda^{j}}<1+2t(1-\lambda)\lambda^{j},\quad t>0.
\]
Moreover,
\[
\begin{aligned}V_{N,\lambda} & \le\frac{1}{N}\sum_{k=0}^{N-1}\int_{0}^{\infty}e^{-\lambda^{k}t}\prod_{j=0}^{k-1}(1+a^{-1}t)^{-a(1-\lambda)\lambda^{j}}dt\\
 & =\frac{1}{N}\sum_{k=0}^{N-1}\int_{0}^{\infty}e^{-\lambda^{k}t}(1+a^{-1}t)^{-a(1-\lambda^{k})}dt\\
 & =\frac{1}{N}\sum_{k=0}^{N-1}\int_{0}^{\infty}\Big(\frac{1+a^{-1}t}{e^{a^{-1}t}}\Big)^{a\lambda^{k}}(1+a^{-1}t)^{-a}dt\\
 & \le\frac{1}{N}\sum_{k=0}^{N-1}\int_{0}^{\infty}(1+a^{-1}t)^{-a}dt=\frac{a}{a-1}.
\end{aligned}
\]
This implies that $\limsup_{N\to\infty,\lambda\to1-}V_{N,\lambda}\le\frac{a}{a-1}$
for any $a>1$. Setting $a\to\infty$ gives $\limsup_{N\to\infty,\lambda\to1-}V_{N,\lambda}\le1$.
This completes the proof of (\ref{eq:-113}).
\end{proof}
\begin{lem}
\label{lem:-8}Let $\{\lambda_{N}\}$ be a sequence in $(0,1)$ such
that 
\begin{equation}
\lim_{N\to\infty}N(1-\lambda_{N})=0.\label{eq:-117}
\end{equation}
Then
\begin{equation}
\lim_{N\to\infty}\frac{1}{\bar{\sigma}^{2}T}\mathrm{Var}[\log I_{T}^{(N,\lambda_{N})}]=\frac{\sigma^{2}}{v_{0}}.\label{eq:-149}
\end{equation}
Therefore, when $v_{0}\not=\sigma^{2}$, in order to have $\mathrm{Var}[\log I_{T}^{(N,\lambda_{N})}]\to\bar{\sigma}^{2}T$,
the rebalancing time step $\Delta t$ has to converge to $0$ faster
than $1-\lambda$ as $\lambda\to1$.
\end{lem}

\begin{proof}
By (\ref{eq:-108}),
\[
\begin{aligned}V_{N,\lambda} & \le\frac{1}{N}\sum_{k=0}^{N-1}\int_{0}^{\infty}e^{-v_{0}\sigma^{-2}\lambda^{k}t}dt=\frac{1}{N}\sum_{k=0}^{N-1}\frac{\sigma^{2}}{v_{0}\lambda^{k}}=\frac{\sigma^{2}(\lambda^{-N}-1)}{v_{0}N(\lambda^{-1}-1)}.\end{aligned}
\]
It follows from (\ref{eq:-117}) that $N\log(\lambda_{N}^{-1})\to0$
and 
\[
\lim_{N\to\infty}\frac{\lambda_{N}^{-N}-1}{N(\lambda_{N}^{-1}-1)}=\lim_{N\to\infty}\dfrac{e^{N\log(\lambda_{N}^{-1})}-1}{N\log(\lambda_{N}^{-1})}\cdot\frac{\log(\lambda_{N}^{-1})}{\lambda_{N}^{-1}-1}=1.
\]
Therefore, 
\[
\limsup_{N\to\infty}V_{N,\lambda_{N}}\le\frac{\sigma^{2}}{v_{0}}.
\]
It remains to show that 
\begin{equation}
\liminf_{N\to\infty}V_{N,\lambda_{N}}\ge\frac{\sigma^{2}}{v_{0}}.\label{eq:-119}
\end{equation}
Clearly,
\begin{equation}
\begin{aligned}V_{N,\lambda} & \ge\frac{1}{N}\sum_{k=0}^{N-1}\int_{0}^{\infty}e^{-v_{0}\sigma^{-2}\lambda^{k}t}\prod_{j=0}^{k-1}e^{-t(1-\lambda)\lambda^{j}}dt=\frac{1}{N}\sum_{k=0}^{N-1}\frac{1}{1+(v_{0}\sigma^{-2}-1)\lambda^{k}}.\end{aligned}
\label{eq:-118}
\end{equation}
When $v_{0}\ge\sigma^{2}$, (\ref{eq:-119}) is an immediate consequence
of (\ref{eq:-118}) and $1/[1+(v_{0}\sigma^{-2}-1)\lambda^{k}]\ge1/[1+(v_{0}\sigma^{-2}-1)]=\sigma^{2}/v_{0}$.
When $v_{0}<\sigma^{2}$, it follows from (\ref{eq:-118}) that
\begin{equation}
\begin{aligned}V_{N,\lambda} & \ge\frac{1}{1+(v_{0}\sigma^{-2}-1)\lambda^{N}}\end{aligned}
.\label{eq:-120}
\end{equation}
By (\ref{eq:-117}), $\lim_{N\to\infty}\lambda_{N}^{N}=\lim_{N\to\infty}e^{N\log(\lambda_{N})}=1$,
which together with (\ref{eq:-120}) implies (\ref{eq:-119}). This
completes the proof.
\end{proof}
\begin{rem}
\label{rem:-3}(i) An example of paths $\lambda=\lambda_{N}$ satisfying
(\ref{eq:-117}) is $\lambda_{N}=1-N^{-2}$.

(ii) We would like to point out that in practice, the volatility $\sigma$
cannot be exactly measured, and consequently, the condition $v_{0}=\sigma^{2}$
cannot be satisfied in general. Therefore, for the general cases,
$\Delta t$ is required to converge to $0$ faster than $1-\lambda$
in order to achieve $\mathrm{Var}[\log I_{T}^{(N,\lambda_{N})}]\to\bar{\sigma}^{2}T$.
\end{rem}

\begin{lem}
\label{lem:-9}Let $\{\lambda_{N}\}$ be a sequence in $(0,1)$ such
that 
\begin{equation}
\lim_{N\to\infty}N^{\gamma}(1-\lambda_{N})=\infty,\label{eq:-128}
\end{equation}
for some $0<\gamma\le1/2$. Then
\begin{equation}
\lim_{N\to\infty}\frac{1}{\bar{\sigma}^{2}T}\mathrm{Var}[\log(I_{T}^{(N,\lambda)})]=1.\label{eq:-146}
\end{equation}
\end{lem}

\begin{proof}
We first show that $\liminf_{N\to\infty}V_{N,\lambda_{N}}\ge1$. When
$v_{0}\le\sigma^{2}$, it follows immediately from (\ref{eq:-118})
that $V_{N,\lambda}\ge1$, which implies $\liminf_{N\to\infty}V_{N,\lambda_{N}}\ge1$.
Suppose that $v_{0}>\sigma^{2}$. By (\ref{eq:-118}),
\[
\begin{aligned}V_{N,\lambda} & \ge\frac{1}{N}\Big(\sum_{0\le k<N^{1-\gamma}}\frac{1}{1+(v_{0}\sigma^{-2}-1)\lambda^{k}}+\sum_{N^{1-\gamma}\le k<N}\frac{1}{1+(v_{0}\sigma^{-2}-1)\lambda^{k}}\Big)\\
 & \ge\frac{\lfloor N^{1-\gamma}\rfloor}{N}\frac{\sigma^{2}}{v_{0}}+\frac{\lfloor N(1-N^{-\gamma})\rfloor}{N}\frac{1}{1+(v_{0}\sigma^{-2}-1)\lambda^{N^{1-\gamma}}}.
\end{aligned}
\]
By (\ref{eq:-128}) and $\gamma\le1/2$, 
\begin{equation}
\lim_{N\to\infty}\lambda_{N}^{N^{1-\gamma}}=\lim_{N\to\infty}e^{-N^{1-\gamma}\log(\lambda_{N}^{-1})}=0.\label{eq:-131}
\end{equation}
Therefore,
\[
\liminf_{N\to\infty}V_{N,\lambda_{N}}\ge\lim_{N\to\infty}\frac{\lfloor N^{1-\gamma}\rfloor}{N}+\lim_{N\to\infty}\frac{\lfloor N(1-N^{-\gamma})\rfloor}{N}=1.
\]
It remains to prove that $\limsup_{N\to\infty}V_{N,\lambda_{N}}\le1$.
Let $m=\lfloor N/2\rfloor+1$ and 
\[
A_{\lambda}=\sum_{0\le k\le3}\int_{0}^{\infty}e^{-v_{0}\sigma^{-2}\lambda^{k}t}dt=\frac{\sigma^{2}}{v_{0}}\sum_{0\le k\le3}\lambda^{-k}.
\]
Then
\begin{equation}
\begin{aligned}V_{N,\lambda} & \le\frac{1}{2(m-1)}\Big(A_{\lambda}+\sum_{k=4}^{2m-1}\int_{0}^{\infty}\prod_{j=0}^{k-1}(1+2t(1-\lambda)\lambda^{j})^{-1/2}dt\Big)\\
 & \le\frac{1}{2(m-1)}\Big[A_{\lambda}+\sum_{k=2}^{m-1}\int_{0}^{\infty}\Big(\prod_{j=0}^{2k-1}+\prod_{j=0}^{2k}\Big)(1+2t(1-\lambda)\lambda^{j})^{-1/2}dt\Big]\\
 & \le\frac{1}{m-1}\Big(\frac{A_{\lambda}}{2}+\sum_{k=2}^{m-1}\int_{0}^{\infty}\prod_{j=0}^{2k-1}(1+2t(1-\lambda)\lambda^{j})^{-1/2}dt\Big)\\
 & \le\frac{1}{m-1}\Big(\frac{A_{\lambda}}{2}+\sum_{k=2}^{m-1}\int_{0}^{\infty}\prod_{j=0}^{k-1}(1+2t(1-\lambda)\lambda^{2j+1})^{-1}dt\Big)\\
 & =\frac{1}{m-1}\Big(\frac{A_{\lambda}}{2}+\frac{1}{2\lambda(1-\lambda)}\sum_{k=2}^{m-1}\int_{0}^{\infty}\prod_{j=0}^{k-1}(1+t\lambda^{2j})^{-1}dt\Big),
\end{aligned}
\label{eq:-129}
\end{equation}
where the last inequality follows from the identity
\[
\prod_{j=0}^{2k-1}(1+2t(1-\lambda)\lambda^{j})^{-1/2}=\prod_{j=0}^{k-1}\big[(1+2t(1-\lambda)\lambda^{2j})^{-1/2}(1+2t(1-\lambda)\lambda^{2j+1})^{-1/2}\big].
\]
By Lemma \ref{lem:-2},
\begin{equation}
\int_{0}^{\infty}\prod_{j=0}^{k-1}(1+t\lambda^{2j})^{-1}dt=\frac{2\log(\lambda^{-1})}{1-\lambda^{2(k-1)}}\le\frac{2\log(\lambda^{-1})}{1-\lambda^{k-1}}.\label{eq:-130}
\end{equation}
For any $0<\epsilon<1$, by (\ref{eq:-129}) and (\ref{eq:-130}),
\begin{equation}
\begin{aligned}V_{N,\lambda} & \le\frac{1}{m-1}\Big(\frac{A_{\lambda}}{2}+\frac{\log(\lambda^{-1})}{\lambda(1-\lambda)}\sum_{k=2}^{m-1}\frac{1}{1-\lambda^{k-1}}\Big).\\
 & =\frac{1}{m-1}\Big[\frac{A_{\lambda}}{2}+\frac{\log(\lambda^{-1})}{\lambda(1-\lambda)}\Big(\sum_{1\le k<m^{1-\gamma}}\frac{1}{1-\lambda^{k}}+\sum_{m^{1-\gamma}\le k<m-1}\frac{1}{1-\lambda^{k}}\Big)\Big]\\
 & \le\frac{1}{m-1}\Big[\frac{A_{\lambda}}{2}+\frac{\log(\lambda^{-1})}{\lambda(1-\lambda)}\Big(\frac{m^{1-\gamma}}{1-\lambda}+\frac{m(1-m^{-\gamma})}{1-\lambda^{m^{1-\gamma}}}\Big)\Big]\\
 & \le\frac{A_{\lambda}}{N-2}+\frac{N/2+1}{N/2-1}\cdot\frac{\log(\lambda^{-1})}{\lambda(1-\lambda)}\Big(\frac{1}{(N/2)^{\gamma}(1-\lambda)}+\frac{1-(N/2)^{-\gamma}}{1-\lambda^{(N/2)^{1-\gamma}}}\Big).
\end{aligned}
\label{eq:-132}
\end{equation}
Clearly, $\lim_{N\to\infty}A_{\lambda_{N}}/(N-2)=0$. By (\ref{eq:-128})
and (\ref{eq:-131}),
\[
\lim_{N\to\infty}\Big(\frac{1}{(N/2)^{\gamma}(1-\lambda_{N})}+\frac{1-(N/2)^{-\gamma}}{1-\lambda_{N}^{(N/2)^{1-\gamma}}}\Big)=1,
\]
which together with (\ref{eq:-132}) implies $\limsup_{N\to\infty}V_{N,\lambda_{N}}\le1$.
This completes the proof.
\end{proof}
\begin{rem}
\label{rem:-4}(i) An example of $\lambda=\lambda_{N}$ satisfying
(\ref{eq:-128}) is $\lambda_{N}=1-N^{-1/2}\log N$.

(ii) The conclusion of Lemma \ref{lem:-9} remains true when $\sigma(t)$
is time dependent. We briefly describe the strategy to see this. Note
that the specific value of $\sigma$ was not used in the proof of
Lemma \ref{lem:-9}. One may choose a sequence of piece-wise constant
$\{\sigma_{n}(t)\}$ such that $\Vert\sigma_{n}-\sigma\Vert_{L^{2}}\to0$.
The result of Lemma \ref{lem:-9} holds for each $\sigma_{n}(t)$
by a conditioning argument. Setting $n\to\infty$ in (\ref{eq:-146})
for $\sigma_{n}$ gives the result for time dependent $\sigma(t)$.
\end{rem}

\section{\label{sec:-2}Numerical tests}

In this section, we present some numerical test results which verify
the conclusions of Theorem \ref{thm:-3} and Proposition \ref{prop:-1}.
For simulation of the underlying risky asset paths, the Euler scheme
with discretisation time step equal to $\Delta t$ is applied to the
process $\log S_{t}$, that is, $\log S_{t+\Delta t}=\log S_{t}+(\rho-\sigma^{2}/2)\Delta t+\sigma\sqrt{\Delta t}Z$
with $Z$ being samples of the standard normal distribution. The C++
code for all numerical tests in this section are available in the
open source library ``\texttt{cltvt}''\footnote{Source code available from \href{https://github.com/liuxuan1111/cltvt.git}{https://github.com/liuxuan1111/cltvt.git}}.

Figure \ref{fig:-5} compares the function $U(\lambda)$ and the bounds
in (\ref{eq:-70}), where the blue curves refer to $U(\lambda)$ and
the two dashed red curves refer to the corresponding upper and lower
bounds. Figure \ref{fig:} presents a similar comparison for the function
$V(\lambda)$ and the bounds in (\ref{eq:-38}). These figures show
that the bounds in (\ref{eq:-70}) and (\ref{eq:-38}) are good approximations
to the functions $U(\lambda)$ and $V(\lambda)$.

\begin{figure}[H]
\centering{}\includegraphics[scale=0.5]{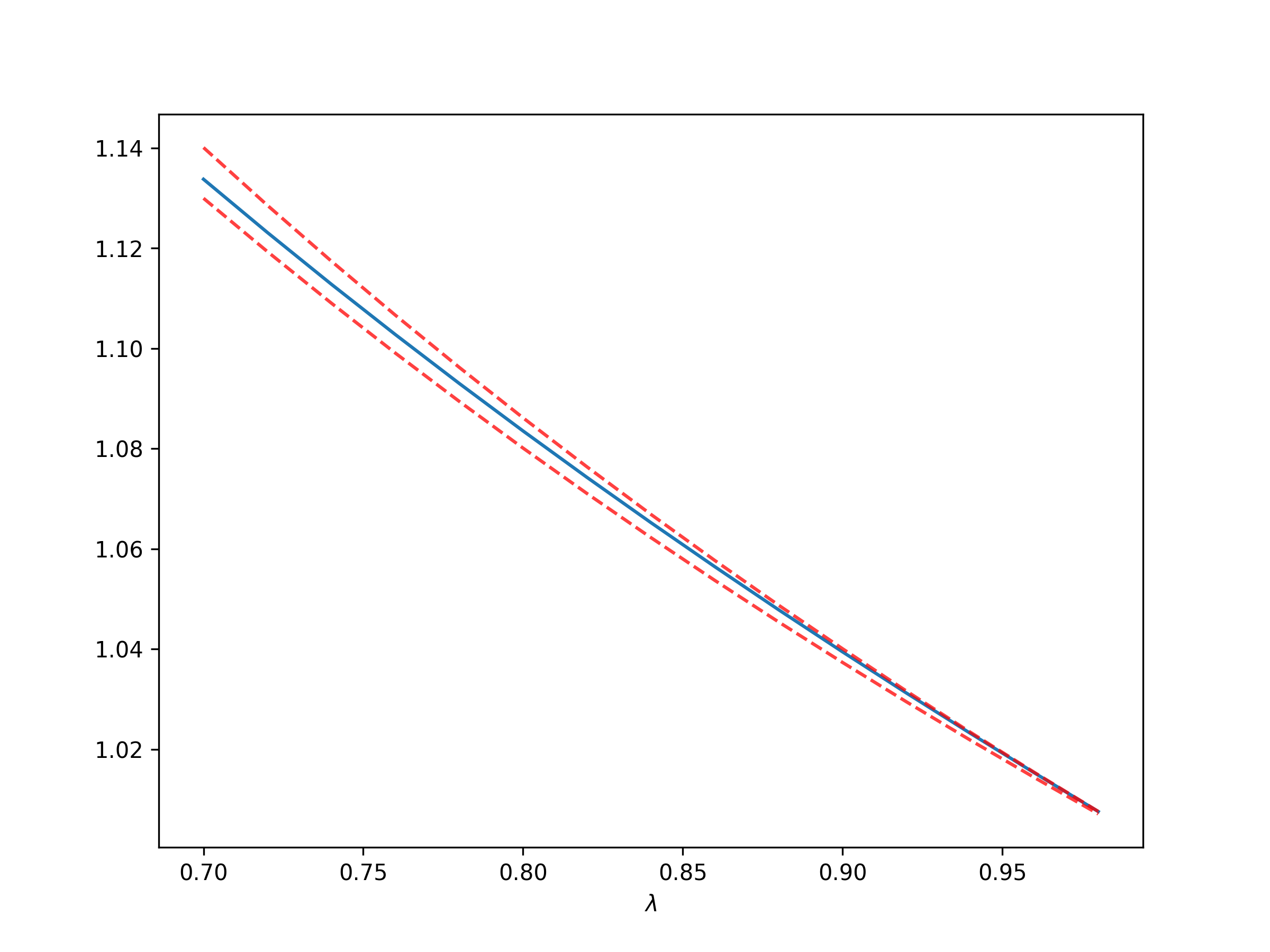}\caption{Approximation of $U(\lambda)$.}
\label{fig:-5}
\end{figure}

\begin{figure}[H]
\centering{}\includegraphics[scale=0.5]{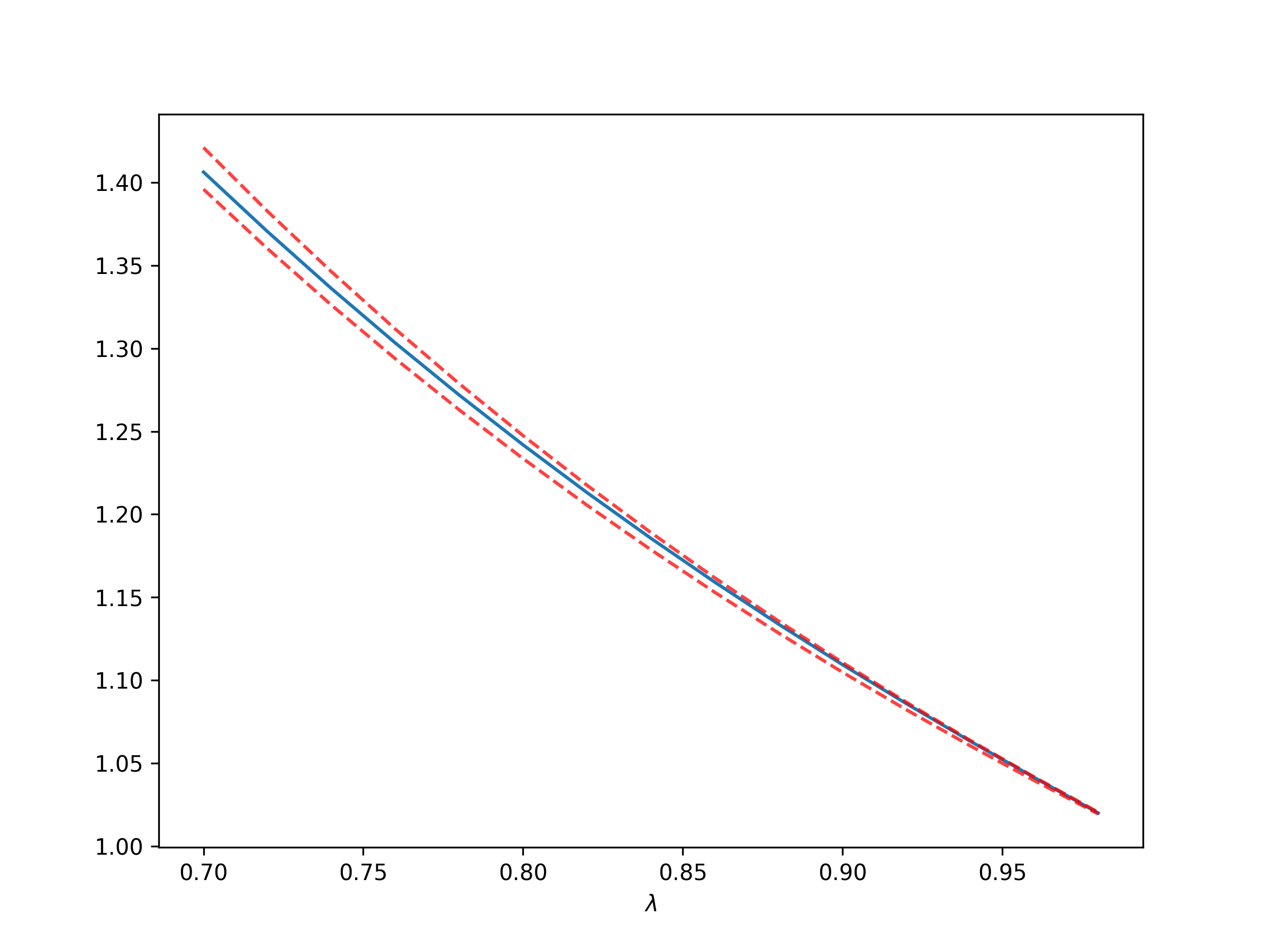}\caption{Approximation of $V(\lambda)$.}
\label{fig:}
\end{figure}

Figure \ref{fig:-1} shows the sample probability density function
of $\log(I_{T})$ and the density function of $\log(X_{T})$. The
parameters we use are $T=1.0$, $\lambda=0.9$, $\sigma=0.5$, $v_{0}=0.02$,
$r=0.05$, $\rho=0.03$, $S_{0}=I_{0}=1$, and $\bar{\sigma}=0.2$.
For Monte Carlo simulation of $I_{T}$, we use $2000$ time steps
for rebalancing and $100000$ number of paths. The sample density
function of $\log(I_{T})$ is computed from the sample histogram using
Gaussian kernel interpolation.

\begin{figure}[H]

\centering{}\includegraphics[scale=0.5]{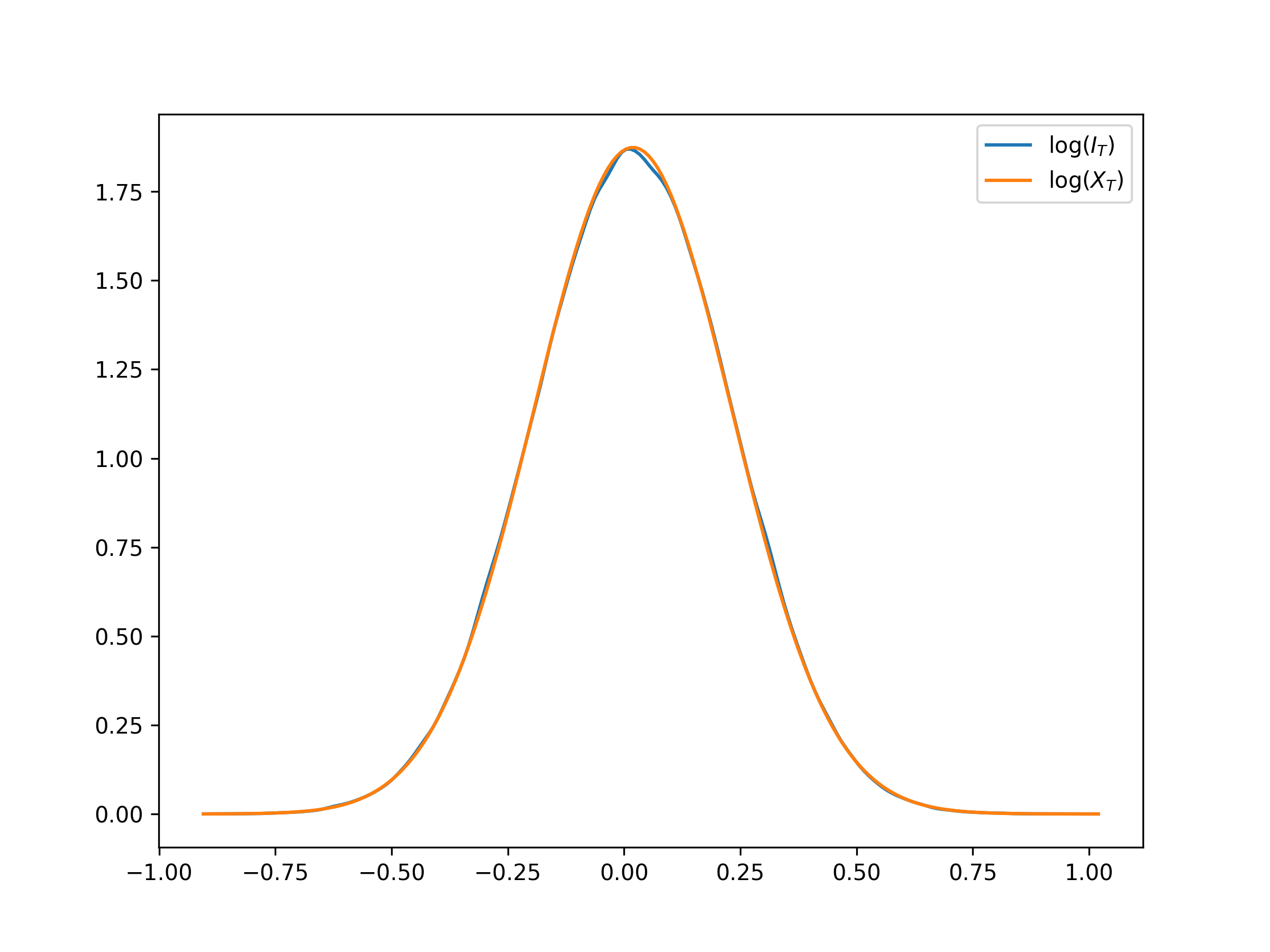}\caption{Approximation of sample density. Parameters used are $T=1.0$, $\lambda=0.9$,
$\sigma=0.5$, $v_{0}=0.02$, $r=0.05$, $\rho=0.03$, $S_{0}=I_{0}=1$,
and $\bar{\sigma}=0.2$}
\label{fig:-1}
\end{figure}

Figure \ref{fig:-2} shows the sample volatility of $\log(I_{I})$
against the limiting volatility $\bar{\sigma}\sqrt{V(\lambda)T}$
for a range of different $\lambda>0$. The solid curves correspond
to number of time steps $N=1000,2000,5000,10000,50000$ respectively\footnote{Assuming $6$ trading hours per day and $252$ trading days per year,
these settings amount to rebalancing time steps being $1.5$ hours,
$45$ minutes, $18$ minutes, $9$ minutes, and 2 minutes respectively.}. The parameters we use are $T=1.0$, $\sigma=0.5$, $v_{0}=0.02$,
$r=0.05$, $\rho=0.03$, $S_{0}=I_{0}=1$, and $\bar{\sigma}=0.2$.
For Monte Carlo simulation of $I_{T}$, we use $100000$ number of
paths. Note that Figure \ref{fig:-2} confirms the conclusion of Lemma
\ref{lem:-8} on the need of larger $N$ to achieve a prescribed volatility
targeting quality when $\lambda$ gets closer to $1$. 

\begin{figure}[H]
\centering{}\includegraphics[scale=0.5]{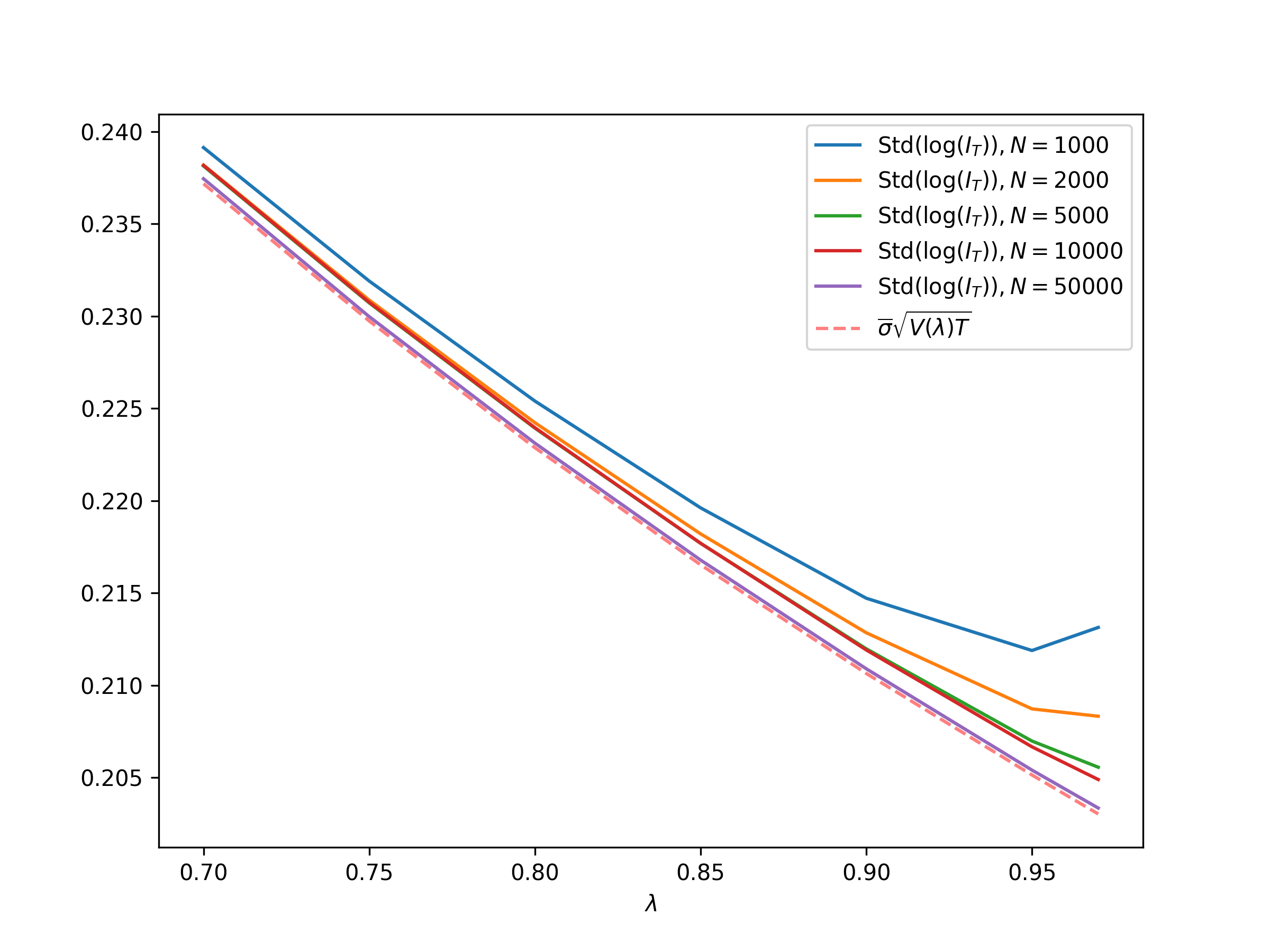}\caption{Convergence of volatility for different $\lambda$ ($x$-axis). Parameters
used are $T=1.0$, $\sigma=0.5$, $v_{0}=0.02$, $r=0.05$, $\rho=0.03$,
$S_{0}=I_{0}=1$, and $\bar{\sigma}=0.2$}
\label{fig:-2}
\end{figure}

Figure \ref{fig:-3} compares the value of a European call option
on a volatility target index with a range of $\lambda$. The curves
in Figure \ref{fig:-3} are computed using Monte Carlo method with
$100000$ sample paths to simulate volatility target index with $N=1000,2000,5000,10000,50000$
time steps, and using the Black--Scholes formula with the limiting
diffusion (\ref{eq:-39}). The parameters for the volatility target
index are $T=1.0$, $\sigma=0.5$, $v_{0}=0.02$, $r=0.05$, $\rho=0.03$,
$S_{0}=I_{0}=1$, and $\bar{\sigma}=0.2$. The strike of the call
option is $K=1.0$.

\begin{figure}[H]
\centering{}\includegraphics[scale=0.5]{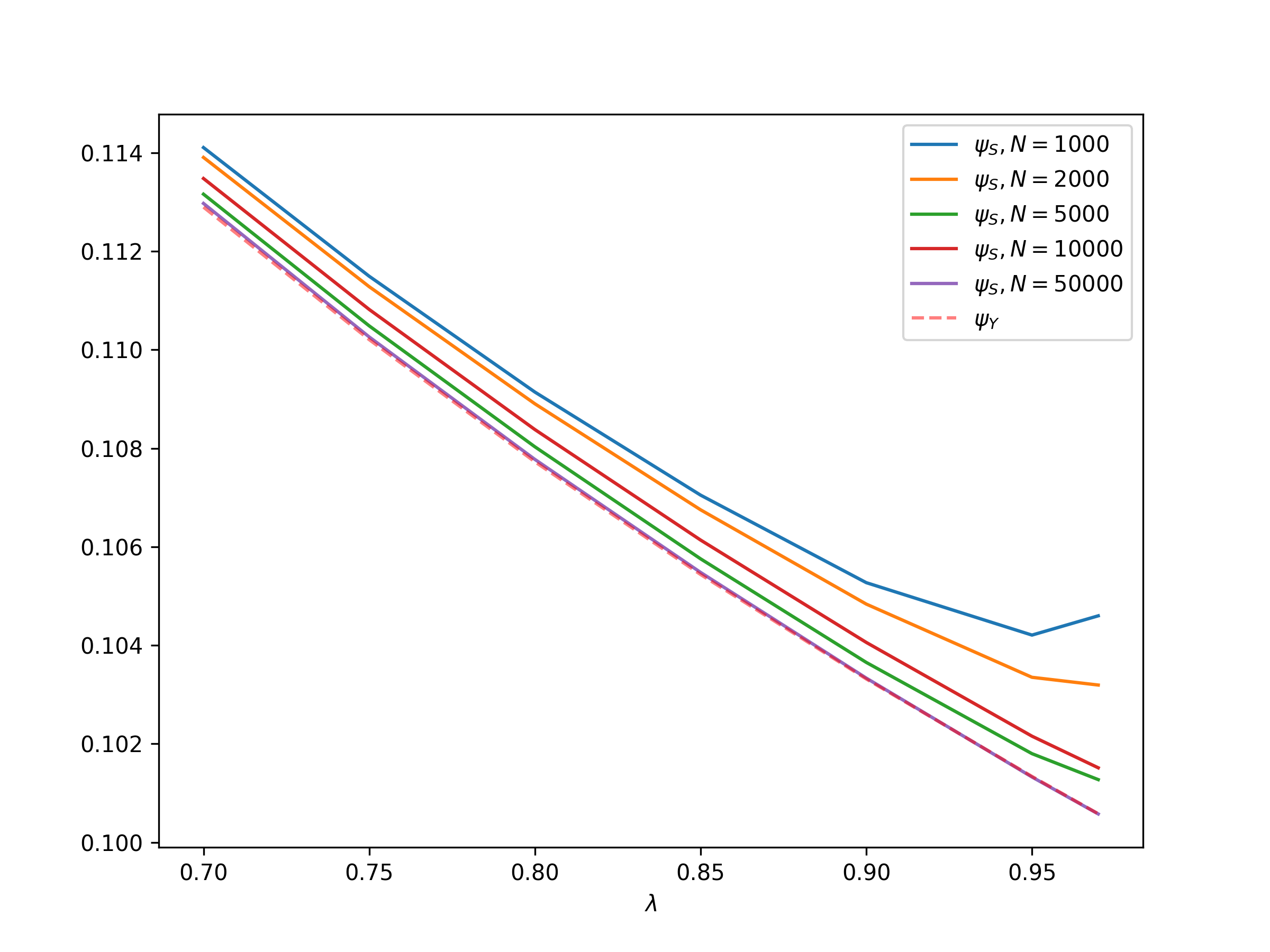}\caption{Convergence of European call option price for different $\lambda$
($x$-axis). Parameters used are $T=1.0$, $\sigma=0.5$, $v_{0}=0.02$,
$r=0.05$, $\rho=0.03$, $S_{0}=I_{0}=1$, and $\bar{\sigma}=0.2$}
\label{fig:-3}
\end{figure}

Figure \ref{fig:-4} compares the vega of a European call option on
a volatility target index with a range of $\lambda$. The curves in
Figure \ref{fig:-4} are computed using Monte Carlo method with $100000$
sample paths to simulate volatility target index with $N=1000,2000,5000,10000,50000$
time steps and $\Delta\sigma=0.001$, and using the Black--Scholes
formula and the vega conversion formula (\ref{eq:-40}). The parameters
for the volatility target index are $T=1.0$, $\sigma=0.5$, $v_{0}=0.02$,
$r=0.05$, $\rho=0.03$, $S_{0}=I_{0}=1$, and $\bar{\sigma}=0.2$.
The strike of the call option is $K=1.0$.

\begin{figure}[H]
\centering{}\includegraphics[scale=0.5]{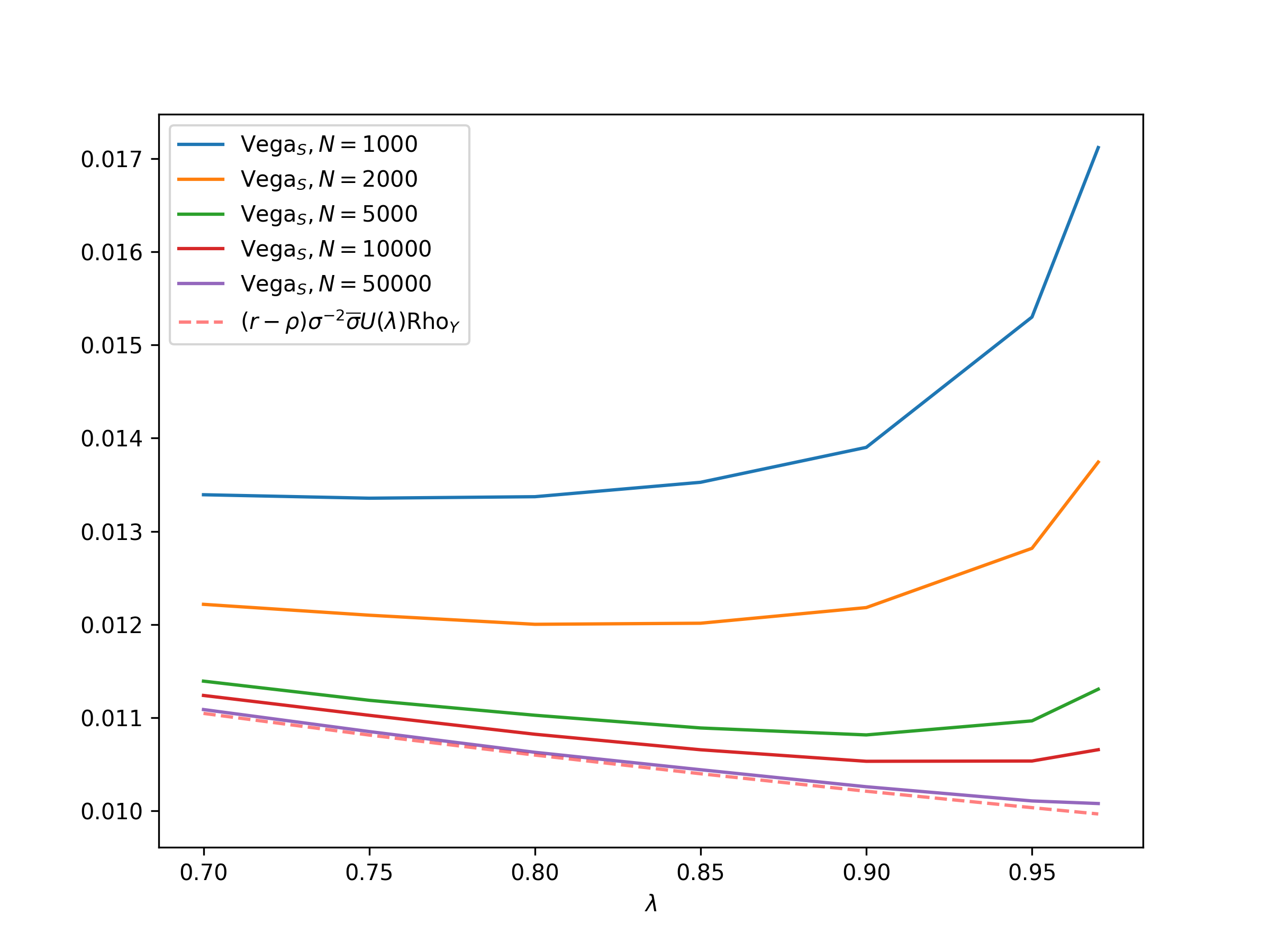}\caption{Convergence of European call option vega for different $\lambda$
($x$-axis). Parameters used are $T=1.0$, $\sigma=0.5$, $v_{0}=0.02$,
$r=0.05$, $\rho=0.03$, $S_{0}=I_{0}=1$, and $\bar{\sigma}=0.2$}
\label{fig:-4}
\end{figure}

\section{\label{sec:-4}Appendix}

In this appendix, we show that $\lim_{N\to\infty}[\log(I_{T}^{(N)})-\log(\tilde{I}_{T}^{(N)})]=0$
in distribution, and hence the definitions (\ref{eq:-86}) and (\ref{eq:-1})
are equivalent when $\Delta t\to0$. We start with some estimates
on the probability of events where remainders of Taylor expansions
can be controlled.
\begin{lem}
\label{lem:}Let $\{\xi_{n}^{(N)}\}_{1\le n\le N}$ be an array of
random variables such that
\begin{equation}
\lim_{N\to\infty}\sum_{n=1}^{N}\mathbb{E}(|\xi_{n}^{(N)}|^{p})=0,\label{eq:-87}
\end{equation}
for some $p>0$. Then, for any $\epsilon>0$,
\begin{equation}
\lim_{N\to\infty}\mathbb{P}\Big(\bigcap_{n=1}^{N}\big\{|\xi_{n}^{(N)}|\le\epsilon\big\}\Big)=1.\label{eq:-90}
\end{equation}
\end{lem}

\begin{proof}
By Chebyshev's inequality,
\[
\begin{aligned}\mathbb{P}\Big(\bigcup_{n=1}^{N}\big\{|\xi_{n}^{(N)}|>\epsilon\big\}\Big)\le\sum_{n=1}^{N}\mathbb{P}(|\xi_{n}^{(N)}|>\epsilon)\le\epsilon^{-p}\sum_{n=1}^{N}\mathbb{E}(|\xi_{n}^{(N)}|^{p}).\end{aligned}
\]
By (\ref{eq:-87}), we obtain that
\[
\lim_{N\to\infty}\mathbb{P}\Big(\bigcup_{n=1}^{N}\big\{|\xi_{n}^{(N)}|>\epsilon\big\}\Big)=0,
\]
which implies (\ref{eq:-90}).
\end{proof}
Let $r_{n}^{(N)}=\frac{1}{\Delta t}\int_{t_{n-1}}^{t_{n}}r(t)dt$,
$\rho_{n}^{(N)}=\frac{1}{\Delta t}\int_{t_{n-1}}^{t_{n}}\rho(t)dt$,
$(\sigma_{n}^{(N)})^{2}=\frac{1}{\Delta t}\int_{t_{n-1}}^{t_{n}}\sigma(t)^{2}dt$,
and 
\begin{align}
Z_{n}^{(N)} & =\frac{1}{\sigma_{n}\sqrt{\Delta t}}\int_{t_{n-1}}^{t_{n}}\sigma(t)dW_{t},\label{eq:-99}\\
R_{n}^{(N)} & =e^{(\rho_{n}-\frac{1}{2}\sigma_{n}^{2})\Delta t+\sigma_{n}\sqrt{\Delta t}Z_{n}}-1.\label{eq:-103}
\end{align}
Then, for each $N$, $\{Z_{n}^{(N)}\}_{n<N}$ are i.i.d. standard
normal distributions. 
\begin{lem}
\label{lem:-7}Let $Z_{n}^{(N)}$ be defined in (\ref{eq:-99}). For
any $\theta>0$ and any $\epsilon>0$,
\begin{align}
\lim_{N\to\infty}\mathbb{P}\Big(\bigcap_{n=1}^{N}\big\{(\Delta t)^{\theta}|Z_{n}^{(N)}|\le\epsilon\big\}\Big) & =1,\label{eq:-104}\\
\lim_{N\to\infty}\mathbb{P}\Big(\bigcap_{n=1}^{N}\big\{(\Delta t)^{\theta}w_{n-1}^{(N)}\le\epsilon\big\}\Big) & =1,\label{eq:-105}
\end{align}
and
\begin{equation}
\lim_{N\to\infty}\mathbb{P}\Big(\bigcap_{n=1}^{N}\big\{(\Delta t)^{\theta}w_{n-1}^{(N)}|Z_{n}^{(N)}|\le\epsilon\big\}\Big)=1.\label{eq:-106}
\end{equation}
\end{lem}

\begin{proof}
The limit (\ref{eq:-104}) follows from
\[
\sum_{n=1}^{N}\mathbb{E}\big(|(\Delta t)^{\theta}Z_{n}^{(N)}|^{1+2/\theta}\big)\le c_{T}N^{-1-\theta}
\]
and Lemma \ref{lem:}. By (\ref{eq:-33}), 
\begin{equation}
A_{p}=\sup_{N\ge1}\sup_{n<N}\mathbb{E}(|w_{n}^{(N)}|^{p})<\infty,\quad p>0.\label{eq:-100}
\end{equation}
Therefore,
\[
\sum_{n=1}^{N}\mathbb{E}\big(|(\Delta t)^{\theta}w_{n-1}^{(N)}|^{2/\theta}\big)\le c_{\theta,T}N^{-1},
\]
which, together with Lemma \ref{lem:}, implies (\ref{eq:-105}).
Similarly,
\[
\sum_{n=1}^{N}\mathbb{E}\big(|(\Delta t)^{\theta}w_{n-1}^{(N)}Z_{n}^{(N)}|^{1+2/\theta}\big)=(\Delta t)^{2+\theta}\sum_{n=1}^{N}\mathbb{E}\big(|w_{n-1}^{(N)}|^{1+2/\theta}\big)\mathbb{E}\big(|Z_{n}^{(N)}|^{1+2/\theta}\big)\le c_{\theta,T}N^{-1-\theta}.
\]
The limit (\ref{eq:-106}) follows from the above and Lemma \ref{lem:}.
\end{proof}
\begin{prop}
\label{prop:}Let $I_{T}^{(N)}$ and $\tilde{I}_{T}^{(N)}$ be defined
by (\ref{eq:-1}) and (\ref{eq:-86}) respectively. Then
\[
\lim_{N\to\infty}\big[\log(\tilde{I}_{T}^{(N)})-\log(I_{T}^{(N)})\big]=0,
\]
in distribution. 
\end{prop}

\begin{proof}
Clearly,
\begin{align*}
\log(\tilde{I}_{T}^{(N)}) & =\sum_{n=1}^{N}\log[1+(1-w_{n-1}^{(N)})r_{n}\Delta t+w_{n-1}^{(N)}R_{n}^{(N)}],\\
\log(I_{T}^{(N)}) & =\sum_{n=1}^{N}\Big(r_{n}^{(N)}+w_{n-1}^{(N)}(\rho_{n}^{(N)}-r_{n}^{(N)})-\frac{1}{2}(w_{n-1}^{(N)}\sigma_{n}^{(N)})^{2}\Big)\Delta t+\sum_{n=1}^{N}w_{n-1}^{(N)}\sigma_{n}^{(N)}\sqrt{\Delta t}Z_{n}^{(N)}.
\end{align*}
Let $M=1+\Vert r\Vert_{L^{\infty}}+\Vert\rho\Vert_{L^{\infty}}+\Vert\sigma^{2}\Vert_{L^{\infty}}<\infty$,
and let $\epsilon\in(0,1)$ be such that
\[
E_{N}=\bigcap_{n=1}^{N}\big(\big\{\Delta t|w_{n-1}^{(N)}|\le\epsilon\big\}\cap\big\{\sqrt{\Delta t}w_{n-1}^{(N)}|Z_{n}^{(N)}|\le\epsilon\big\}\cap\big\{\sqrt{\Delta t}|Z_{n}^{(N)}|\le\epsilon\big\}\big).
\]
Then, by Lemma \ref{lem:-7},
\begin{equation}
\lim_{N\to\infty}\mathbb{P}(E_{N})=1.\label{eq:-96}
\end{equation}
In what follows, to simplify notations, we will suppress the superscript
and denote $w_{n}^{(N)}$, $r_{n}^{(N)}$, $\rho_{n}^{(N)}$, $\sigma_{n}^{(N)}$,
$Z_{n}^{(N)}$ by $w_{n}$, $r_{n}$, $\rho_{n}$, $\sigma_{n}$,
$Z_{n}$ respectively whenever no confusion occurs. In view of Lemma
\ref{lem:-7}, the uniform $L^{p}$ bound (\ref{eq:-100}), and
\begin{align*}
\big|e^{x}-1-x-\frac{1}{2}x^{2}\big| & \le c|x|^{3},\quad|x|\le\frac{1}{2},\\
\Big|\log(1+x)-x+\frac{1}{2}x^{2}\Big| & \le c|x|^{3},\quad|x|\le\frac{1}{2},
\end{align*}
by choosing $\epsilon>0$ sufficiently small, we have
\[
\begin{aligned} & \log(\tilde{I}_{T}^{(N)})\chi_{E_{N}}\\
 & =\big[(1-w_{n-1})r_{n}\Delta t+w_{n-1}R_{n}-\frac{1}{2}[(1-w_{n-1})r_{n}\Delta t+w_{n-1}R_{n}]^{2}\big]\chi_{E_{N}}+\text{O}((\Delta t)^{3/2})\\
 & =\Big[(1-w_{n-1})r_{n}\Delta t+w_{n-1}\Big((\rho_{n}-\frac{1}{2}\sigma_{n}^{2})\Delta t+\sigma_{n}\sqrt{\Delta t}Z_{n}+\frac{1}{2}\sigma_{n}^{2}\Delta tZ_{n}^{2}\Big)\\
 & \quad-\frac{1}{2}w_{n-1}^{2}\sigma_{n}^{2}\Delta tZ_{n}^{2}\Big]\chi_{E_{N}}+\text{O}((\Delta t)^{3/2})\\
 & =\Big[\Big(r_{n}\Delta t+w_{n-1}(\rho_{n}-r_{n})-\frac{1}{2}w_{n-1}^{2}\sigma_{n}^{2}\Big)\Delta t+w_{n-1}\sigma_{n}\sqrt{\Delta t}Z_{n}\\
 & \quad+\frac{1}{2}w_{n-1}(1-w_{n-1})\sigma_{n}^{2}\Delta t(Z_{n}^{2}-1)\Big]\chi_{E_{N}}+\text{O}((\Delta t)^{3/2})\\
 & =\Big(\log(I_{T}^{(N)})+\frac{1}{2}w_{n-1}(1-w_{n-1})\sigma_{n}^{2}\Delta t(Z_{n}^{2}-1)\Big)\chi_{E_{N}}+\text{O}((\Delta t)^{3/2}),
\end{aligned}
\]
where we denote $\xi_{n}^{(N)}=\text{O}((\Delta t)^{\theta})$ if
\[
\lim_{N\to\infty}\sup_{n<N}(\Delta t)^{-\theta}\mathbb{E}\big(|\xi_{n}^{(N)}|\big)<\infty.
\]
Therefore,
\begin{equation}
\lim_{N\to\infty}\Big(\log(\tilde{I}_{T}^{(N)})-\log(I_{T}^{(N)})-\frac{1}{2}\sum_{n=1}^{N}w_{n-1}(1-w_{n-1})\sigma_{n}^{2}\Delta t(Z_{n}^{2}-1)\Big)\chi_{E_{N}}=0,\label{eq:-110}
\end{equation}
in distribution. Note that $\sum_{k=1}^{n}w_{k-1}(1-w_{k-1})\sigma_{k}^{2}\Delta t(Z_{k}^{2}-1)$
is a martingale, and hence,
\[
\begin{aligned}\mathbb{E}\Big[\Big(\sum_{n=1}^{N}w_{n-1}(1-w_{n-1})\sigma_{n}^{2}\Delta t(Z_{n}^{2}-1)\Big)^{2}\Big] & \le c_{M,T}N^{-2}\sum_{n=1}^{N}\mathbb{E}\Big[\Big(w_{n-1}^{2}(1-w_{n-1})^{2}(Z_{n}^{2}-1)\Big)^{2}\Big]\\
 & =c_{M,T}N^{-2}\sum_{n=1}^{N}\mathbb{E}\Big[\Big(w_{n-1}^{2}(1-w_{n-1})^{2}\Big)^{2}\Big]\\
 & \le c_{M,T}N^{-2}\sum_{n=1}^{N}\big[\mathbb{E}(w_{n-1}^{4})+\mathbb{E}(w_{n-1}^{8})\big],
\end{aligned}
\]
which, together with (\ref{eq:-100}), implies that
\begin{equation}
\lim_{N\to\infty}\sum_{n=1}^{N}w_{n-1}(1-w_{n-1})\sigma_{n}^{2}\Delta t(Z_{n}^{2}-1)=0,\label{eq:-111}
\end{equation}
in distribution. It follows from (\ref{eq:-110}) and (\ref{eq:-111})
that
\begin{equation}
\lim_{N\to\infty}\big[\log(\tilde{I}_{T}^{(N)})-\log(I_{T}^{(N)})\big]\chi_{E_{N}}=0,\label{eq:-97}
\end{equation}
in distribution. 

For any continuous function $f\in C(\mathbb{R})\cap L^{\infty}(\mathbb{R})$,
we have 
\[
\big|\mathbb{E}\big[f(\log(\tilde{I}_{T}^{(N)}))-f(\log(I_{T}^{(N)}))\big]\big|\le\big|\mathbb{E}\big[f\big(\log(\tilde{I}_{T}^{(N)})\chi_{E_{N}}\big)-f\big(\log(I_{T}^{(N)})\chi_{E_{N}}\big)\big]\big|+2\Vert f\Vert_{L^{\infty}}(1-\mathbb{P}(E_{N})).
\]
It follows from (\ref{eq:-96}), (\ref{eq:-97}), and the above that
\[
\lim_{N\to\infty}\mathbb{E}\big[f(\log(\tilde{I}_{T}^{(N)}))-f(\log(I_{T}^{(N)}))\big]=0.
\]
This completes the proof.
\end{proof}

\end{document}